\documentclass{article}
\usepackage[utf8]{inputenc}
\usepackage{graphicx,xcolor}
\usepackage{amsmath,amsfonts,amsthm}
\usepackage[utf8]{inputenc}
\usepackage{hyperref}
\usepackage{xfrac}
\usepackage{mathrsfs}
\newcommand\res{\mathop{\hbox{\vrule height 7pt width .5pt depth 0pt
			\vrule height .5pt width 6pt depth 0pt}}\nolimits}
\newcommand\LL{\res}
\newcommand\weakto{\rightharpoonup}
\newcommand\eps{\varepsilon}
\newcommand\Id{\mathrm{Id}}

\newcommand\rank{\mathrm{rank\,}}
\newcommand\sym{\mathrm{sym}}

\newcommand\loc{\mathrm{loc}}

\newcommand\Tr{\mathrm{Tr\,}}
\newcommand\dist{\mathrm{dist}}
\newcommand\qc{\mathrm{qc}}

\newcommand\qcinfty{{\mathrm{qc},\infty}}

\newcommand\conv{\mathrm{conv}}
\newcommand\scal{\mathrm{scal}}
\newcommand\simp{\mathrm{{simp}}}
\newcommand\SO{\mathrm{SO}}

\newcommand\R{\mathbb{R}}
\newcommand\N{\mathbb{N}}
\newcommand\Z{\mathbb{Z}}
\newcommand\calG{\mathcal{G}}
\newcommand\calU{\mathcal{U}}
\newcommand\calA{\mathcal{A}}
\newcommand\calL{\mathcal{L}}
\newcommand\calH{\mathcal{H}}
\newcommand\calF{\mathcal{F}}
\newcommand\calM{\mathcal{M}}

\newtheorem{theorem}{Theorem}[section]

\newtheorem{proposition}[theorem]{Proposition}
\newtheorem{lemma}[theorem]{Lemma}
\newtheorem{remark}[theorem]{Remark}
\newtheorem{corollary}[theorem]{Corollary}

\numberwithin{equation}{section}
\usepackage{fancyhdr}
\newcommand\Functeps{\mathcal F_\eps}
\newcommand\Functlim{\mathcal F_0}
\newcommand\dx{{\mathrm d}x}
\newcommand\dy{{\mathrm d}y}
\newcommand\ds{{\mathrm d}s}
\newcommand\dt{{\mathrm d}t}
\newcommand\dd{{\mathrm d}}

\newcommand\Psiinfty{{\Psi_\infty}}

\begin{document}
\begin{center}
%   {\Large
% {Vectorial, geometrically nonlinear, phase-field approximation of a cohesive fracture energy}}\\[5mm]
{\Large Phase-field approximation of 
a 
vectorial, geometrically nonlinear
cohesive fracture energy}\\[5mm]
{\today}\\[5mm]
Sergio Conti$^{1}$, Matteo Focardi$^{2}$ and Flaviana Iurlano$^{3}$\\[2mm]
{\em $^{1}$
 Institut f\"ur Angewandte Mathematik,
Universit\"at Bonn,\\ 53115 Bonn, Germany}\\[1mm]
 
{\em $^{2}$ DiMaI, Universit\`a di Firenze,\\ 50134 Firenze, Italy}\\[1mm]
{\em $^{3}$ Sorbonne Université, CNRS, Université de Paris,\\Laboratoire Jacques-Louis Lions, 75005 Paris, France}
\\[3mm]
    \begin{minipage}[c]{0.8\textwidth}
We consider a family of vectorial models for cohesive fracture, which may incorporate $\SO(n)$-invariance. The deformation belongs to the space of generalized functions of bounded variation and the energy contains an (elastic) volume energy, an opening-dependent jump energy concentrated on the fractured surface, and a Cantor part representing diffuse damage. We show that this type of functional can be naturally obtained as $\Gamma$-limit of an appropriate phase-field model. The energy densities entering the limiting functional can be expressed, in a partially implicit way, in terms of those appearing in the phase-field approximation.\end{minipage}
\end{center}

\tableofcontents 
\section{Introduction}

In variational models of nonlinear elasticity a hyper-elastic body with reference configuration $\Omega\subset\R^n$ ($n=2,3$) undergoes a deformation $u:\Omega\to\R^m$, whose stored energy reads as
\begin{equation}\label{elasticenergy}
\int_\Omega\Psi(\nabla u)\dx.
\end{equation}
External loads can be included, adding linear perturbations to this energy, and Dirichlet boundary conditions, restricting the set of admissible deformations $u$.
The energy density $\Psi:\R^{m\times n}\to[0,+\infty)$, acting on the deformation gradient $\nabla u$, 
is typically assumed to be minimized by matrices in the set of proper rotations $\SO(n)$  (with $m=n$) and to have $p$-growth at infinity, $p>1$. Correspondingly, the natural space for the deformation $u$ is a (subset of) the Sobolev space $W^{1,p}(\Omega;\R^m)$. There is an extensive literature on the theory of existence of minimizers of this type of functionals, and in particular the key property of weak lower semicontinuity of \eqref{elasticenergy} is closely related to the quasiconvexity of the energy density $\Psi$. 

Fracture phenomena, both brittle and cohesive, require a richer modeling framework.
Physically, cohesive fracture is often understood as
a gradual separation phenomenon: load-displacement curves usually exhibit an initial increase of the load up to a critical value, and a subsequent decrease to zero, which is the value indicating the complete separation \cite{BFM,dugdale,barenblatt,FokouaContiOrtiz2014}. See \cite{DelpieroTruskinovsky1998,DelpieroTruskinovsky2001} for discussions on different load-displacement behaviours. Evolutionary models (prescribing the crack path) have been studied in \cite{DMZ,BFM,Cagnetti,CT,LarsenSlastikov,Almi,ACFS,NegriScala,ThomasZanini,NegriVitali,CLO}, see also references therein. See \cite{DMG,CCF} for further results on the topic.

Variational models of fracture are typically formulated using the space $(G)BV$ of (generalised) functions of bounded variation 
\cite{FrancfortMarigo,BFM} and energy functionals of the form
\begin{equation}\label{fractureenergy}
\int_{\Omega}W(\nabla u)\dx+\int_\Omega l(\text{d}D^cu)+\int_{J_u}g([u],\nu_u)\text{d}\calH^{n-1}.
\end{equation}
The deformation $u\in (G)BV(\Omega;\R^m)$ may exhibit discontinuities along a $(n-1)$-dimensional set $J_u$. We denote by $[u]$ and $\nu_u$ the opening of the crack and the normal vector to the crack set $J_u$, respectively, while $D^cu$ represents the Cantor derivative of $u$ (see \cite{AFP} for the definition and the relevant properties of functions of bounded variation). Working within deformation theory, the functional \eqref{fractureenergy} contains both energetic and dissipative terms, which are physically distinct 
but need not be separated for this variational modeling.

The densities $W$, $l$, and $g$ entering \eqref{fractureenergy} 
need to satisfy suitable growth conditions. 
Lower semicontinuity of the functional imposes several restrictions, as for example that $l$ is positively one-homogeneous and quasiconvex, $W$ quasiconvex, and $g$ subadditive. 
Furthermore, $l$ needs to match, after appropriate scaling, both the behavior of $W$ at infinity and the behavior of $g$ near zero. These properties will be discussed in more detail below  (see, for example, Proposition~\ref{c:behaviour g at 0}).

The qualitative properties of $W$, $l$ and $g$ are selected 
according to the specific model of interest.
For instance, the brittle regime is modelled by a constant surface density $g$ and a superlinear bulk energy density $W$.
These choices in turn imply {that $l(\xi)=\infty$ for $\xi\ne0$}, so that $D^cu$ necessarily vanishes. The functional setting of the problem is then provided by the space of (generalised) special functions with bounded variation $(G)SBV(\Omega)$.
In contrast, in cohesive models $g$ is usually assumed to be approximately linear for small amplitudes and bounded.

The direct numerical simulation of functionals of the type \eqref{fractureenergy} is highly problematic,
due to the difficulty of finding good discretizations for $(G)BV$ functions and of differentiating the functional with respect to the coefficients entering the finite-dimensional approximation.
Therefore a number of regularizations have been proposed, of which one of the most successful {is given by} phase-field functionals. These are energies depending on a pair of variables $(u,v)$, having a Sobolev regularity, where $u$ {represents a}  regularization of a discontinuous displacement, while $v\in[0,1]$ can be interpreted as a damage parameter, indicating the amount of damage at each point of the body (where $v=1$ corresponds to the undamaged material and $v=0$ to the completely damaged material). The basic structure of a phase-field model is
\begin{equation}\label{phasefieldenergy}
\calF_\eps(u,v):=\int_\Omega \left( f_\eps^2(v) \Psi(\nabla u) + \frac{(1-v)^2}{4\eps} + \eps |\nabla v|^2 \right) \dx,
\end{equation} 
where $\eps>0$ is a small parameter, $f_\eps$ is a damage coefficient acting on the damage variable $v$, increasing from $0$ to $1$, and $\Psi$ is an elastic energy density, as in \eqref{elasticenergy}. The first term in \eqref{phasefieldenergy} represents the stored elastic energy, the other two terms represent the stored energy and dissipation due to the damage. 

Finding a variational approximation of the fracture model \eqref{fractureenergy} by phase-field models means to construct $f_\eps$ and $\Psi$ such that the functionals \eqref{phasefieldenergy}  converge, in the sense of $\Gamma$-convergence, to \eqref{fractureenergy} as $\eps\to0$. 
This is not an easy task in general. The brittle case ($g$ constant) in an antiplane shear, linear, framework ($m=1$, 
$\Psi$ quadratic) was the first outcome of this type \cite{AT90,AT}. 
It has been extended in several directions for different aims, giving rise to a very vast literature of both theoretical results \cite{Sh,
AFM,chambolle,addendum,HenaoMoracorralXu,amb-lem-roy,dm-iur,iur,FocardiIurlano,iur12,BEZ,CicaleseFocardiZeppieri2021phase} and numerical simulations \cite{bel-cos,BarSochenKiryati2006,bou,BFM,bur-ort-sul,bur-ort-sul2,BelzBredies} (for other regularizations, see also \cite{AFP,bra-dm-gar,Braides,Fusco2003,BraidesGelli2006} and references therein).
In particular, the extension of the results in \cite{AT} to the vector-valued (nonlinear) brittle case has been provided in \cite{Focardi2001}. The variational approximation of cohesive models is considerably more involved. The antiplane shear, linear, case was obtained through a double $\Gamma$-limit of energies with $1$-growth in \cite{AlicBrShah}, then generalized to the vector-valued case in \cite{AlFoc}. A drawback of these results is the $1$-growth with respect to $\nabla u$, which makes the approximants mechanically less meaningful and  numerically less helpful.

To overcome these problems, in \cite{ContiFocardiIurlano2016} 
we proposed a different approximation of \eqref{fractureenergy} 
in the antiplane shear case, with quadratic models of the form \eqref{phasefieldenergy}, based on a damage coefficient $f_\eps$ 
of the type
\begin{equation}\label{eqdamagecoeffintro}
f_\eps(s):=1\wedge\eps^{\sfrac12}\frac{\ell s}{1-s}\qquad s\in[0,1],\,\ell>0\,,
\end{equation}
and obtained $\Gamma$-convergence to a model 
of the type \eqref{fractureenergy} in the scalar ($m=1$) case.
We remark that $f_\eps$ is equal to $1$ when $v\sim1$ (elastic response) {and} to $0$ when $v\sim0$ (brittle fracture response). Moreover, the first addend in the energy in \eqref{phasefieldenergy} competes against the second term if $v$ is less than but close to $1$, and with all the terms of \eqref{phasefieldenergy} otherwise (pre-fracture response). 
This phase-field approximation of this scalar cohesive fracture 
% proposed and studied in \cite{ContiFocardiIurlano2016} 
was investigated numerically in \cite{FreddiIurlano}. A 1D cohesive quasistatic evolution (not prescribing the crack path) is presented in \cite{BCI} and related to the phase-field models of \cite{ContiFocardiIurlano2016}. A different approximation of \eqref{fractureenergy}, still in the scalar-valued framework, is obtained in \cite{DMOT} using elasto-plastic models.

In this paper we study the approximation of vector-valued cohesive models of the type \eqref{fractureenergy} via phase-field models of the type \eqref{phasefieldenergy} with the damage coefficient \eqref{eqdamagecoeffintro}, as proposed in \cite{ContiFocardiIurlano2016}.
In particular, this permits to extend the results of \cite{ContiFocardiIurlano2016} to a geometrically nonlinear framework, we refer to \eqref{e:f eps}-\eqref{eqpsipsiinf}
for the specific hypotheses on $\Psi$.
The main result is given in Theorem~\ref{t:finale}, the precise assumptions are discussed in Section~\ref{ss:data}.

In order to illustrate our result, let us consider the simplest model for the energy density $\Psi$ in finite kinematics and $m=n$,
\begin{equation}\label{PsiK}
\Psi_{2}(\xi):=\dist^2(\xi, \SO(n))=\min_{R\in \SO(n)}|\xi-R|^2.
\end{equation}
 With this choice, our main result Theorem~\ref{t:finale} states that the phase-field energies \eqref{phasefieldenergy} $\Gamma$-converge in the $L^1$-topology as $\eps\to0$ to the energy \eqref{fractureenergy}, with
\begin{equation}\label{eqdefWqc}
W(\xi):=(\dist^2(\cdot,\SO(n))\wedge\ell\,\dist(\cdot,\SO(n)))^\qc(\xi)\,,
\end{equation}
and
\[
l(\xi):=\ell|\xi|,\quad g(z,\nu):=g_\scal(|z|)\,,
\]
for every $\xi\in\R^{m\times n}$, $z\in\R^m$, $\nu\in S^{n-1}$, where $g_\scal$ is the surface energy density appearing in the {scalar} model (cf. formula \eqref{e:gscal} for the definition of $g_\scal$, item (iii) in Proposition~\ref{p:Psi1/2infty} 
with $W=h^\qc$ and $l=h^{\qc,\infty}$ {to justify the second equality}, 
{and Corollary \ref{euclidean} for the third equality}). 
As remarked above, $g$ coincides with $l$ asymptotically for infinitesimal amplitudes. 
Even in this simple case, the expression for $W$ is somewhat implicit, as it involves a quasiconvex envelope, which in most cases can only be approximately computed numerically. We remark that even $\Psi_2$ itself as defined in \eqref{PsiK} is not quasiconvex, we refer to \cite[Example~4.2]{Silhavy} for an explicit formula for its quasiconvex envelope $\Psi_2^{\qc}$ in the two-dimensional case.

We recall that in the scalar case several different choices for $f_\eps$ are possible without changing the overall {effect} of the approximation (cf. \cite[Section~4]{ContiFocardiIurlano2016}). A negative {power-law divergence} at $1$ however leads to a corresponding power-law behaviour of $g$ close to $0$ (cf. \cite[Theorem~7.4]{ContiFocardiIurlano2016}).
We expect these findings to have a natural generalization to the current vectorial setting, this requires additional technical ingredients that will be the object of future work \cite{ContiFocardiIurlanopgrowth}.

Let us now briefly discuss some aspects of the proof of Theorem~\ref{t:finale}.
One of the main difficulties is to identify the correct limit densities 
$W$, $g$, and $l$, given the density $\Psi$ and the damage coefficient $f_\eps$ of the phase-field \eqref{phasefieldenergy}. 
We do not expect that the cohesive energies that arise in the limit 
of our approximation exhaust all possible energies of the form \eqref{fractureenergy}, with densities $W$, $g$, and $l$ satisfying 
the growth conditions and matching properties specified above. Indeed, we prove that, even in the simplest case $\Psi(\xi):=|\xi|^2$,
$W$ is not convex (see Lemma~\ref{lemmahqcconb} below). 
{Thus,} at least in this case, the limit energy is not given by the relaxation 
of a functional defined on $SBV(\Omega)$ (cf. \cite[{Remark 2.2}]{BraidesCoscia}). Convex functions may be obtained as densities of the bulk term of the energy under more specific choices of the damage variable 
(see for example \cite{BIR}, where the damage variable is a characteristic function).

The effective surface energy density $g$ of the $\Gamma$-limit 
of the family $(\calF_\eps)$ is defined in an abstract fashion 
by an asymptotic minimization formula as the $\Gamma$-limit 
of a simpler family of functionals computed on functions jumping 
on a hyperplane (cf. \eqref{eqdefGsnu}). 
Alternative characterizations of $g$ useful along the proofs are provided both in Propositions~\ref{proplbboundary} and \ref{periodicity}, in which we show that the test sequences in the very definition of $g$ can be assumed to be periodic 
in $(n-1)$ mutually orthogonal directions and with $L^2$ integrability, and in Proposition~\ref{p:chargT}, where $g$ 
is represented in terms of an asymptotic homogenization formula.
Finally, the energy density $l$ of the Cantor part turns out to coincide with the recession function $W^\infty$ of $W$. Furthermore, 
an explicit characterization of $l$ in terms of $\Psi$ is given {in} Proposition~\ref{p:identification Cantor energy density}.

The proof of the lower bound in $BV$ is based on the blow-up technique. Roughly, to get the local estimate for the diffuse part 
given $(u_\eps,v_\eps)\to(u,v)$ in $L^1$, we analyze the 
asymptotic behaviour of the phase-field energies $\calF_\eps$ restricted on the $\delta$-superlevel sets of $v_\eps$, $\delta\in(0,1)$, and then let $\delta\uparrow 1$. 
More precisely, in Lemma~\ref{lemmaabdelta} we bound from below $\calF_\eps(u_\eps,v_\eps)$ in \eqref{phasefieldenergy} pointwise with a functional defined on $(G)SBV$, that is independent of $v_\eps$ and that is computed on a truncation of $u_\eps$ with the characteristic function of a suitable superlevel set of $v_\eps$ (depending on $\delta$). This is actually true up to an error related to the measure of the corresponding sublevel set of $v_\eps$, and up to prefactors depending on $\delta$ which are converging to $1$ as $\delta\uparrow 1$ for the volume term and vanishing for the surface term. 
The lower semicontinuity in $L^1$ of the diffuse part of such a functional then implies the lower bound. In addition, a slight variation of this argument shows directly that $(GBV(\Omega))^m$ 
is the domain of the $\Gamma$-limit.

Instead, to prove the local estimate for the surface part we show 
that under a surface scaling assumption we may
replace $v_\eps$ by its truncation at the threshold $\gamma_\eps$, being $\gamma_\eps$ the smallest $z\in[0,1]$ satisfying $f_\eps(z)=1$. The mentioned asymptotic minimization formula defining $g$ then provides a natural lower bound. 
The liminf inequality in $GBV$ is finally obtained by a further truncation argument.

The upper bound in $BV$ is proven through an integral representation argument. In particular, a direct computation provides a rough linear estimate from above, in fact optimal for the diffuse part. This allows to apply the representation result for linear functionals given in \cite{BouchitteFonsecaMascarenhas}. The sharp estimate for the surface density is obtained using the aforementioned characterization of $g$ involving periodic boundary conditions. The full upper bound in $GBV$ follows by a truncation argument.

The paper is structured as follows. In Section~\ref{ss:data} we present the model, introducing the main definitions and stating the 
$\Gamma$-convergence result {in} Theorem~\ref{t:finale}. In 
Section~\ref{simplified} we focus on a simplified model and we 
prove that in this case the limiting volume energy {density} 
$W$, obtained by quasiconvexification as in \eqref{eqdefWqc}, 
is not convex (Lemma~\ref{lemmahqcconb}). 
In Section~\ref{ss:properties g} several properties of the surface and Cantor densities are discussed. {In particular, 
Propositions~\ref{proplbboundary} and \ref{periodicity} deal with the change of boundary conditions within the minimum problem defining $g$. Proposition~\ref{p:chargT} provides an equivalent expression of 
$g$.}
Section~\ref{lowerbound} is devoted to the proof of the lower bound: Proposition~\ref{lowerboundBVsfc} proves the surface estimate in $BV$. The lower bound in $BV$ for the diffuse part is {addressed} in Proposition~\ref{lowerboundBVdff2}. {Finally, in 
Theorem~\ref{t:lbcomp}} the lower bound is extended to the full space $GBV$ {via a continuity argument (cf. Proposition~\ref{contF0})}. The proof of the upper bound is the object of Section~\ref{upperbound}, which concludes the proof of Theorem~\ref{t:finale}. Finally, Section~\ref{compandconv} addresses the problems of compactness and convergence of minimizers.

\section{Model}\label{model}
\subsection{General definitions}\label{ss:data}
In the entire paper $\Omega\subset\R^n$ is a bounded, open {set} with Lipschitz boundary, $\calA(\Omega)$ denotes 
the family of open subsets of $\Omega$ and $|\cdot|$ denotes the Euclidean norm, $|{\xi}|^2:=\sum_{ij}{\xi}_{ij}^2=\Tr\big({\xi}^T{\xi}\big)$ 
for ${\xi}\in\R^{m\times n}$.

For all $\eps>0$ we consider the functional $\Functeps:L^1(\Omega;\R^{m+1})\times\calA(\Omega)\to[0,\infty]$ given by
\begin{equation}\label{functeps}
 \Functeps(u,v;A):= \int_A \left( f_\eps^2(v) \Psi(\nabla u) + \frac{(1-v)^2}{4\eps} + \eps |\nabla v|^2 \right) \dx
\end{equation}
if $(u,v)\in W^{1,2}(\Omega;\R^m)\times W^{1,2}(\Omega;[0,1])$ and $\infty$ otherwise, where for every $s\in[0,1)$ we set
\begin{equation}\label{e:f eps}
 f(s):=\frac{\ell s}{1-s},\qquad
 f_\eps(s):= 1\wedge \eps^{\sfrac12} f(s),\qquad 
f_\eps(1):=1\,;
\end{equation} 
and $\ell>0$ is a parameter representing the critical yield stress. 
We write briefly $\Functeps(u,v):=\Functeps(u,v;\Omega)$, and analogously for all the functionals that shall be introduced in what follows.

We assume that $\Psi:\R^{m\times n}\to[0,\infty)$ is continuous and such that
\begin{equation}\label{e:Psi gc}
\Big(\frac1c |\xi|^2-c\Big)\vee0\le \Psi(\xi)\le c(|\xi|^2+1)
\hskip1cm\text{ for all }\xi\in\R^{m\times n}.
\end{equation}
We assume the ensuing limit to exist
\begin{equation}\label{eqdefPsiinfty}
 \Psiinfty(\xi):=\lim_{t\to\infty} \frac{\Psi{(t\xi)}}{t^2}\,,
\end{equation}
and that it is uniform on the set of $\xi$ with $|\xi|=1$. This means that for every $\delta>0$ there is $t_\delta>0$ such that $|\Psi(t\xi)/t^2-\Psiinfty(\xi)|\le \delta$ for all $t\ge t_\delta$ and all $\xi$ with $|\xi|=1$, which is the same as 
\begin{equation}\label{eqpsipsiinf}
|\Psi(\xi)-\Psiinfty(\xi)|\le \delta|\xi|^2\hskip5mm \text{ for all $|\xi|\ge t_\delta$}.
\end{equation}
By scaling, $\Psiinfty(t\xi)=t^2\Psiinfty(\xi)$ {and in particular $\Psiinfty(0)=0$.} Uniform convergence also implies $\Psiinfty\in C^0(\R^{m\times n})$.

We define $h:\R^{m\times n}\to[0,\infty)$ by
 \begin{equation}\label{e:h}
h(\xi):=\Psi(\xi)\wedge \ell{\Psi}^{\sfrac12}(\xi)
%=(\Psi_+(\xi)\wedge \ell\Psi_+^{\sfrac12}(\xi))-\Psi_-(\xi)
 \end{equation}
and denote by $h^\qc$ its quasiconvex envelope, 
\begin{equation}\label{eq:hqc}
 h^\qc(\xi):=\inf \Big\{\int_{(0,1)^n} h(\xi+\nabla \varphi) \dx: \varphi\in C^\infty_c((0,1)^n;\R^m) \Big\}.
\end{equation}
From \eqref{e:Psi gc} we infer that for every $\xi\in\R^{m\times n}$ 
\begin{equation}\label{e:h gc}
\Big(\frac1c |\xi|-c\Big)\vee0\leq h^\qc(\xi)\leq h(\xi)\leq c(|\xi|+1).
\end{equation}
Let $h^\qcinfty$ be its recession function,
\begin{equation}\label{eqdefqcinfty}
 h^\qcinfty(\xi):=\limsup_{t\to\infty} \frac{h^\qc(t\xi)}{t}.
\end{equation}
We remark that the definitions of $h^\qcinfty$ and $\Psiinfty$ differ, to reflect the different growth of the two functions, quadratic for $\Psi$ and linear for $h$. Recall that $h^\qcinfty$ is itself a quasiconvex function
\cite[Rem.~2.2 (ii)]{fonsecamueller1993relaxation}. 
Therefore, it is locally Lipschitz continuous (cf. for instance \cite[Theorem~5.3  (ii)]{Dacorogna}). 
Moreover, in Proposition~\ref{p:identification Cantor energy density} below we shall prove that 
\begin{equation}\label{eqhinftypsiinf}
 h^\qcinfty(\xi)=\ell(\Psi^{\sfrac12})^\qcinfty(\xi)\,,
\end{equation}
where the latter quantity is defined as in \eqref{eq:hqc}-\eqref{eqdefqcinfty}.
We remark that, at variance with the convex case, one cannot in general replace the $\limsup$ in \eqref{eqdefqcinfty} by a limit \cite[Theorem~2]{muller1992quasiconvex}.

For all open subsets $A\subseteq\R^n$, $u\in W^{1,2}(A;\R^m)$ and $v\in W^{1,2}(A;[0,1])$ 
it is convenient to introduce the functional 
\begin{equation}\label{Feps*}
\calF_\eps^{\infty}(u,v;A):=\int_A \Big(
%\frac{\eps {\ell^2} v^2}{(1-v)^2} 
{\eps f^2(v)}\Psiinfty(\nabla u) 
+ \frac{(1-v)^2}{4\eps}+\eps|\nabla v|^2 \Big)\dx.
\end{equation}
The first term is interpreted to be zero whenever $\nabla u=0$, even if $v=1$.
For {any} $\nu\in S^{n-1}$ we {fix a cube} $Q^\nu$ with side length 1, centered in the origin, and with one side parallel to $\nu$. 
We write $Q^\nu_r:=rQ^\nu$.
We
define $g:\R^m\times S^{n-1}\to[0,\infty)$ by
\begin{equation}\label{eqdefGsnu}
g(z,\nu):=\inf \{\liminf_{j\to\infty} 
\calF^{\infty}_{\eps_j}(u_j,v_j, Q^\nu): \|u_j- z\chi_{\{x\cdot\nu>0\}}\|_{L^1(Q^\nu)}\to0, \eps_j\to 0\}.
\end{equation}
Here $u_j\in W^{1,2}(Q^\nu;\R^m)$ and $v_j\in W^{1,2}(Q^\nu;[0,1])$; obviously one can restrict to sequences $v_j\to1$ in $L^1(Q^\nu)$.
We refer to Section~\ref{ss:properties g} for the discussion of several properties of $g$.

We will prove the following result.
\begin{theorem}\label{t:finale}
Let $\calF_\eps$ be the functional defined in \eqref{functeps}. Then for all $(u,v)\in L^1(\Omega;\R^{m+1})$ 
it holds
$$
\Gamma({L^1})\text{-}\lim_{\eps\to0}\calF_\eps(u,v)=\calF_0(u,v)\,,
$$
where
\begin{equation}\label{F0}
\calF_0(u,v):=\int_\Omega h^\qc(\nabla u)\dx +\int_\Omega h^{\qcinfty}(\mathrm d D^cu)
+ \int_{J_u}g([u],\nu_u)d\calH^{n-1},
\end{equation}
if $u\in (GBV\cap L^1(\Omega))^m$ and $v=1$ $\calL^n$-a.e., and $\calF_0(u,v):=\infty$ otherwise.  
\end{theorem}	
\begin{remark}
{One can imagine several natural generalizations of 
Theorem~\ref{t:finale}.
For example, one could allow $\Psi$ 
to take negative values, replacing \eqref{e:Psi gc}  by 
\begin{equation*}
\frac 1c |\xi|^2-c\le \Psi(\xi)\le c(|\xi|^2+1).
%\hskip1cm\text{ for all }\xi\in\R^{m\times n}\,.
\end{equation*}
Whereas in purely elastic models like \eqref{elasticenergy} one can add a constant to the energy density without any change in the analysis, 
the presence of the prefactor $f_\eps^2(v)$ renders this modification nontrivial, and influences several steps in the proof. 
Indeed, the construction in Step 1 of the proof of Theorem~\ref{t:limsupndim} shows that the definition of $h$ in \eqref{e:h} needs to be replaced by
 \[
 h(\xi):=\Psi(\xi)\wedge\ell\Psi_+^{\sfrac12}(\xi)\,. \]
}{Alternatively, one could replace the quadratic growth of $\Psi$ in \eqref{e:Psi gc} by
$p$-growth, $p>1$. The requirement that the effective energy scales linearly for large strains leads to corresponding adaptations in the other parts of the functional.
}

{For simplicity we only address here 
the growth condition in \eqref{e:Psi gc}.}

% {
% The conclusion of Theorem~\ref{t:finale} holds even in case $\Psi$ 
% may take negative values. More precisely, if \eqref{e:Psi gc} is substituted by 
% \begin{equation*}
% \frac 1c |\xi|^2-c\le \Psi(\xi)\le c(|\xi|^2+1)
% %\hskip1cm\text{ for all }\xi\in\R^{m\times n}\,,
% \end{equation*}
% then the $\Gamma$-convergence of $(\calF_\eps)$ to $\calF_0$ can be proven provided that $h$ is defined by 
% \[
% h(\xi):=\Psi(\xi)\wedge\ell\Psi_+^{\sfrac12}(\xi)\,.
% \]
% We establish Theorem~\ref{t:finale} under the slightly milder growth condition in \eqref{e:Psi gc} only, 
% since otherwise one would have to deal with additional technicalities.
% 
% Finally, the quadratic growth of $\Psi$ in \eqref{e:Psi gc} is not essential for Theorem~\ref{t:finale}: 
% any $p$-growth, $p>1$, can be considered similarly.
% }
\end{remark}

\paragraph{Notation.}
For $A$ open we denote by $\calM^+(A)$ the set of positive Radon measures on the set $A$, and by
$\calM_b^+(A)$ the subset of bounded measures.
% We write $\calA(\Omega)$ for the family of open subsets of $\Omega$.
For $A\in\calA(\Omega)$, 
\begin{equation*}
\begin{split}
 \Gamma(L^1)\hbox{-}\liminf \calF_\eps(u,v;A):=
 \inf\Bigl\{ & \liminf_{\eps\to0}\calF_\eps(u_\eps,v_\eps;A): \\
 &
 (u_\eps,v_\eps)\to (u,v) \text{ in } L^1(\Omega;\R^{m+1})\Bigr\}
 \end{split}
\end{equation*}
and correspondingly for the $\Gamma\mbox{-}\limsup$. 
We drop the dependence on the reference set $A$ if $A=\Omega$.
We refer to Section~\ref{s:scalar case} for the definition of the vector measure $D^cu$ if $u\in (GBV(\Omega))^m$.

\subsection{Simplified model}\label{simplified}
In this Section we consider the simplified case 
$\Psi_\simp({\xi}):=|{\xi}|^2$, the corresponding unrelaxed energy density $h_\simp:\R^{m\times n}\to [0,\infty)$,
\begin{equation}\label{eqdefh}
 h_\simp(\xi):={|\xi|^2\wedge \ell |\xi|},
\end{equation}
its quasiconvex envelope $h_\simp^\qc$ as in \eqref{eq:hqc}, 
and its recession function $h_\simp^\qcinfty$ as in \eqref{eqdefqcinfty}.
These functions only depend on the space dimension and the single parameter $\ell>0$, which could be eliminated by scaling.

In this case it is possible to obtain simple closed-form expressions 
for several of the quantities defined above. However, 
an explicit characterization of the quasiconvex envelope in \eqref{eq:hqc} remains difficult. Indeed, we show 
in Lemma~\ref{lemmahqcconb}\ref{lemmahqcconb3}
below that even in this simplified setting the result is not convex. Since it has linear growth, lower bounds with polyconvexity cannot be used, and an explicit determination of $h^\qc_\simp$ seems difficult. We believe this to be a strong indication that in most cases of interest the function $h^\qc$ can only be approximated numerically, and not computed explicitly. Lemma~\ref{lemmahqcconb} and this observation are not used in the proof of Theorem~\ref{t:finale}.

\begin{lemma}\label{lemmahqc}
For $n,m\ge 1$ let $h_\simp:\R^{m\times n}\to[0,\infty)$ be defined as in \eqref{eqdefh}. Then:
\begin{enumerate}
\item\label{lemmahqcconv} its convex envelope is
 \begin{equation}\label{eqhcoonv}
 h_\simp^\conv(\xi)=\begin{cases}
              |\xi|^2 , & \text{ if } |\xi|\le \frac\ell2,\\
              \ell|\xi|-\frac{\ell^2}4,& \text{ if } |\xi|> \frac\ell{2};
             \end{cases}
             \end{equation}
\item\label{lemmahqcgrowth}
$\ell|\xi|-\frac{\ell^2}4\le h_\simp^\qc(\xi)\le \ell |\xi|$ for all $\xi\in\R^{m\times n}$;
\item\label{lemmahqcregr} $h_\simp^\qcinfty(\xi)=\ell |\xi|$ and the $\limsup$ in \eqref{eqdefqcinfty} is a limit.
% \item\label{lemmahqcqcnotc} If $n,m\ge 2$ then $h^\qc$ is not convex;
% \item\label{lemmahqcsubadd} there is $c>0$ such that
% \begin{equation}\label{eqhsubadditive}
%  h_\simp^\qc(\xi+\eta)\le c (h_\simp^\qc(\xi)+h_\simp^\qc(\eta)) \text{ for all } \xi,\eta\in\R^{m\times n}.
% \end{equation}
\end{enumerate}
\end{lemma}
\begin{proof}
\ref{lemmahqcconv}:
To prove \eqref{eqhcoonv}
we consider $h_\scal:[0,\infty)\to[0,\infty)$ defined by
\begin{equation}\label{e:hascal}
h_\scal(t):= {t^2 \wedge \ell t}
\end{equation}
and compute its convex envelope
 \begin{equation}\label{e:hascalconv}
  {h^\conv_{\scal}}(t)=\begin{cases}
              t^2, & \text{ if } 0\le t\le \frac\ell{2},\\
               \ell t-\frac{\ell^2}4,& \text{ if } t> \frac\ell{2}.
             \end{cases}
 \end{equation}
 Let $\eta\in\R^{m\times n}$ with $|\eta|=1$. Then
 {$h_\simp(t\eta)=h_{\scal}(t)$}, hence {$h_\simp^\conv(t\eta)\le {h_\scal^\conv} (t)$}. This proves one inequality in 
 \eqref{eqhcoonv}. At the same time, 
 {$h^\conv_{\scal}(|\xi|)\le h_{\scal}(|\xi|)=h_\simp(\xi)$}, and  
 the function 
 {$\xi\mapsto{h_\scal^\conv}(|\xi|)$} is convex, since $h^\conv_{\scal}$ is convex and nondecreasing in $[0,\infty)$ and $\xi\mapsto |\xi|$ is convex.   This proves the second inequality in  \eqref{eqhcoonv}.
 
 \ref{lemmahqcgrowth}: This follows immediately from the fact that
 $\ell|\xi|-\frac{\ell^2}4\le h^\conv_{\simp}(\xi)\le h_{\simp}^\qc(\xi)\le h_{\simp}(\xi)\le \ell|\xi|$ for any $\xi\in\R^{m\times n}$.

 \ref{lemmahqcregr}: This follows immediately from the  definition and \ref{lemmahqcgrowth}.
 
% \ref{lemmahqcsubadd}: 
% We first show that
% \begin{equation}\label{eqhsubadd}
%  h(A+B)\le 2 ((h(A)+h(B)) \text{ for all } A,B\in\R^{m\times n}.
% \end{equation}
% We can assume $A,B\in\R^{m\times n}$ and $|A|\le|B|$.
% We distinguish three cases. 
% 
% If $|A|\wedge|B|\ge\ell$ then
% $h(A)+h(B)=\ell(|A|+|B|)\ge \ell |A+B|\ge h(A+B)$.
% 
% If $|A|\vee|B|\le\ell$ then
% $h(A)+h(B)=|A|^2+|B|^2\ge \frac12 |A+B|^2\ge \frac12 h(A+B)$.
% 
% If $|A|<\ell<|B|$ then $|A+B|\le 2|B|$, so that 
% $h(A)+h(B)\ge \ell|B|\ge\frac12 \ell|A+B|\ge \frac12 h(A+B)$. This concludes the proof of \eqref{eqhsubadd}.
% 
% We then show that \eqref{eqhsubadd} implies \eqref{eqhsubadditive}. Fix $\eps>0$. Let $\varphi$, $\psi\in W^{1,\infty}_0((0,1)^n;\R^m)$ be such that
% \begin{equation*}
% \int_{(0,1)^n} h(A+\nabla \varphi) \dx \le h^\qc(A)+\eps \text{ and }
% \int_{(0,1)^n} h(B+\nabla\psi) \dx \le h^\qc(B)+\eps.
% \end{equation*}
% Then $\varphi+\psi\in W^{1,\infty}_0((0,1)^n;\R^m)$, hence
% \begin{equation*}
% \begin{split}
% h^\qc(A+B)\le & \int_{(0,1)^n} h(A+B+\nabla(\varphi+\psi)) \dx \\
% \le& 2\int_{(0,1)^n} h(A+\nabla \varphi) \dx +\int_{(0,1)^n} h(B+\nabla\psi) \dx \\
% \le& 2(h^\qc(A)+h^\qc(B))+2\eps.
% \end{split}
% \end{equation*}
% Since $\eps$ was arbitrary, this concludes the proof (with $c=2$).
\end{proof}

We next prove that the quasiconvex envelope $h_\simp^\qc$ is not convex.
For this we need a linear algebra statement that we present first.
 
 \newcommand\Amatrix{{\xi}}   

\begin{lemma}\label{lemmaTSA}
 Let 
 \begin{equation}
  \R^{m\times n\times n}_\sym:=\{\Gamma\in \R^{m\times n\times n}:
  \Gamma_{ijk}=\Gamma_{ikj} \}
 \end{equation}
and consider for $\Amatrix\in\R^{m\times n}$ the {linear} map 
$T:\R^{m\times n\times n}_\sym\to \R^{m\times n\times n}$ 
of the form
 \begin{equation}
  (T\Gamma)_{ijk}:=\Gamma_{ijk}- \Amatrix_{ij}\sum_{a,b}\Amatrix_{ab}\Gamma_{abk}\,.
 \end{equation}
If $\rank \Amatrix\ge 2$, then
$T$ is injective. In particular, it has an inverse $S:T(\R^{m\times n\times n}_\sym)\to \R^{m\times n\times n}_\sym$.
\end{lemma}
\begin{proof}
 It suffices to show that there is no  $\Gamma\in \R^{m\times n\times n}_\sym$ with $T\Gamma=0$ and $\Gamma\ne 0$. We assume it exists and define $v\in\R^n$ {componentwise} by
 \begin{equation}
  v_k:=\sum_{a,b} \Amatrix_{ab}\Gamma_{abk}.
 \end{equation}
Then $T\Gamma=0$ is equivalent to
\begin{equation*}
 \Gamma_{ijk}-\Amatrix_{ij}v_k=0,
\end{equation*}
hence $\Gamma_{ijk}=\Amatrix_{ij}v_k$, for all $i$, $j$, and $k$.  
{Moreover, $\Gamma\neq 0$ in turn implies that $v\neq 0$.}
From $\Gamma\in \R^{m\times n\times n}_\sym$ we obtain
\begin{equation*}
 \Amatrix_{ij}v_k=\Amatrix_{ik}v_j.
\end{equation*}
As $\rank \Amatrix\ge2$ there is a vector $w\in\R^n$ with $v\cdot w=0$ and ${\xi}w\ne0$. We take the scalar product of the previous equation with $w$ and obtain
\begin{equation*}
\sum_k \Amatrix_{ij}v_kw_k=\sum_k \Amatrix_{ik}v_jw_k
\end{equation*}
which gives $0=v_j (\Amatrix w)_i$ for all {$i$ and} $j$. As $v\ne0$ and $\Amatrix w\ne0$, this is a contradiction. 
\end{proof}

\begin{lemma}\label{lemmahqcconb}
Let $\Amatrix\in\R^{m\times n}$. 
\begin{enumerate}
 \item\label{lemmahqcconb1} If $|\Amatrix|\le \frac\ell2$, then $h_\simp(\Amatrix)=h_\simp^\qc(\Amatrix)=h_\simp^\conv(\Amatrix)$.
\item\label{lemmahqcconb2}
If $\rank \Amatrix\le 1$, then $h_\simp^\qc(\Amatrix)=h_\simp^\conv(\Amatrix)$.
\item\label{lemmahqcconb3}
If $\rank \Amatrix\ge 2$ and $|\Amatrix|>\frac\ell2$, then 
$h_\simp^\conv(\Amatrix)<h_\simp^\qc(\Amatrix)$.
\end{enumerate}
\end{lemma}
\begin{proof}
 We work for $\ell=1$ (the general case can be reduced to this one by a rescaling), to shorten notation we write $h$ for $h_\simp$.
 
\ref{lemmahqcconb1}: It is clear that $h^\conv\le h^\qc\le h$. 
If $|\Amatrix|\le \frac12$ then $h^\conv(\Amatrix)=h(\Amatrix)$ 
{(cf. \eqref{eqhcoonv}),} and the assertion then follows.
 
 \ref{lemmahqcconb2}:
 If $\rank \Amatrix=1$ with $|\Amatrix|>\frac12$, then  for any 
 $t>|\Amatrix|$ one has
 \begin{equation*}
  \Amatrix = \frac{t-|\Amatrix|}{t-\frac12} \frac{\Amatrix}{2|\Amatrix|} + 
   \frac{|\Amatrix|-\frac12 }{t-\frac12} \frac{t\Amatrix}{|\Amatrix|}
 \end{equation*}
 and by rank-one convexity of $h^\qc$ we obtain
 \begin{equation*}
  h^\qc(\Amatrix)\le \frac{t-|\Amatrix|}{t-\frac12} 
  h\Big(\frac{\Amatrix}{2|\Amatrix|}\Big) + 
   \frac{|\Amatrix|-\frac12 }{t-\frac12} 
   h\Big(t\frac{\Amatrix}{|\Amatrix|}\Big)
   \le \frac{t-|\Amatrix|}{t-\frac12} \frac14 + 
   \frac{|\Amatrix|-\frac12 }{t-\frac12} t.
 \end{equation*}
Taking $t\to\infty$ shows that $h^\qc(\Amatrix)\le |\Amatrix|-\frac14=h^\conv(\Amatrix)$. 
Recalling $h^\conv\le h^\qc$ concludes the proof.

\ref{lemmahqcconb3}:
 We assume that $\rank \Amatrix\ge 2$ and $|\Amatrix|> \frac12$, and 
 show that $h^\conv(\Amatrix)<h^\qc(\Amatrix)$. From the explicit formulas given in Lemma~\ref{lemmahqc}\ref{lemmahqcconv} we know that $h^\conv(\Amatrix)<h(\Amatrix)$, from general theory $h^\conv\le h^\qc$.
 
 Assume by contradiction that $h^\conv(\Amatrix)=h^\qc(\Amatrix)$. Then there is 
a sequence $\varphi_j\in C^{\infty}((0,1)^n;\R^m)$ such that $\varphi_j(x)=\Amatrix x$ on $\partial(0,1)^n$ and
\begin{equation}\label{eqhid2}
 h^\conv(\Amatrix)=\lim_{j\to\infty} \int_{(0,1)^n} h(\nabla \varphi_j) \dx.
\end{equation}
We consider the affine function $L:\R^{m\times n}\to\R$, 
\begin{equation*}
L({\eta}):=\frac{{\eta}\cdot \Amatrix}{|\Amatrix|}-\frac14.
\end{equation*}
One easily
checks that $h^\conv(t\Amatrix)=L(t\Amatrix)=t|\xi|-\frac14$ 
for $t\ge \frac1{2|\Amatrix|}$ {(cf. \eqref{eqhcoonv}),} and 
since $|\Amatrix|>\frac12$ this in particular holds for $t= 1$.
 Linearity and the boundary values of $\varphi_j$ imply
\begin{equation*}
\int_{(0,1)^n}
L(\nabla \varphi_j)\dx 
=L\left(
\int_{(0,1)^n} \nabla \varphi_j\, \dx \right)
=L(\Amatrix).
\end{equation*}
Subtracting from \eqref{eqhid2},
and letting $g:=h-L$, leads to
\begin{equation}\label{eqghLid2}
\lim_{j\to\infty}\int_{(0,1)^n}
g(\nabla \varphi_j)\dx =0.
\end{equation}
We next show that $g({\eta})$ controls the distance of the matrix ${\eta}$ from the set $\R \Amatrix$. To do this, 
for ${\eta}\in\R^{m\times n}$ we define the orthogonal projections
\begin{equation*}
 {\eta}^\parallel := \frac{{\eta}\cdot \Amatrix}{|\Amatrix|} \in\R\hskip5mm\text{ and }\hskip5mm
 {\eta}^\perp:= {\eta}-\frac{\Amatrix}{|\Amatrix|} {\eta}^\parallel\in \R^{m\times n},
\end{equation*}
so that $|{\eta}|^2=|{\eta}^\parallel|^2+|{\eta}^\perp|^2$
and $L({\eta})= {\eta}^\parallel-\frac14$.

We first consider the case $|{\eta}|\ge 1$, so that $h({\eta})=|{\eta}|$.
Assume for a moment that both
${\eta}^\parallel$
and ${\eta}^\perp$ do not vanish.
Letting $\gamma:=|{\eta}^\perp|/|{\eta}^\parallel|$,
\begin{equation*}
 g({\eta})=|{\eta}|-L({\eta})\ge |{\eta}^\parallel| \sqrt{1+\gamma^2}-|{\eta}^\parallel|=
 \frac{\sqrt{1+\gamma^2}-1}{\gamma} |{\eta}^\perp|.
\end{equation*}
Let now $\eps\in(0,1]$. If $\gamma\le\eps$, then $|{\eta}^\perp|\le\eps |{\eta}^\parallel|$.
Otherwise, by monotonicity of $t\mapsto (\sqrt{1+t^2}-1)/t$ we have
$g({\eta})\ge (\sqrt{1+\eps^2}-1)|{\eta}^\perp|/\eps$. Therefore 
\begin{equation}\label{ghammage12}
 |{\eta}^\perp|\le \eps |{\eta}^\parallel| + \frac{\eps}{\sqrt{1+\eps^2}-1} g({\eta})%\quad
% \text{ for all } {\eta}\in\R^{m\times n} \text{ with }|{\eta}|\ge 1.
\end{equation}
{for all ${\eta}\in\R^{m\times n} \text{ with }|{\eta}|\ge 1$
(the two cases ${\eta}^\parallel=0$
and ${\eta}^\perp=0$ follow by continuity).}
% 
% If , then {$\eta=\eta^\perp$ and}
% $g({\eta})=|{\eta}|-L({\eta})=|{\eta}^\perp|+\frac14\ge |{\eta}^\perp|$. If $\eta^\perp=0$, then
% {$\eta=\eta^\parallel \xi/|\xi|$ and} $g(\eta)=
% |\eta^\parallel|-\eta^\parallel+
% \frac14\ge\frac14$.
% Otherwise, l
If instead $|{\eta}|\le 1$,
\begin{equation*}
 g({\eta})=|{\eta}|^2-L({\eta})=|{\eta}^\parallel|^2+|{\eta}^\perp|^2-{\eta}^\parallel+\frac14 \ge |{\eta}^\perp|^2.
\end{equation*}
Therefore for any $\eps\in(0,1]$ we have {for all 
${\eta}\in\R^{m\times n} \text{ with } |{\eta}|\le 1$}
\begin{equation}\label{ghammale12}
 |{\eta}^\perp|\le \eps +\frac1\eps |{\eta}^\perp|^2\le \eps + \frac1\eps g({\eta})\,. 
 %\text{ for all }  {\eta}\in\R^{m\times n} \text{ with } |{\eta}|\le 1.
\end{equation}
Combining \eqref{ghammage12} and \eqref{ghammale12} we see that for any $\eps\in(0,1]$ there is $C_\eps>0$ such that
{for all ${\eta}\in\R^{m\times n}$}
\begin{equation*}
 |{\eta}^\perp|\le \eps (|{\eta}^\parallel|+1)+C_\eps g({\eta})\,. 
 %\text{ for all }  {\eta}\in\R^{m\times n} .
\end{equation*}
In particular, for any $j$ we have
\begin{equation*}
 |\nabla \varphi_j^\perp|\le \eps (|\nabla \varphi_j^\parallel|+1)+C_\eps g(\nabla \varphi_j).
\end{equation*}
We integrate over $(0,1)^n$, take the limit $j\to\infty$ and
recall that  $g(\nabla \varphi_j)\to0$ in $L^1$ by \eqref{eqghLid2}. We obtain
\begin{equation*}
 \limsup_{j\to\infty} \int_{(0,1)^n} |\nabla \varphi_j^\perp| \dx \le 
 \eps \limsup_{j\to\infty} \int_{(0,1)^n}  (|\nabla \varphi_j^\parallel| +1) \dx 
\end{equation*}
for any $\eps\in(0,1]$. 
By \eqref{eqhid2} and Lemma~\ref{lemmahqc}\ref{lemmahqcgrowth} the sequence 
$\nabla \varphi_j$ is bounded in $L^1$, and since $\eps$ was arbitrary we conclude that
\begin{equation}\label{eqephijperpzero2}
 \limsup_{j\to\infty} \int_{(0,1)^n} |\nabla \varphi_j^\perp| \dx=0.
\end{equation}

We next prove that \eqref{eqephijperpzero2} implies that $\nabla \varphi_j$ converges to the constant $\Amatrix$ strongly in weak-$L^1$. 
To do this we show that standard singular integral estimates imply rigidity. To simplify notation, we write $u_j(x):=\varphi_j(x)-\Amatrix x$ and
% $p:=(\partial_1\varphi_1+\partial_2\varphi_2)/2$ and 
$R_j:=\nabla \varphi_j^\perp= \nabla u_j^\perp$, both
extended by zero to the rest of $\R^n$,
in the next steps. 
We observe that
\begin{equation*}
 R_j=\nabla u_j- \Amatrix \frac{\Amatrix\cdot \nabla u_j}{|\Amatrix|^2}
 =\nabla u_j-  \tilde \Amatrix(\tilde  \Amatrix\cdot \nabla u_j)
\end{equation*}
where $\tilde \Amatrix:=\frac\Amatrix{|\Amatrix|}$.
Taking a derivative, and writing components, we obtain
\begin{equation*}
(\nabla R_j)_{cdk}=(\nabla^2 u_j)_{cdk}- \tilde \Amatrix_{cd} \sum_{a,b} \tilde \Amatrix_{ab}( \nabla^2 u_j)_{abk} = (T (\nabla^2u_j){)_{cdk}},
\end{equation*}
with $T$ obtained from $\tilde \Amatrix$ as in Lemma~\ref{lemmaTSA}.
Let $S$ be the inverse operator. Then
\begin{equation*}
 \nabla^2u_j = S(\nabla R_j),
\end{equation*}
so that in particular $\Delta u_j$ is given by a linear combination of the components of $\nabla R_j$, with coefficients which depend only on $\Amatrix$.
As $u_j(x)=0$ outside $(0,1)^n$, we obtain, denoting by $N$ the fundamental solution of Laplace's equation in $\R^n$
{(which solves $-\Delta N=\delta_0$),}
\begin{equation*}
{-}\partial_r u_j
{=\partial_r (N \ast \Delta u_j)=}
\partial_r (N\ast \Tr S(\nabla R_j))
 =\Tr S(\Lambda_{{r}}(R_j)), %\quad \text{for } r=1,\dots,n,
\end{equation*}
{for every $r=1,\dots,n$,}
where we have set ${(\Lambda_r(R_j))}_{cdk}:={\partial_r\partial_k} N*(R_j)_{cd}$ {(recall that $R_j=0$ outside of $(0,1)^n$)}, 
and $(\Tr\Gamma)_l:=\sum_{i=1}^n\Gamma_{lii}$, 
for every $l=1,\dots,m$ and $\Gamma\in\R^{m\times n\times n}$.
By \cite[Theorem~4(b), page~42]{Steinsmall} we see that the operator $R\mapsto\Lambda_{{r}}(R)$ is of weak type 
$(1,1)$, so that
\begin{equation*}
 \|\nabla u_j\|_{w-L^1((0,1)^n)}\le c \|R_j\|_{L^1((0,1)^n)},
\end{equation*}
with $c$ depending only on $\Amatrix$.
Recalling the definition of $u_j$ and $R_j$ as well as \eqref{eqephijperpzero2},
\begin{equation*}
\lim_{j\to\infty}
 \|\nabla \varphi_j-\Amatrix\|_{w-L^1((0,1)^n)}\le c \lim_{j\to\infty} \|\nabla \varphi_j^\perp\|_{L^1((0,1)^n)}=0.
\end{equation*}
To conclude the proof we choose $z\in (h^\conv(\Amatrix), h(\Amatrix))$ (here we use again that $|\Amatrix|>\frac12$).
By continuity of $h$, there is
$\delta>0$ such that
$h(\eta)\ge z$ for all $\eta\in \R^{m\times n}$ with $|\eta-\Amatrix|<\delta$. By definition of the weak-$L^1$ norm,
\begin{equation*}
\limsup_{j\to\infty} \calL^n(\{x\in (0,1)^n: |\nabla \varphi_j-\Amatrix|\ge\delta\}) \le \limsup_{j\to\infty} \frac{ \|\nabla \varphi_j-\Amatrix\|_{w-L^1}}{\delta} =0\,.
\end{equation*}
Therefore, recalling that $h\ge 0$ pointwise,
\begin{equation*}
\begin{split}
 \liminf_{j\to\infty} 
 \int_{(0,1)^n} h(\nabla \varphi_j)\dx \ge&
 \liminf_{j\to\infty} 
 z\calL^n(\{x\in (0,1)^n: |\nabla \varphi_j-\Amatrix|{<}\delta\}) \\
=& z>h^\conv(\Amatrix).
\end{split}
 \end{equation*}
This contradicts \eqref{eqhid2} and concludes the proof.
\end{proof}

\clearpage
\section{Energy densities of the surface and Cantor part}\label{ss:properties g}

In this section we discuss several properties of the energy densities $g$ and $h^\qcinfty$.
{We warn the reader that while the results dealing with $g$ contained in subsections~\ref{ss:g} 
{and~\ref{subsecstructpropg}}
will be crucial in the proof of ¸Theorem~\ref{t:finale}, those in subsection \ref{ss:hqcinfty} will not be employed in that proof. Actually, Proposition~\ref{p:behaviour g at 0} and Corollary~\ref{c:behaviour g at 0} take advantage of Theorem~\ref{t:finale} itself (in particular of the lower semicontinuity of $\Gamma$-limits).}
%\MFF{La rappresentazione di $h^\qcinfty$ in termini di $\Psi$  non usa il Theorem~\ref{t:finale}!}
\subsection{Equivalent characterizations of $g(z,\nu)$}\label{ss:g}
{We show below that
	we may reduce the test sequences in the definition of $g(z,\nu)$ in \eqref{eqdefGsnu} to those converging
	in $L^2$ and satisfying periodic boundary conditions in $(n-1)$ directions orthogonal to $\nu$ and mutually 
	orthogonal to each other. This is the content of the next two propositions, {which will be crucial in the proof of the upper bound for the surface part (Theorem~\ref{t:limsupndim} Step 2). The proof draws inspiration from that of \cite[Lemma 4.2]{BarrosoFonseca}.}}
We fix a mollifier {$\varphi_1\in C^\infty_c(B_1)$, with $\int_{B_1}\varphi_1\dx=1$, and set $\varphi_\eps(x):=\eps^{-n}\varphi_1(x/\eps)$ in $B_\eps$}. 
% \MFF{vogliamo citare Barroso-Fonseca come fonte di ispirazione?}
\begin{proposition}\label{proplbboundary}
	Assume {an} optimal sequence in \eqref{eqdefGsnu} converges in $L^2(Q^\nu;\R^{m+1})$. 
	Then there are $\eps_j\to0$, $(u_j^*,v_j^*)\to (z\chi_{\{x\cdot\nu>0\}},1)$ in $L^2(Q^\nu;\R^{m+1})$, {with $v_j^*\in[0,1]$
		$\calL^n$-a.e. in $\Omega$,} such that
	\begin{equation*}
	\lim_{j\to\infty} 
	\calF^{\infty}_{\eps_j}({{u_j^*,v_j^*}}; Q^\nu)\le  g(z,\nu)
	\end{equation*}
	and
	\begin{equation}\label{eqpropboundrybv}
	u_j^*=(z\chi_{\{x\cdot\nu>0\}})\ast\varphi_{\eps_j}\,,\hskip1cm
	v_j^*=\chi_{\{|x\cdot\nu|\ge2\eps_j\}}\ast\varphi_{\eps_j} \hskip1cm
	\text{ on } \partial Q^\nu.
	\end{equation}
\end{proposition}
\begin{proof}
	{\bf Step 1. Construction of $u_j^*$ and $v_j^*$.}
	Pick {$\eps_j\to0$, $v_j$ and} $u_j\to z\chi_{\{x\cdot\nu>0\}}$ in $L^2(Q^\nu;\R^m)$ such that
	\begin{equation*}
	g(z,\nu)= \lim_{j\to\infty} 
	\calF^{\infty}_{\eps_j}(u_j,v_j; Q^\nu).
	\end{equation*}
	To simplify {the} notation we write
	\begin{equation}\label{eqdefUjVj}
	U_j:=(z\chi_{\{x\cdot\nu>0\}})\ast\varphi_{\eps_j}, \hskip1cm
	V_j:=\chi_{\{|x\cdot\nu|\ge2\eps_j\}}\ast\varphi_{\eps_j}.
	\end{equation}
	Obviously $\|U_j-z\chi_{\{x\cdot\nu>0\}}\|_{L^2(Q^\nu)}\to0$, 
	so that $ \|u_j-U_j\|_{L^2(Q^\nu)}\to0$. 
	Moreover, by construction $U_j=z\chi_{\{x\cdot\nu>0\}}$ if $|x\cdot\nu|\ge\eps_j$, 
	$V_j=0$ if $|x\cdot\nu|\le\eps_j$, and $V_j=1$ if $|x\cdot\nu|\ge3\eps_j$. 
	Therefore, by $\Psiinfty(0)=0$ and $f(0)=0$, we have
	\[
	\calF^{\infty}_{\eps_j}(U_j,V_j;Q^\nu)=\calF^{\infty}_{\eps_j}(0,V_j;Q^\nu)
	\leq c+\eps_j\int_{\{x\in Q^\nu:\eps_j< |x\cdot\nu|<3\eps_j\}}|\nabla V_j|^2\dx\leq c\,,
	\]
	as $\|\nabla V_j\|_{L^{\infty}(\R^m)}\leq \frac c{\eps_j}$, where $c$ is a constant 
	independent of $j\in\N$.
	
	Next, we choose a sequence $\eta_j\to0$ such that
	\begin{equation}\label{eqconvujUj}
	\frac{\eps_j+\|u_j-U_j\|_{L^2(Q^\nu)}^{2/3}}{\eta_j}\to0
	\end{equation}
	and set
	$K_j:=\lfloor \eta_j/\eps_j\rfloor$, we can assume $K_j\ge 4$.
	We let $\hat R_k^j:=Q^\nu_{1-k\eps_j}\setminus Q^\nu_{1-(k+1)\eps_j}$, where we write for brevity $Q^\nu_r:=rQ^\nu$ for the scaled cube.
	We select $k_j\in \{{K_j+1},\dots, {2K_j}\}$ such that, 
	writing $R_j:=\hat R_{k_j}^j$,
	\begin{equation}\label{eqchoice1}
	\|u_j-U_j\|_{L^2(R_j)}^2\le \frac c{K_j}\|u_j-U_j\|_{L^2(Q^\nu)}^2
	\end{equation}
	and 
	%  C{ADD explicit computation for $\calF^*_{\eps_j}(U_j,V_j,R_j)
	% $******************************************}
	\begin{equation}\label{eqchoice2}
	\calF^{\infty}_{\eps_j}(u_j,v_j;R_j)+
	\calF^{\infty}_{\eps_j}(U_j,V_j;R_j)
	\le \frac c{K_j}  .
	\end{equation}
	We fix $\theta_j\in C^1_c(Q^\nu_{1-k_j\eps_j})$ with
	$\theta_j=1$ on $Q^\nu_{1-(k_j+1)\eps_j}$
	and $|\nabla \theta_j|\le 3/\eps_j$, 
	and define
	\begin{equation*}
	u_j^*:= \theta_j u_j  + (1-\theta_j) U_j.
	\end{equation*}
	The construction of $v_j^*$ is more complex. In the interior part, it should match $v_j$. In the exterior, $V_j$. In the interpolation region, it should be not larger than $v_j$ and $V_j$, but also not larger than $1-\eta_j$. Therefore we first define
	\begin{equation}\label{eqdefhatv}
	\hat  v_j(x):= \min\{1, 1-\eta_j+\frac1{\eps_j} \dist(x, {R_j})\},
	\end{equation}
	{which coincides with $1-\eta_j$ in the interpolation region $R_j$, and with 1 at distance larger than {$\eta_j\eps_j$} from it,} then
	\begin{equation}\label{eqdefhatvv}
	\hat  V_j(x):= \min\{1, V_j(x)+\frac1{\eps_j} \dist(x, Q^\nu\setminus Q^\nu_{1-(k_j+1)\eps_j})\}
	\end{equation}
	which coincides with $V_j$ outside $Q^\nu_{1-(k_j+1)\eps_j}$, and with 1 inside $Q^\nu_{1-(k_j+3)\eps_j}$ {as well as for $|x\cdot\nu|\geq3\eps_j$ (cf. the definition of $V_j$)}, and finally
	\begin{equation}
	\tilde v_j:=\min\{1,v_j
	+\frac2{k_j\eps_j} \dist(x, Q^\nu_{1-k_j\eps_j})\}\,.
	\end{equation}
	We then combine these three ingredients to obtain
	\begin{equation*}
	v_j^*:=\min\{\tilde v_j, \hat V_j, \hat v_j\}.
	\end{equation*}
	On $\partial Q^\nu$ the first and the last term are {equal to} $1$, hence $v_j^*=\hat V_j=V_j$.
	
	{\bf Step 2. Estimate of the elastic energy.} By the definition of $u_j^*$, 
	\begin{equation*}
	|\nabla u_j^*|\le |\nabla u_j|+|\nabla U_j| + \frac{3}{\eps_j} |u_j-U_j| 
	\end{equation*}
	therefore in $R_j$
	\begin{equation*}
	{\Psiinfty}(\nabla u_j^*)\le c{\Psiinfty}(\nabla u_j)+c{\Psiinfty}(\nabla U_j) + \frac{c}{\eps_j^2} |u_j-U_j|^2 .
	\end{equation*}
	{We recall that  $v_j^*\le \min\{v_j, V_j, 1-\eta_j\}$ in $R_j$ 
		and that ${[0,1)\ni}t\mapsto t/(1-t)$ is increasing. 
		Since by construction $v_j^*=V_j=0$ on $\{\nabla U_j\ne0\}{\cap R_j}$ the term $\Psiinfty(\nabla U_j)$ can be ignored. Therefore}
	\begin{equation*}
	\frac{\eps_j(v_j^*)^2}{(1-v_j^*)^2} {\Psiinfty}(\nabla u_j^*)
	\le c\frac{\eps_j v_j^2}{(1-v_j)^2} {\Psiinfty}(\nabla u_j)
	%+c\frac{\eps_j V_j^2}{(1-V_j)^2} {\Psiinfty}(\nabla U_j)
	+ c\frac{\eps_j}{\eta_j^2} \frac{|u_j-U_j|^2}{\eps_j^2}.
	\end{equation*}
	Integrating over $R_j$ and using
	\eqref{eqchoice2} in the first term, 
	\eqref{eqchoice1} in the {second} one,
	\begin{equation*}
	\int_{R_j}\frac{\eps_j(v_j^*)^2}{(1-v_j^*)^2} {\Psiinfty}(\nabla u_j^*)\dx
	\le \frac c{K_j}+
	c\frac {\|u_j-U_j\|_{L^2(Q^\nu)}^2}{K_j\eps_j \eta_j^2}.
	\end{equation*}
	Using first that 
	the definition of $K_j$ implies $\lim_{j\to\infty} K_j\eps_j/\eta_j=1$ and then \eqref{eqconvujUj},
	\begin{equation*}
	\limsup_{j\to\infty} 
	\frac {\|u_j-U_j\|_{L^2(Q^\nu)}^2}{K_j\eps_j \eta_j^2}
	=
	\limsup_{j\to\infty}
	\frac {\|u_j-U_j\|_{L^2(Q^\nu)}^2}{\eta_j^3}
	=0.
	\end{equation*}
	Therefore
	\begin{equation*}
	\limsup_{j\to\infty} \int_{R_j}\frac{\eps_j(v_j^*)^2}{(1-v_j^*)^2} {\Psiinfty}(\nabla u_j^*)\dx
	=0.
	\end{equation*}
{Using again that the supports of {$\nabla U_j$} and $V_j$ are disjoint, we have}
	\begin{equation*}
	\int_{Q^\nu\setminus Q^\nu_{1-k_j\eps_j}}\frac{\eps_jV_j^2}{(1-V_j)^2} {\Psiinfty}(\nabla U_j)\dx
	=0. %\le c \eta_j.
	\end{equation*}
	Therefore
	\begin{equation}\label{eqfinelastcu}
	\limsup_{j\to\infty} \int_{Q^\nu}\frac{\eps_j(v_j^*)^2}{(1-v_j^*)^2} {\Psiinfty}(\nabla u_j^*)\dx
	\le \limsup_{j\to\infty} \int_{Q^\nu}\frac{\eps_jv_j^2}{(1-v_j)^2} {\Psiinfty}(\nabla u_j)\dx.
	\end{equation}
	
	{\bf Step 3. Estimate of the energy of the phase field.} By the definition of $v^*_j$,
	\begin{equation}\label{eqcalFst0}
	\calF^{\infty}_{\eps_j}(0,v_j^*; Q^\nu)
	\le 
	\calF^{\infty}_{\eps_j}(0,\tilde v_j; Q^\nu)+
	\calF^{\infty}_{\eps_j}(0,\hat V_j; Q^\nu)+
	\calF^{\infty}_{\eps_j}(0,\hat v_j; Q^\nu).
	\end{equation}
	From \eqref{eqdefhatv} we have $|1-\hat v_j|\le \eta_j$
	with $|\{\hat v_j\ne 1\}|\le c\eps_j$ and
	$|\nabla\hat v_j|\le 1/\eps_j$ 
	with $|\{\nabla\hat v_j\ne 0\}|\le c\eps_j\eta_j$, so that
	\begin{equation*}
	\calF^{\infty}_{\eps_j}(0,\hat v_j; Q^\nu)=
	\int_{Q^\nu}\Big(\frac{(1-\hat v_j)^2}{4\eps_j}+\eps_j|\nabla\hat v_j|^2\Big) \dx \le c \eta_j.
	\end{equation*}
	From the definition of $V_j$ and $\hat V_j$, 
	we see that $|\{\hat V_j\ne 1\}|\le c\eta_j\eps_j$ 
	and $\eps_j|\nabla\hat V_j|\le c$, so that
	\begin{equation*}
	\calF^{\infty}_{\eps_j}(0,\hat V_j; Q^\nu)\le c \eta_j. 
	\end{equation*}
	{Similarly, 
		$\tilde v_j=v_j$ in ${Q^\nu_{1-k_j\eps_j}}$,
		$|\tilde v_j-1|\le |v_j-1|$,
		{and $|\nabla \tilde v_j|\le |\nabla v_j|+2/(k_j\eps_j)$ in $Q^\nu \setminus Q^\nu_{1-k_j\eps_j}$}
		lead to 
		\begin{multline*}
		\calF^{\infty}_{\eps_j}(0,\tilde v_j; Q^\nu)
		\le 
		\calF^{\infty}_{\eps_j}(0,v_j; Q^\nu)
		+\frac{4\eps_j \calL^n(Q^\nu\setminus Q^\nu_{1-k_j\eps_j})}{k_j^2\eps_j^2}\\
		{+ \frac{4}{k_j\eps_j^{1/2}}} 
		{
		\calF^\infty_{\eps_j}(0,v_j;Q^\nu)^{1/2}\calL^n({Q^\nu}\setminus Q^\nu_{1-k_j\eps_j})^{{1/2}}}
		\le 
		\calF^{\infty}_{\eps_j}(0,v_j; Q^\nu)
		+{\frac{c}{k_j^{1/2}}}.
		\end{multline*}
		Recalling $k_j\ge {K_j+1}\to\infty$ and}
	$\eta_j\to0$, \eqref{eqcalFst0} leads to
	\begin{equation*}
	\limsup_{j\to\infty} 
	\calF^{\infty}_{\eps_j}(0,v_j^*; Q^\nu)
	\le 
	\limsup_{j\to\infty} 
	\calF^{\infty}_{\eps_j}(0,v_j; Q^\nu).
	\end{equation*}
	Combining this with \eqref{eqfinelastcu}  concludes the proof.
\end{proof}

{We are now ready to perform the claimed reduction on the test sequences in the definition of $g(\cdot,\nu)$ in \eqref{eqdefGsnu}. 
	%	by means of the functions $\mathcal{T}_k$ defined in \eqref{e:Tk} and a truncation argument.
	{To this aim we fix a sequence $(a_k)_k\subset(0,\infty)$ such that $a_k<a_{k+1}$, $a_k\uparrow\infty$, and such that 
		{there are functions} $\mathcal{T}_k\in {C^1_c}(\R^m;\R^m)$ satisfying
		\begin{equation}\label{e:Tk}
		\mathcal{T}_k(z):=\begin{cases}
		z, & \text{ if }|z|\leq a_k,\cr
		0, & \text{ if } |z|\geq a_{k+1} 
		\end{cases}
		\end{equation}            
		and $\|\nabla \mathcal{T}_k\|_{L^\infty(\R^m)}\leq 1$.
		Following De Giorgi's averaging/slicing procedure on the codomain, the family $\mathcal{T}_k$ will be used in several instances along the paper to obtain {from} a sequence converging in $L^1$ to a limit belonging to $L^\infty$, a sequence with the same $L^1$ limit which is in addition equi-bounded in $L^\infty$. 
		Moreover, this substitution can be done up to paying an error in energy which can be made arbitrarily small.} 
	\begin{proposition}\label{periodicity}
		For any $(z,\nu)\in \R^m\times S^{n-1}$ and
		any $\eps_j^*\downarrow0$ there is $(u_j^*,v_j^*)\to (z\chi_{\{x\cdot\nu>0\}},1)$ in $L^2(Q^\nu;\R^{m+1})$, 
		{with $v_j^*\in[0,1]$ $\calL^n$-a.e. in $\Omega$,} such that
		\begin{equation}\label{e:g periodic sequence}
		\lim_{j\to\infty} 
		\calF^{\infty}_{\eps_j^*}({{u_j^*,v_j^*}}; Q^\nu){=}  g(z,\nu)
		\end{equation}
		and
		\begin{equation*}
		u_j^*=(z\chi_{\{x\cdot\nu>0\}})\ast\varphi_{\eps_j^*}\,,\hskip1cm
		v_j^*=\chi_{\{|x\cdot\nu|\ge2\eps_j^*\}}\ast\varphi_{\eps_j^*} \hskip1cm
		\text{ on } \partial Q^\nu.
		\end{equation*}
	\end{proposition}
	\begin{proof} {{\bf Step 1. Reduction to an optimal sequence in \eqref{eqdefGsnu} converging in $L^2(Q^\nu;\R^{m+1})$.}
			Let $\eps_j\to 0$, $(u_j,v_j)\to (z\chi_{\{x\cdot\nu>0\}},1)$ in $L^1(Q^\nu;\R^{m+1})$ be such that
			\begin{equation*}
			g(z,\nu)= \lim_{j\to\infty}\calF^{\infty}_{\eps_j}(u_j,v_j; Q^\nu).
			\end{equation*}
			Recall that $v_j\in[0,1]$ $\calL^n$-a.e. in $\Omega$, therefore $v_j\to 1$ in $L^2(Q^\nu)$.
			We claim that for all $j,\,M\in\N$ there is $k_{M,j}\in\{M+1,\ldots,2M\}$ such that 
			\begin{equation}\label{e:truncation g}
			\calF^{\infty}_{\eps_j}(\mathcal{T}_{k_{M,j}}(u_j), v_j; Q^\nu)\leq\Big(1+\frac cM\Big)\calF^{\infty}_{\eps_j}(u_j, v_j; Q^\nu)\,,
			\end{equation}
			where $c>0$ is a constant independent of $M$ and $j$.
			% Given this for granted we conclude as follows. First, we may extract a subsequence of $j\in\N$, not relabeled for convenience, 
			% such that for all $M\in\N$ the index $k_{M,j}$ is actually independent of $j$. Then denote  $k_{M,j}$ simply by $k_M$.
			If ${a_M}>1+|z|=1+\|z\chi_{\{x\cdot\nu>0\}}\|_{L^\infty(Q^\nu)}$ then
			%$\mathcal{T}_{k<_M}(u_j)=z\chi_{\{x\cdot\nu>0\}}$ on $\partial Q^\nu$ and 
			${\mathcal{T}_{k_M,j}}(u_j)\to z\chi_{\{x\cdot\nu>0\}}$ in $L^2(Q^\nu;\R^m)$, and \eqref{e:truncation g} yields
			\[
			\limsup_{j\to\infty}\calF^{\infty}_{\eps_j}(\mathcal{T}_{k_M,j}(u_j), v_j; Q^\nu)\leq\Big(1+\frac cM\Big)g(z,\nu)\,,
			\]
			in turn implying by the arbitrariness of $M\in\N$
			\begin{equation*}
			g(z,\nu)=\inf \{\liminf_{j\to\infty} 
			\calF^{\infty}_{\eps_j}(u_j,v_j; Q^\nu): \|u_j- z\chi_{\{x\cdot\nu>0\}}\|_{L^2(Q^\nu)}\to0, \eps_j\to 0\}.
			\end{equation*}
			We are left with establishing \eqref{e:truncation g}. To this aim consider $\mathcal{T}_k(u_j)$ and note that 
			\begin{align}\label{e:Feps* Tk}
			\calF^{\infty}_{\eps_j}&(\mathcal{T}_k(u_j),v_j;Q^\nu)=\calF^{\infty}_{\eps_j}(u_j,v_j;\{|u_j|\leq a_k\})\notag\\
			&+\calF^{\infty}_{\eps_j}(\mathcal{T}_k(u_j),v_j;\{a_k<|u_j|< a_{k+1}\})%\notag\\&
			+\calF^{\infty}_{\eps_j}(0,v_j;\{|u_j|\geq a_{k+1}\})\,.
			\end{align}
			We estimate the {second term in \eqref{e:Feps* Tk}}. The growth conditions on $\Psi$ (cf. \eqref{e:Psi gc}) and
			$\|\nabla \mathcal{T}_k\|_{L^\infty(\R^m)}\leq 1$ yield for a constant $c>0$
			\begin{align}\label{e:Feps* Tk 2}
			\calF^{ \infty}_{\eps_j}&(\mathcal{T}_k(u_j),v_j;\{a_k<|u_j|<a_{k+1}\})\notag\\
			&\leq c\int_{\{a_k<|u_j|< a_{k+1}\}}{\eps_j}f^2(v_j)\Psiinfty(\nabla u_j)\dx
			+
			{\mathcal F^{\infty}_{\eps_j}(0,v_j;\{a_k<|u_j|< a_{k+1}\})}\,.
			\end{align}
			% and  
			% \begin{align}\label{e:Feps* Tk 3}
			%  \calF^*_{\eps_j}(0,v_j;\{|u_j|\geq a_{k+1}\})=
			%  {\mathcal F^*_{\eps_j}(0,}v_j;\{|u_j|\geq a_{k+1}\})\,.
			% \end{align}
			Collecting \eqref{e:Feps* Tk} and \eqref{e:Feps* Tk 2} 
			and using $\calF_{\eps_j}^\infty(u_j,v_j;A)+\calF_{\eps_j}^\infty(0,v_j;B)\le \calF_{\eps_j}^\infty(u_j,v_j;A\cup B)$ for $A$ and $B$ disjoint
			we conclude that 
			\begin{equation*}
			\calF^{\infty}_{\eps_j}(\mathcal{T}_k(u_j),v_j;Q^\nu)\leq \calF^{\infty}_{\eps_j}(u_j,v_j;Q^\nu)
			+ c\int_{\{a_k<|u_j|< a_{k+1}\}}{\eps_j}f^2(v_j)\Psiinfty(\nabla u_j)\dx\,.
			\end{equation*}
			Let now $M\in\N$, by averaging there exists $k_{M,j}\in\{M+1,\ldots,2M\}$ such that 
			\begin{align*}
			\calF^{\infty}_{\eps_j}(\mathcal{T}_{k_{M,j}}(u_j),v_j;Q^\nu)&\leq\frac 1M\sum_{k=M+1}^{2M}\calF^{\infty}_{\eps_j}(\mathcal{T}_k(u_j),v_j;Q^\nu)\\
			&\leq\Big(1+\frac cM\Big)\calF^{\infty}_{\eps_j}(u_j,v_j;Q^\nu)\,,
			\end{align*}
			i.e. \eqref{e:truncation g}.
			\medskip
			
			{\bf Step 2. Conclusion.}
			%Estimate \eqref{e:g periodic sequence} holds assuming the optimal sequence in \eqref{eqdefGsnu} converges in $L^2$.}}
			In view of Step 1 there is an optimal sequence for $g(z,\nu)$ in \eqref{eqdefGsnu} converging in 
			$L^2(Q^\nu;\R^{m+1})$.} Let $(\eps_k, u_k, v_k)$ be the sequence from Proposition~\ref{proplbboundary}. 
{
Since $\lim_{k\to\infty}\lim_{j\to0} \eps_j^*/\eps_k=0$, we can
 select a nondecreasing sequence $k(j)\to\infty$}
% 			For every $j$ there is $k(j)\ge j$ 
such that {$\lambda_j:=\eps_j^*/\eps_{k(j)}\to0$.}
		We 
let $\tilde Q^\nu:=(\Id-\nu\otimes \nu)Q^\nu\subset \nu^\perp\subset\R^n$ {and}
		select $x_1, \dots, x_{I_j}\in \tilde Q^\nu$, with $I_j:=
		\lfloor1/\lambda_j\rfloor^{n-1}$,
		such that $x_i+\tilde Q^\nu_{\lambda_j}$ are pairwise disjoint subsets of $\tilde Q^\nu$. We set
		\begin{equation*}
		u_j^*(x):=\begin{cases}
		u_{k(j)}(\frac{x-x_i}{{\lambda_j}}), & \text{ if } x-x_i\in Q^\nu_{{\lambda_j}} \text{ for some $i$},\\
		{U_j^*}(x), & \text{ otherwise in $Q^\nu$}
		\end{cases}
		\end{equation*}
		and
		\begin{equation*}
		v_j^*(x):=\begin{cases}
		v_{k(j)}(\frac{x-x_i}{{\lambda_j}}), & \text{ if } x-x_i\in Q^\nu_{{\lambda_j}} \text{ for some $i$},\\
		{V_j^*}(x),  & \text{ otherwise in $Q^\nu$},
		\end{cases}
		\end{equation*}
		where {$U_j^*$ and $V_j^*$} are defined as in \eqref{eqdefUjVj}
{using $\eps_j^*$. One easily verifies that
$U_j^*(x)=U_{k(j)}(\frac{x-y}{\lambda_j})$ for all $y\in \nu^\perp$, and the same for $V$.}
		By the boundary conditions \eqref{eqpropboundrybv}, these functions are continuous and therefore in $W^{1,2}(Q^\nu;\R^{m+1})$. We further estimate
		\begin{equation*}
		\calF^{\infty}_{\eps_j^*}(u_j^*, v_j^*; Q^\nu)\le
		I_j \lambda_j^{n-1} \calF^{\infty}_{\eps_{k(j)}}(u_{k(j)}, v_{k(j)}; Q^\nu)
		+ c \calH^{n-1}({\tilde  Q}^\nu\setminus \cup_i (x_i+{\tilde  Q} ^\nu_{\lambda_j}){)}.
		\end{equation*}
		Taking $j\to\infty$, and recalling that 
		$\limsup_{j} \calF^{\infty}_{\eps_{k(j)}}(u_{k(j)}, v_{k(j)}; Q^\nu)\le g(z,\nu)$, concludes the proof.
	\end{proof}
	
	{In what follows we provide an equivalent characterization for the surface energy $g$ 
		{in the spirit of \cite[Proposition~4.3]{ContiFocardiIurlano2016}}.
		\begin{proposition}\label{p:chargT}
			For any $(z,\nu)\in \R^m\times S^{n-1}$ one has
			\begin{equation}\label{e:characg}
			g(z,\nu)={\lim_{T\to\infty}}\inf_{(u,v)\in \calU_{z,\nu}^T}\frac{1}{T^{n-1}}{\calF^{\infty}_{1}(u,v;Q^\nu_T)}\,,
			%\int_{Q^\nu_T} \Big(\frac{ {\ell^2} v^2}{(1-v)^2} \Psiinfty(\nabla u) + \frac{(1-v)^2}{4}+|\nabla v|^2\Big)\dx,
			\end{equation}
			where
			\begin{multline*}\calU^T_{z,\nu}:=\Big\{(u,v)\in W^{1,2}(Q^\nu_T;\R^{m+1})\colon 0\leq v\leq 1,\  
			v=\chi_{\{|x\cdot\nu|\ge 2\}}\ast\varphi_{1}\, \text{ and }\\
			u=(z\chi_{\{x\cdot\nu>0\}})\ast\varphi_{1}
			\text{ on } \partial Q^\nu_{{T}} \Big\}.
			\end{multline*}
		\end{proposition}
		\begin{proof}
			For every $(z,\nu)\in \R^m\times S^{n-1}$ and $T>0$ set 
				\[
				g_T(z,\nu):=\inf_{(u,v)\in \calU_{z,\nu}^T}\frac{1}{T^{n-1}}\calF^{\infty}_{1}(u,v;Q^\nu_T)\,.
				% \int_{Q^\nu_T} \Big(\frac{ {\ell^2} v^2}{(1-v)^2} \Psiinfty(\nabla u) + \frac{(1-v)^2}{4}+|\nabla v|^2 \Big)\dx\,.
				\]
				We first prove that
				\begin{equation}\label{e:limsup leq g}
				\limsup_{T\to\infty}g_T(z,\nu)\leq g(z,\nu)\,.
				\end{equation}
				Indeed, if $T_j\uparrow\infty$ is a sequence achieving the superior limit on the left{-}hand side above, 
				thanks to Proposition~\ref{periodicity} we may
				%set $\eps_j:=\frac1{T_j}$ and 
				consider %a sequence $(u_j,v_j)$ as in : that is 
				$(u_j,v_j)\in W^{1,2}(Q^\nu;\R^{m+1})$ with $0\leq v_j\leq 1$, $(u_j,v_j)\to (z\chi_{\{x\cdot\nu>0\}},1)$ 
				in $L^2(Q^\nu;\R^{m+1})$, 
				\begin{equation}\label{e:bound}
				u_j=(z\chi_{\{x\cdot\nu>0\}})\ast\varphi_{\frac1{T_j}}\,,\qquad
				v_j=\chi_{\{|x\cdot\nu|\ge \frac2{T_j}\}}\ast\varphi_{\frac1{T_j}}
				\text{ on } \partial Q^\nu,
				\end{equation}
				and
				\begin{equation}\label{e:estf}
				\lim_{j\to\infty}\calF^{\infty}_{\frac1{T_j}}(u_j,v_j; Q^\nu)=g(z,\nu).
				\end{equation}
				Then, define $(\tilde{u}_j(y),\tilde{v}_j(y)):=\big(u_j(\frac y{T_j}),v_j(\frac y{T_j})\big)$ for $y\in Q^\nu_{T_j}$, 
				and note that by a change of variable it is true that
				\begin{equation*}
				\frac{1}{T_j^{n-1}}\calF^{\infty}_{1}(\tilde{u}_j,\tilde{v}_j;Q^\nu_{T_j})=
				\calF^{\infty}_{\frac1{T_j}}(u_j,v_j;Q^\nu)\,,
				\end{equation*}
				and that $(\tilde{u}_j,\tilde{v}_j)\in\calU^{T_j}_{z,\nu}$ 
				in view of \eqref{e:bound}. Then, by \eqref{e:estf}, the choice of 
				$T_j$ and the definition of $g_T(z,\nu)$ we conclude straightforwardly \eqref{e:limsup leq g}.
				
				In order to prove the converse inequality
				\begin{equation}\label{e:liminf geq g}
				\liminf_{T\to\infty}g_T(z,\nu)\geq g(z,\nu)\,,
				\end{equation}
				we assume for the sake of notational simplicity $\nu=e_n$.
				We then fix $\rho>0$ and take $T>6$, depending on $\rho$, 
				and $(u_T,v_T)\in \calU^T_{z,e_n}$ such that
				\begin{equation}\label{e:estf2}
				\frac{1}{T^{n-1}}\calF^{\infty}_{1}(u_T,v_T;Q^{e_n}_T)\leq
				%\tilde g(z,\nu)+{\rho}.
				\liminf_{T\to\infty}g_T(z,{e_n})+\rho\,.
				\end{equation}
			Let $\eps_j\to0$ and set
			\[
			u_j(y):=\left\{ \begin{array}{ll} 
			{\displaystyle u_T\left(\frac{y}{\eps_{j}}-d\right)}, 
			& \textrm{if } y\in %\eps_j (Q^{\nu}_T+d), \text{ for some }  
			\eps_j (Q^{e_n}_T+d)\subset\subset Q^{e_n},\\
			{{(z\chi_{\{x\cdot {e_n} >0\}}\ast \varphi_{1})(\frac{y}{\eps_j})}}, & \text{otherwise in }Q^{e_n},
			\end{array} \right.
			\]
			\[
			v_j(y):=\left\{ \begin{array}{ll} 
			{\displaystyle v_T\left(\frac{y}{\eps_{j}}-d\right)} ,
			& \textrm{if } y\in %\eps_j (Q^{\nu}_T+d), \text{ for some } 
			\eps_j (Q^{e_n}_T+d)\subset\subset Q^{e_n},\\
			{{(\chi_{\{|x\cdot {e_n}| >2\}}\ast \varphi_{1})(\frac{y}{\eps_j})}}, & \text{otherwise in }Q^{e_n},
			\end{array} \right.
			\]
			with $d\in\Z^{n-1}\times\{0\}$.
			Then, $(u_j,v_j)\to(z\chi_{\{x\cdot{e_n} >0\}},1)$ in {$L^1(Q^{e_n};\R^{m+1})$}, and {letting $I_{\eps_j}:=\{d\in \Z^{n-1}\times\{0\}:\,\eps_j (Q^{{e_n}}_T+d)\subset\subset Q^{e_n}\}$}, a change of variable yields 
			(cf. also the discussion after \eqref{eqdefUjVj})
			% \MFF{Continuo ad avere dei dubbi: scriverei 
			% $(z\chi_{\{x\cdot{e_n}>0\}}\ast\varphi_{1})(\frac y{\eps_j})$ per $u_T$ e $(\chi_{\{|x\cdot{e_n}|\ge 2\}}\ast\varphi_{1})(\frac y{\eps_j})$ per $v_T$ perch\'e trovo pi\'u chiaro che $u_j,\,v_j$ siano Sobolev, anche se le funzioni si corrispondono con un cambio di variabile. Inoltre, la stima del termine Modica-Mortola non \`e proprio quella di  \eqref{eqdefUjVj}: alla fine dovrebbe esserci un errore di ordine $\eps_j\calH^{n-2}(\partial'(Q^{e_n}\cap\{x_n=0\}))$, il bordo preso nella topologia dell'iperpiano. Mi pare che manchi $\eps_j$ a dividere nel termine al terzo rigo della formula successiva.}
			\begin{align*}
			g(z,{e_n})\leq&\limsup_{j\to\infty}\calF^\infty_{\eps_j}(u_j,v_j;Q^{{e_n}})\\
			\leq&\limsup_{j\to\infty}\Big(\sum_{d\in I_{\eps_j}}
			% {\scriptsize 
			% 		\begin{array}{c}
			% 		d\!\in\! \Z^{n-1}\!\times\!\{0\}\\ 
			% 		{\eps_j (Q^{\nu}_T\!+\!d)}\!\subset\!\subset\! Q^{\nu}
			% 		\end{array}}}
			\calF_{\eps_j}^\infty(u_j,v_j;\eps_j(Q^{{e_n}}_T+d))\\
			&+{\frac{c}{\eps_j}}\calL^{n}\Big({Q^{{e_n}}\cap}\{{\eps_j\leq}|{x_n}|\leq 3\eps_j\}\setminus \!\!\bigcup_{d\in I_{{\eps_j}}}
			% {\scriptsize 
			% 		\begin{array}{c}
			% 		d\!\in\! \Z^{n-1}\!\times\!\{0\}\\ 
			% 		{\eps_j (Q^{\nu}_T\!+\!d)}\!\subset\!\subset\! Q^{\nu}
			% 		\end{array}}}\!\!
			\eps_j (Q^{{e_n}}_T+d)\Big)\Big)\\
			=& \limsup_{j\to\infty}
			%\sum_{{\scriptsize 
			%\begin{array}{c}
			%d\!\in\! \Z^{n-1}\!\times\!\{0\}\\ 
			%{\eps_j (Q^{\nu}_T\!+\!d)}\!\subset\!\subset\! Q^{\nu}
			%\end{array}}}
			%\int_{ Q^{\nu}_T}\eps_j^{n-1} \left( f^2(v_T) \Psiinfty\left(\nabla u_T\right) + \frac{(1-v_T)^2}{4} + |\nabla v_T|^2 \right) \dx\\
			\eps_j^{n-1}
			% \#\{d\in\Z^{n-1}\!\times\!\{0\}:\,{\eps_j (Q^{\nu}_T\!+\!d)}\!\subset\!\subset\! Q^{\nu}\}
			%\end{array}}}
			%		\}
			\#I_{\eps_j}\,\calF_{1}^\infty(u_T,v_T;Q^{{e_n}}_T) \\
			%=\limsup_{j\to\infty}\frac{1}{\eps_j^{n-1}T^{n-1}}
			%\int_{ Q^{\nu}_T}\eps_j^{n-1} \left( f^2(v_T) \Psiinfty\left(\nabla u_T\right) + \frac{(1-v_T)^2}{4} + |\nabla v_T|^2 \right) \dx\\
			\leq&\frac1{T^{n-1}}\calF_{1}^\infty(u_T,v_T;Q^{{e_n}}_T)\leq%\tilde g(z,\nu)
			\liminf_{T\to\infty}g_T(z,{e_n})+\rho\,,	
			\end{align*}
			by the choice of $(u_T,v_T)$ and $T$ {(cf. \eqref{e:estf2})}. 
			As $\rho\to 0$ we get {\eqref{e:liminf geq g}}.
			
			Estimates \eqref{e:limsup leq g} and \eqref{e:liminf geq g} yield the existence of the limit of $g_T(z,\nu)$ as $T\uparrow\infty$ and equality \eqref{e:characg}, as well.
		\end{proof}

With this representation {of $g$} at hand we can obtain a version of Proposition~\ref{periodicity} which also accounts for {a} regularization term 
{of the form} $\eta_\eps\int \Psi(\nabla u)\dx$.
		\begin{proposition}\label{propuvjump}
			For any $\eps_j\downarrow0$ and $\eta_j\downarrow0$
			with $\eta_j/\eps_j\to0$, and any $(z,\nu)\in \R^m\times S^{n-1}$
			there is $(u_j,v_j)\to (z\chi_{\{x\cdot\nu>0\}},1)$ in 
			$L^2(Q^\nu;\R^{m+1})$, with $v_j\in[0,1]$ $\calL^n$-a.e. 
			in $Q^\nu$, such that
			\begin{equation*}%\label{e:g periodic sequence}
			\lim_{j\to\infty} 
			\calF^{\infty}_{\eps_j}({{u_j,v_j}}; Q^\nu) = g(z,\nu)\,
			\end{equation*}
			\[
			\lim_{j\to\infty} \eta_j \int_{Q^\nu} |\nabla u_j|^2\dx=0\,,
			\]
			and
			\begin{equation*}
			u_j=(z\chi_{\{x\cdot\nu>0\}})\ast\varphi_{\eps_j}\,,\hskip1cm
			v_j=\chi_{\{|x\cdot\nu|\ge2\eps_j\}}\ast\varphi_{\eps_j} \hskip1cm
			\text{ on } \partial Q^\nu.
			\end{equation*}
		\end{proposition}
		\begin{proof}
			We use the same construction as above {({without} loss of generality, explicitly written only for $\nu=e_n$),} 
			and compute similarly
			\begin{align*}
			\|\nabla u_j\|_{L^2(Q^{{e_n}})}^2
			&\leq\sum_{d\in I_{\eps_j}}
			\|\nabla u_j\|_{L^2(\eps_j(Q^{{e_n}}_T+d))}^2
			+\frac{c}{\eps_j^2}\calL^{n}\left(Q^{{e_n}}\cap\{|{x_n}|\leq \eps_j\}%\setminus \!\!\bigcup_{d\in I_\eps}\eps_j (Q^{\nu}_T+d)
			\right)\\
			&= 
			\eps_j^{n-1}\# I_{\eps_j}\,\|\nabla u_T\|_{L^2(Q_T)}^2
			+\frac{c}{\eps_j}\le \frac{C_T}{\eps_j}.
			\end{align*}
			To conclude the proof it suffices to choose $T_j\to\infty$ so slow that 
			$\eta_j C_{T_j}/\eps_j\to0$.
		\end{proof}

		\begin{corollary}\label{euclidean}
			If $\Psi{{_\infty}}(\xi)=|\xi|^2$, then $g(z,\nu)=g_\scal(|z|)$ 
			for all $(z,\nu)\in \R^m\times S^{n-1}$, {where 
			$g_\scal$ is defined as the right-hand side of equation \eqref{e:characg} with $n=m=1$.} 
		\end{corollary}
	{For an equivalent definition of $g_\scal$ see equation \eqref{e:gscal} below and \cite[Proposition 4.3]{ContiFocardiIurlano2016}.}
		\begin{proof}
			By \cite[Proposition~4.3]{ContiFocardiIurlano2016} {or by Proposition~\ref{p:chargT}}, the following characterization holds for $g_\scal$:
			$$g_\scal(s)=\lim_{T\uparrow\infty }\inf_{(\alpha,\beta)\in\calU^T_s}\mathscr{F}^\infty_1(\alpha,\beta;(-\sfrac T2,\sfrac T2)),$$	
			with
			$$\mathscr{F}^\infty_1(\alpha,\beta;(-\sfrac T2,\sfrac T2)):=\int_{-\frac T2}^{\frac T2}(f^2(\beta)|\alpha'|^2+\frac{(1-\beta)^2}{4}+|\beta'|^2)\dx$$
			and
			\begin{multline*}
			\calU^T_s:=\{\alpha,\beta\in {W^{1,2}}((-\sfrac T2,\sfrac T2))\colon  
			0\leq \beta\leq 1,\ \beta(\pm\sfrac T2)=1\\
			\alpha(-\sfrac T2)=0\,,\ \alpha(\sfrac T2)=s\}.
			\end{multline*}
			Let $(z,\nu)\in \R^m\times S^{n-1}$, $z\ne0$. We first prove that
			\begin{equation}\label{ggscal}
			g(z,\nu)\geq g_\scal(|z|).
			\end{equation}
			If $T>0$ and $(u,v)\in \calU^T_{z,\nu}$ (see Proposition~\ref{p:chargT} for the definition of $ \calU^T_{z,\nu}$), then for $\calH^{n-1}$-a.e. $y\in \tilde Q^\nu_T:=(\Id-\nu\otimes \nu)Q^\nu_T\subset \nu^\perp$ the slices
			$$u^\nu_y(t):=\frac{z}{|z|}\cdot u(y+t\nu),\quad v^\nu_y(t):=v(y+t\nu)$$
			belong to $\calU^T_{|z|}$ and satisfy by Fubini's theorem
			\begin{multline*}
			\frac{1}{T^{n-1}}\calF^\infty_1(u,v;Q^\nu_T)\geq \frac{1}{T^{n-1}}\int_{\tilde Q^\nu_T}\mathscr{F}^\infty_1(u^\nu_y,v^\nu_y;(-\sfrac T2,\sfrac T2))\, \mathrm{d}\calH^{n-1}(y)\\
			\geq \inf_{(\alpha,\beta)\in\calU^T_{|z|}}\mathscr{F}^\infty_1(\alpha,\beta;(-\sfrac T2,\sfrac T2)).	
			\end{multline*}
			Taking the infimum over $(u,v)\in\calU^T_{z,\nu}$ and passing to the limit $T\to\infty$ we get \eqref{ggscal}.	
			
			Let us show now that
			\begin{equation}\label{gscalg}
			g(z,\nu)\leq g_\scal(|z|).
			\end{equation}
			Let $T>0$ and $(\alpha,\beta)\in \calU^{T}_{|z|}$. Fixed $\eps_j\to0$, we will construct a competitor $(u_j,v_j)$ for the problem \eqref{eqdefGsnu} defining $g$. We set
			\[
			u_j(x):=\left\{ \begin{array}{ll} 
			{\displaystyle \alpha\Big(\frac{T}{\eps_j}x\cdot\nu\Big)\frac{z}{|z|}}, 
			& \textrm{if } |x\cdot \nu|\leq \frac{\eps_j}{2},\ x\in Q^\nu,\\
			z\chi_{\{x\cdot \nu >0\}}, & \text{otherwise in }Q^\nu,
			\end{array} \right.
			\]
			\[
			v_j(x):=\left\{ \begin{array}{ll} 
			{\displaystyle \beta\Big(\frac{T}{\eps_j}x\cdot\nu\Big)}, 
			& \textrm{if } |x\cdot \nu|\leq \frac{\eps_j}{2},\ x\in Q^\nu,\\
			1, & \text{otherwise in }Q^\nu.
			\end{array} \right.
			\]
			Hence by a change of variables we have $\|u_j-z\chi_{\{x\cdot\nu>0\}}\|_{L^1(Q^\nu)}\to0$ and 
			\[
			\calF^\infty_{\sfrac{\eps_j}{T}}(u_j,v_j;Q^\nu)= \mathscr{F}^\infty_1(\alpha,\beta;(-\sfrac T2,\sfrac T2))\,.
			\]
			Therefore, we conclude that
			\[
			g(z,\nu)\leq\mathscr{F}^\infty_1(\alpha,\beta;(-\sfrac T2,\sfrac T2))\,.
			%\liminf_{j\to\infty}\calF^\infty_{\sfrac{\eps_j}{T}}(u_j,v_j,Q^\nu).	
			\]
			As $(\alpha,\beta)\in \calU^T_{|z|}$ varies, we obtain \eqref{gscalg}.
		\end{proof}
		
		\begin{remark}\label{sliceable}
			The same argument shows that if $\Psi$ satisfies 
			$\Psi{_\infty}(\xi)\geq\Psi{_\infty}(\xi\nu\otimes\nu)$ for every 
			$\xi\in\R^{m\times n}$ and $\nu\in S^{n-1}$, then for all 
			$(z,\nu)\in \R^m\times S^{n-1}$
			\begin{equation*}%\label{e:Psi sliceable}
			g(z,\nu)=\lim_{T\uparrow\infty}
			\inf_{\widetilde{\calU}_z^T}
			\int_{-\sfrac T2}^{\sfrac T2}\Big(f^2(\beta(t))\Psi_\infty\big({\alpha'(t)}\otimes\nu\big)+\frac{(1-\beta(t))^2}{4}+|\beta'(t)|^2\Big)\dd t
			\end{equation*}
			where
			\begin{multline*}
			\widetilde{\calU}^T_z:=\{(\alpha,\beta)\in {W^{1,2}}((-\sfrac T2,\sfrac T2);\R^{m+1})\colon  0\leq \beta\leq 1,\, \beta(\pm\sfrac T2)=1\\ 
			\alpha(-\sfrac T2)=0,\,\alpha(\sfrac T2)=z\}.
			\end{multline*}
		\end{remark}
		\subsection{Structural properties of $g(z,\nu)$}\label{subsecstructpropg}
		We {next deduce} the coercivity properties of $g$.
		\begin{lemma}\label{lemmapropg}
			There is $c>0$ such that, for all $z,\nu\in \R^m\times S^{n-1}$,
			\begin{equation*}
			\frac1c (|z|\wedge 1) \le g(z,\nu) \le c (|z|\wedge 1).
			\end{equation*}
		\end{lemma}
		{We provide here a direct proof of the lemma. Alternatively, these bounds may be derived estimating $\calF_\eps$ by its 1D counterpart (as in \eqref{eqfepsfeps1} below) and recalling the bounds holding for $g_\scal$, see \cite[Prop.~4.1]{ContiFocardiIurlano2016}.}
		%These bounds {could be derived} from {the fact that $\calF_\eps$ can be estimated by its 1D counterpart (see \eqref{eqfepsfeps1}) and from} \cite[Prop.~4.1]{ContiFocardiIurlano2016}. For completeness we provide a self-contained proof.
		
		\begin{proof}{
				We start with the lower bound.
				Let $z\in\R^m$, $\nu\in S^{n-1}$, and fix sequences $\eps_j\to0$, $v_j$ and $u_j\to z\chi_{{\{x\cdot\nu>0\}}}$ {in $L^1(Q^\nu;\R^m)$} such that $\calF^{\infty}_{\eps_j}(u_j,v_j; Q^\nu)\to g(z,\nu)$. {For every $j$ and $y_j\in \nu^\perp\cap Q^\nu$ we define}
				$v_j^*\in W^{1,2}((-\frac12,\frac12);[0,1])$
				and $u_j^*\in W^{1,2}((-\frac12,\frac12);\R^m)$ by
				$v_j^*(t):=v_j(y_j+t\nu)$ and
				$u_j^*(t):=u_j(y_j+t\nu)$. {The set of $y_j\in \nu^\perp\cap Q^\nu$ such that}
				\begin{equation*}
				\|u_j^*-z\chi_{{\{t\ge0\}}}\|_{L^1((-\frac12,\frac12))}\le 3
				\|u_j-z\chi_{{\{x\cdot\nu\ge0\}}}\|_{L^1(Q^\nu)} 
				\end{equation*} 
				{has measure at least} $\frac23$ and, using \eqref{e:Psi gc} to estimate 
				$\frac1c|(u_j^*)'|^2(t) \le \Psiinfty(\nabla u_j)(y_j+t\nu)$, 
				{the set of $y_j\in \nu^\perp\cap Q^\nu$ such that}
				\begin{equation*}
				\int_{(-\frac12,\frac12)} \Big(\frac{\eps_j {\ell^2} (v_j^*)^2}{(1-v_j^*)^2} \frac{|(u_j^*)'|^2}c + \frac{(1-v_j^*)^2}{4\eps_j}+\eps_j|( v_j^*)'|^2 
				\Big)\dt \le 3\calF_{\eps_j}^{\infty}(u_j,v_j;Q^\nu)
				\end{equation*}
				{also has measure at least $\frac23$. Therefore we can fix $y_j$ such that both inequalities hold.}
				If $g(z,\nu)<\infty$, then necessarily $v_j^*\to1$ in $L^2((-\frac12,\frac12))$, and it has a continuous representative. We can therefore assume that $\sup v_j^*\ge \frac34$ for large $j$.
				If $\inf v_j^*\le \frac12$ then
				\begin{equation*}
				\begin{split}
				\left. \frac12 (1-v)^2\right|_{1/2}^{3/4}\le &\int_{(-\frac12,\frac12)} |(1-v_j^*)(v_j^*)'| \dt\\
				\le& 
				\int_{(-\frac12,\frac12)}  \frac{(1-v_j^*)^2}{4\eps_j}+\eps_j|( v_j^*)'|^2 
				\dt\le {3}\calF_{\eps_j}^{\infty}(u_j,v_j;Q^\nu).
				\end{split}
				\end{equation*}}
			{Otherwise, $v_j^*\ge\frac12$ pointwise and
				\begin{equation*}
				\begin{split}
				\int_{(-\frac12,\frac12)} \Big(\frac{\eps_j {\ell^2} (v_j^*)^2}{(1-v_j^*)^2} \frac{|(u_j^*)'|^2}c + \frac{(1-v_j^*)^2}{4\eps_j}
				\Big)\dt \ge & \frac1{2c^{1/2}} \ell \int_{(-\frac12,\frac12)} |(u_j^*)'| \dt.
				%  \ge& \frac1{2c^{1/2}} \ell |z|-c \|u_j^*-z\chi_{t\ge0}\|_{L^1((-\frac12,\frac12))}. 
				\end{split}
				\end{equation*}
				{Since $ \|u_j^*-z\chi_{t\ge0}\|_{L^1((-\frac12,\frac12))}\to0$, 
					there are $t_j, t'_j$ such that
					$u_j^*(t_j)\to0$, $u_j^*(t_j')\to z$, 
					and therefore
					$\liminf_{j\to\infty} \int_{(-\frac12,\frac12)} |(u_j^*)'| \dt
					\ge \liminf_{j\to\infty} |u_j^*(t_j)-u_j^*(t_j')|=|z|$.}
				We conclude that $\liminf_{j\to\infty}\calF_{\eps_j}^{\infty}({u_j,v_j};Q^\nu)\ge c (1\wedge \ell |z|)$.
			}
			
			{We turn to the upper bound.
				We define
				$u_j(x):=u_j^*(x\cdot\nu)$, 
				$v_j(x):=v_j^*(x\cdot\nu)$,  where, denoting by $AI$ the affine interpolation between 
				the boundary data in the relevant segments,
				\begin{equation*}
				u_j^*(t):=\begin{cases}
				0, & \text{ if } t\le-\eps_j,\\
				z, & \text{ if } t\ge \eps_j,\\
				AI, %\frac{t+\eps_j}{2\eps_j} z 
				& \text{ if } -\eps_j<t<\eps_j,
				\end{cases}
				\hskip2mm
				v_j^*(t):=\begin{cases}
				(1-(\ell{|z|})^{1/2})_+, & \text{ if } |t|\le \eps_j,\\
				1, & \text{ if } |t|\ge 2\eps_j,\\
				AI, & \text{ if } |t|\in (\eps_j,2\eps_j).
				\end{cases}
				\end{equation*}
				If $\ell {|z|}<1$, then the upper bound in \eqref{e:Psi gc} leads to
				\begin{equation*}
				\calF_{\eps_j}^{\infty}(u_j,v_j;Q^\nu)
				\le 2\eps_j \frac{\eps_j\ell^2 c ({|z|}/2\eps_j)^{2}}{\ell {|z|}} + 4\eps_j \frac{\ell {|z|}}{4\eps_j}
				+ 2\eps_j \eps_j \frac{\ell {|z|}}{\eps_j^2}
				=(\frac12c+1+2) \ell {|z|}.
				\end{equation*}
				If instead $\ell {|z|}\ge 1$ the first term vanishes, and
				\begin{equation*}
				\calF_{\eps_j}^{\infty}(u_j,v_j;Q^\nu)
				\le 0+4\eps_j \frac{1}{4\eps_j}
				+ 2\eps_j \eps_j \frac{1}{\eps_j^2}
				=3.\qedhere
				\end{equation*}
			}
		\end{proof}
		
		We prove next the subadditivity and continuity of $g$.
		
		{\begin{lemma}\label{lemmapropg2}
				\begin{enumerate}
					\item\label{lemmapropgsubadd}
					For any $\nu\in S^{n-1}$ and $z^1,z^2\in\R^m$ one has 
					\[
					g(z^1+z^2,\nu)\le g(z^1,\nu)+g(z^2,\nu).
					\]
					\item\label{lemmapropgcont}
					$g\in C^0(\R^m\times S^{n-1})$.
				\end{enumerate}
			\end{lemma}
		}
		\begin{proof}
			{\ref{lemmapropgsubadd}: 
				Fix $z^1,z^2\in\R^m$, $\nu\in S^{n-1}$.
				Let $(u_j^i, v_j^i)$ be the sequences from Proposition~\ref{periodicity} corresponding to $\eps_j:=1/j$ and the pair 
				$(\nu,z^i)$, for $i=1,2$. We implicitly extend both periodically in the directions of $\nu^\perp\cap Q^\nu$, 
				and constant in the direction $\nu$. In particular, for $\{x\cdot\nu\ge \frac12\}$ we have $u_j^i=z^i$ and 
				$v_j^i=1$; for $\{x\cdot\nu\le-\frac12\}$ we have $u_j^i=0$ and $v_j^i=1$ for $i\in\{1,2\}$ and all $j$.
				
				We use a rescaling similar to the one of Proposition~\ref{periodicity}.
				We fix a sequence $M_j\in\N$, $M_j\to\infty$,  and define $(u_j,v_j)\in W^{1,2}(\R^n;\R^m\times[0,1])$ by 
				\begin{equation*}
				u_j(x):=\begin{cases}
				u_j^1(M_jx+\frac12\nu), & \text{ if } x\cdot\nu<0,\\
				z^1+u_j^2(M_jx-\frac12\nu), & \text{ if } x\cdot\nu\ge0,
				\end{cases}
				\end{equation*}
				and, correspondingly,
				\begin{equation*}
				v_j(x):=\begin{cases}
				v_j^1(M_jx+\frac12\nu), & \text{ if } x\cdot\nu<0,\\
				v_j^2(M_jx-\frac12\nu), & \text{ if } x\cdot\nu\ge0.
				\end{cases}
				\end{equation*}
				By the periodicity of $(u^i_j,v^i_j)$  in the directions of $\nu^\perp\cap Q^\nu$, 
				these maps belong to $W^{1,2}(Q^\nu;\R^m)$. Furthermore, $u_j=0$ and $v_j=1$ if $x\cdot\nu\leq-\frac1{M_j}$,
				$u_j=z^1+z^2$ and $v_j=1$ if $x\cdot\nu\geq\frac1{M_j}$, and $(u_j,v_j)$ is 
				$\frac1{M_j}$-periodic in the directions of $\nu^\perp\cap Q^\nu$.
				Therefore, by changing variables we find
				\begin{equation*}
				\begin{split}
				& \|u_j- (z^1+z^2)\chi_{\{x\cdot\nu\ge0\}}\|_{L^1(Q^\nu)}
				=\|u_j- (z^1+z^2)\chi_{\{x\cdot\nu\ge0\}}\|_{L^1(Q^\nu\cap\{|x\cdot\nu|\leq\frac1{M_j}\})}\\
				&=\frac1{M_j^n}\|u_j^1\|_{L^1(M_jQ^\nu\cap\{|x\cdot\nu|\leq\frac12\})}
				+\frac1{M_j^n}\|u_j^2-z^2\|_{L^1(M_jQ^\nu\cap\{|x\cdot\nu|\leq\frac12\})}\\
				&=\frac1{M_j}\|u_j^1\|_{L^1(Q^\nu)}
				+\frac1{M_j}\|u_j^2-z^2\|_{L^1(Q^\nu)}\\
				&\le \frac1{M_j} 
				\|{u_j^1}- z^1\chi_{\{x\cdot\nu\ge0\}}\|_{L^1(Q^\nu)}+
				\frac{|z^1|}{2M_j} +\frac{1}{M_j}
				\|u_j^2- z^2\chi_{\{x\cdot\nu\ge0\}}\|_{L^1(Q^\nu)}
				+\frac{|z^2|}{2M_j},
				\end{split}
				\end{equation*}
				so that $u_j\to(z^1+z^2)\chi_{\{x\cdot\nu\ge0\}}$ in $L^1(Q^\nu;\R^m)$. Arguing similarly, we infer
				\begin{equation*}
				\begin{split}
				\calF^{\infty}_{\eps_j/M_j} (u_j, v_j; Q^\nu)=
				\calF^{\infty}_{\eps_j} (u^1_j, v_j^1; Q^\nu)+
				\calF^{\infty}_{\eps_j} (u^2_j, v_j^2; Q^\nu).
				\end{split}
				\end{equation*}
				The conclusion follows taking the limit $j\to\infty$.
			}
			
			{\ref{lemmapropgcont}: 
				By \ref{lemmapropgsubadd} and Lemma~\ref{lemmapropg}   we have $g(z,\nu)\le g(z',\nu)+c\ell|z-z'|$, which implies that for any $\nu\in S^{n-1}$ the function $g(\cdot,\nu)$ is $c\ell$-Lipschitz continuous. Therefore it suffices to prove continuity in $\nu$ at any fixed $z$.
				
				Since $\Psiinfty$ is continuous and positive on the compact set $S^{nm-1}\subseteq\R^{m\times n}$, there is a monotone modulus of continuity
				$\omega:[0,\infty)\to[0,\infty)$, with $\omega_\rho\to0$ as $\rho\to0$, such that
				\begin{equation*}
				\Psiinfty(\xi)\le (1+\omega_{|\xi-\xi'|})\Psiinfty(\xi') \text{ for } |\xi|=|\xi'|=1.
				\end{equation*}
				This implies that
				\begin{equation}\label{eqpsiinfunifcont}
				\Psiinfty(\eta)\le (1+\omega_{|R-\Id|})\Psiinfty(\eta R) \text{ for any } \eta\in\R^{m\times n}, R\in O(n)
				\end{equation}
				(it suffices to insert $\eta/|\eta|$ and $\eta R/|\eta|$ in the above expression).}
			
			{
				Fix $\nu\in S^{n-1}$,
				a sequence $\eps_j\to0$, and let $(u_j, v_j)$ be as in  Proposition~\ref{periodicity}, extended periodically in the directions of 
				$\nu^\perp\cap Q^\nu$ and constant along $\nu$, as in the proof of 
				\ref{lemmapropgsubadd}.
				Let $\tilde\nu\in S^{n-1}$, $\tilde\nu\ne \nu$,
				and choose $R\in O(n)$ such that {$\nu=R\tilde\nu$} and $|R-\Id|\le c |\nu-\tilde\nu|$
				{(for example, $R$ can be the identity on vectors orthogonal to both $\nu$ and $\tilde \nu$, and map $(\tilde\nu,\tilde\nu^\perp)$ to $(\nu,\nu^\perp)$ in this two-dimensional subspace)}.
				We fix a sequence $M_j\to\infty$ (for example, $M_j:=j$) and
				define 
				\begin{equation*}
				\tilde u_j(x):=u_j(M_jRx)\,,\hskip5mm
				\tilde v_j(x):=v_j(M_jRx)\,.
				\end{equation*}
				From $u_j\to z\chi_{\{x\cdot\nu\ge0\}}$ in $L^1_\loc(\R^n;\R^m)$ we obtain
				$\tilde u_j\to z\chi_{{\{x\cdot\tilde\nu\ge0\}}}$. 
				Further,  $\nabla \tilde u_j(x)=M_j\nabla u_j (M_jRx){R}$, which implies, recalling \eqref{eqpsiinfunifcont},
				\begin{equation*}
				\Psiinfty(\nabla \tilde u_j)(x)
				=M_j^2\Psiinfty(\nabla u_j R)(M_jRx)
				\le M_j^2(1+\omega_{|R-\Id|})\Psiinfty(\nabla u_j)(M_jRx).
				\end{equation*}}
			{
				Inserting in the definition of 
				$\calF_{\eps_j}^{\infty}(\tilde u_j, \tilde v_j; Q^{\tilde\nu})$ and using a change of variables leads to
				\begin{equation*}
				\calF_{\sfrac{\eps_j}{M_j}}^{\infty}(\tilde u_j, \tilde v_j; Q^{\tilde\nu})
				\le 
				(1+\omega_{|R-\Id|})
				M_j^{1-n}\calF_{\eps_j}^{\infty}(u_j, v_j; M_j{R}Q^{\tilde\nu}).
				\end{equation*}
				We observe that, although ${R}\tilde\nu=\nu$, we cannot in general expect ${R}Q^{\tilde\nu}=Q^\nu$.
				However, as $(u_j,v_j)$ are periodic in the directions orthogonal to $\nu$, the $(n-1)$-dimensional square 
				$\nu^\perp\cap M_j{R}Q^{\tilde\nu}$ can be covered by
				at most $M_j^{n-1} + c M_j^{n-2}$ disjoint translated copies of 
				the $(n-1)$-dimensional unit square $\nu^\perp\cap Q^\nu$. Therefore
				\begin{equation*}
				\begin{split}
				g(z, \tilde\nu)\le &
				\limsup_{j\to\infty} \calF_{\sfrac{\eps_j}{M_j}}^{\infty}(\tilde u_j, \tilde v_j; Q^{\tilde\nu})\\
				\le &
				(1+\omega_{|R-\Id|})\limsup_{j\to\infty}  (1+\frac c{M_j})
				\calF_{\eps_j}^{\infty}(u_j, v_j; Q^{\nu})\\
				= &  (1+\omega_{|R-\Id|}) g(z,\nu)
				\le (1+\omega_{c|\nu-\tilde\nu|}) g(z,\nu).\qedhere
				\end{split}
				\end{equation*}
			}
		\end{proof}
		\subsection{Density of the Cantor part}\label{ss:hqcinfty}
		
		We study now the behaviour of the surface energy density $g$ at small jump amplitudes.
		The next result is probably well known to experts. Despite this, we give a self-contained proof since we have not found a precise reference in the literature. Similar constructions are performed in \cite[Proposition 5.1]{AFP} for isotropic functionals defined on vector-valued measures. {The $L^1$ lower semicontinuity of $\calF_0$ is assumed to hold in Proposition~\ref{p:behaviour g at 0} below, as already mentioned at the beginning of Section \ref{ss:properties g}. Such a property follows, for instance, from the validity of Theorem~\ref{t:finale}. We stress again that Proposition~\ref{p:behaviour g at 0} is not used in the proof of Theorem~\ref{t:finale}, rather it provides a further piece of information on $g$ showing its linear behavior at small amplitudes.}
		
		\begin{proposition}\label{p:behaviour g at 0}
			{Assume that the functional $\calF_0$ defined in \eqref{F0}
				is $L^1(\Omega;\R^m)$ lower semicontinuous. Then, for all %{$z\in\R^m$ and} 
				$\nu\in S^{n-1}$} we have
			\[
			\lim_{z\to 0}\frac{g(z,\nu)}{h^\qcinfty(z\otimes \nu)}=1\,.
			\]
		\end{proposition}
		\begin{proof}
			With fixed $\nu\in S^{n-1}$, {let $x_0\in\Omega$ and $\rho>0$ be such that $Q^\nu_\rho(x_0)\subset\Omega$. Upon translating and scaling, it is not restrictive to assume $x_0=0$ and $\rho=1$}.
			For every $z\in\R^m$ consider the sequence
			\begin{equation}\label{eqdefwjx}
			w_j(x):=\varphi(j x\cdot\nu)z\,,\qquad {x\in Q^\nu,}
			\end{equation}
			where $\varphi(t):=(t\wedge 1)\vee 0$ for every $t\in\R$. 
			Clearly, $w_j\to u_z(x):=z\chi_{\{x\cdot\nu\geq 0\}}$ in $L^1(Q^\nu;\R^m)$, and thus by the 
			$L^1(Q^\nu;\R^m)$ lower semicontinuity of $\calF_0$ we conclude that
			\begin{align}\label{e:g leq hqcinfty}
			g(z,\nu)&=\calF_0(u_z,1;Q^\nu)\leq\liminf_{j\to\infty}\calF_0(w_j,1;Q^\nu)
			=\liminf_{j\to\infty}\int_{Q^\nu}h^\qc(\nabla w_j)\dx\notag\\
			&=\liminf_{j\to\infty}\int_{\{x\in Q^\nu:\,0\leq x\cdot\nu\leq\sfrac1j\}}h^\qc(j z\otimes\nu)\dx
			=\liminf_{j\to\infty}\frac{h^\qc(j z\otimes\nu)}{j}\notag\\
			&\leq h^\qcinfty(z\otimes\nu)\,.
			\end{align}
			
			On the other hand, given $z\in\R^m$ and any couple of sequences $z_j\to z$ and $t_j\to 0{^+}$, denote by $M_j$ the integer 
			part of $t_j^{-1}$ and define for every $k\in\N$, {$k\geq3$}, 
			\[
			u_{j,k}(x):=\sum_{i=0}^{M_j-1}i t_jz_j\chi_{[\frac{i}{kM_j},\frac{i+1}{kM_j})}(x\cdot\nu)
			+{z}\chi_{[\frac1k,\frac12]}(x\cdot\nu)\,.
			\]
			We show that $u_{j,k}$ converges, as $j\to\infty$, to $w_k$ as defined in \eqref{eqdefwjx} {for every {$k\geq3$}}. Indeed, 
			for $s:=x\cdot\nu\in 
			[\frac{i}{kM_j},\frac{i+1}{kM_j})\subseteq{[0,\frac1k)}$
			we have 
			\begin{align*}
			|it_jz_j-zks|\le|z-z_j|+|z_j|\,|it_j-ks|&\le
			|z-z_j|+{|z_j|\left(\frac{i}{M_j}\left|M_jt_j-1\right|+\frac1{M_j} \right)}\\&
			{\le|z-z_j|+|z_j|\left(\left|M_jt_j-1\right|+\frac1{M_j} \right)}\to0
			\end{align*}
			uniformly in $i$, hence $\|w_k-u_{j,k}\|_{L^\infty(Q^\nu;\R^m)}\to0$
			%   \[
			%   \|w_k-u_{j,k}\|_{L^\infty(Q^\nu;\R^m)}\leq t_j|z_j-z|+|z|(2M_j^{-1}-t_j)\to 0\,,
			%  \]
			as $j\to\infty$. Further,
			\begin{multline*}
			Du_{j,k}= D^ju_{j{,k}}=(t_jz_j\otimes \nu)
			\calH^{n-1}\res\cup_{i=1}^{M_j-{1}}\{x\in Q^\nu:\,x\cdot\nu=\textstyle{\frac{i}{k M_j}}\}\\{+((z-(M_j-1)t_jz_j)\otimes \nu)
				\calH^{n-1}\res\{x\in Q^\nu:\,x\cdot\nu=\textstyle{\frac{1}{k}}\}}\,.
			\end{multline*}
			Therefore, by the $L^1(Q^\nu;\R^m)$ lower semicontinuity of $\calF_0$ we conclude that
			\begin{align*}
			\frac1k h^\qc(kz\otimes\nu)&=\calF_0(w_k,1;Q^\nu)\leq\liminf_{j\to\infty}\calF_0(u_{j,k},1;Q^\nu)\\
			&=\liminf_{j\to\infty}\int_{J_{u_{j,k}}}g([u_{j,k}](x),\nu)\,d \calH^{n-1}(x)\\
			&=\liminf_{j\to\infty}(M_j-{1})g(t_jz_j,\nu)=\liminf_{j\to\infty}\frac {g(t_jz_j,\nu)}{t_j}\,.
			\end{align*}
			As this holds for every sequence, this implies
			% in turn implying 
			\begin{equation}\label{e:hqcinfty leq g parziale}
			h^\qcinfty(z\otimes\nu)\leq\liminf_{(t,z')\to(0,z)}\frac {g(tz',\nu)}{t}\,.
			\end{equation}
Indeed, the superior limit in the definition of $h^\qcinfty$ is actually a limit on rank-$1$ directions being $h^\qcinfty$ convex on those directions.
			
			Let now $\widetilde{z}_j\to 0$ be a sequence for which
			\[
			\liminf_{z\to 0}\frac{g(z,\nu)}{h^\qcinfty(z\otimes\nu)}=
			\lim_{j\to \infty}\frac{g(\widetilde{z}_j,\nu)}{h^\qcinfty(\widetilde{z}_j\otimes\nu)}\,.
			\]
			Upon setting $z_j:=\frac{\widetilde{z}_j}{{|}\widetilde{z}_j{|}}$,   
			up to subsequences we may assume that $z_j\to z_\infty\in S^{n-1}$. In addition, $t_j:={|}\widetilde{z}_j{|}\to 0$. Therefore, being
			$h^\qcinfty$ one-homogeneous we have that
			\[
			\frac{g(\widetilde{z}_j,\nu)}{h^\qcinfty(\widetilde{z}_j\otimes\nu)}=
			\frac{g(t_jz_j,\nu)}{t_j}\frac1{h^\qcinfty(z_j\otimes\nu)}\,.
			\]
			By the latter equality, by \eqref{e:hqcinfty leq g parziale} and by the continuity of $h^\qcinfty$ we infer
			\begin{equation}\label{e:hqcinfty leq g}
			\liminf_{z\to 0}\frac{g(z,\nu)}{h^\qcinfty(z\otimes\nu)}\geq 1\,.
			\end{equation}
			The conclusion follows at once from \eqref{e:g leq hqcinfty} and \eqref{e:hqcinfty leq g}.
		\end{proof}

		We now identify $h^\qcinfty$ explicitly as stated in \eqref{eqhinftypsiinf}.
		\begin{proposition}\label{p:identification Cantor energy density} 
			For all $\xi\in\R^{m\times n}$
			\begin{equation}\label{eqhinftypsiinf2}
			h^\qcinfty(\xi)=\ell(\Psi^{\sfrac12})^\qcinfty(\xi)\,.
			\end{equation}
		\end{proposition}
		\begin{proof}
			With fixed $\xi\in\R^{m\times n}$, the very definition of $h$ in \eqref{e:h} and the growth condition \eqref{e:Psi gc} easily imply 
			\[
			h^\qcinfty(\xi)=\limsup_{t\to\infty}\frac{h^\qc(t\xi)}{t}\leq
			\limsup_{t\to\infty}\frac{\ell(\Psi^{\sfrac12})^\qc(t\xi)}{t}=\ell(\Psi^{\sfrac12})^\qcinfty(\xi)\,.
			\]
			Let $\eps>0$, then for every $t>0$ consider $\varphi_t\in {C^\infty_c}(Q_1;\R^m)$ such that 
			\begin{equation}\label{e:phit}
			h^\qc(t\xi)\geq\int_{Q_1}h(t\xi+\nabla\varphi_t(x))\dx-\eps\,.
			\end{equation}
			Note that
			\begin{align*}
			E_t:=&\{x\in Q_1:\,h(t\xi+\nabla\varphi_t(x))=\Psi(t\xi+\nabla\varphi_t(x))\}\\
			=&
			\{x\in Q_1:\,\Psi^{\sfrac12}(t\xi+\nabla\varphi_t(x))\leq\ell\}\,,
			\end{align*}
			so that
			\[
			\int_{E_t}\ell\Psi^{\sfrac12}(t\xi+\nabla\varphi_t(x))\dx\leq\ell^2\,.
			\]
			Therefore, being $h\geq 0$ (cf. again \eqref{e:Psi gc}) from \eqref{e:phit} we infer that
			\[
			h^\qc(t\xi)\geq\int_{Q_1}\ell\Psi^{\sfrac12}(t\xi+\nabla\varphi_t(x))\dx-\ell^2-\eps
			\geq \ell(\Psi^{\sfrac12})^\qc(t\xi)-\ell^2-\eps\,,
			\]
			from which we conclude that
			\[
			h^\qcinfty(\xi)\geq%\limsup_{t\to\infty}\frac{h^\qc(t\xi)}{t}\geq
			\limsup_{t\to\infty}\frac{\ell(\Psi^{\sfrac12})^\qc(t\xi)}{t}=
			\ell(\Psi^{\sfrac12})^\qcinfty(\xi)\,.\qedhere
			\]
		\end{proof}
		From Propositions~\ref{p:behaviour g at 0} and \ref{p:identification Cantor energy density} 
		we deduce straightforwardly the ensuing statement.
		\begin{corollary}\label{c:behaviour g at 0}
			For all $\nu\in S^{n-1}$ we have
			\[
			\lim_{z\to 0}\frac{g(z,\nu)}{\ell(\Psi^{\sfrac12})^\qcinfty(z\otimes \nu)}=1\,.
			\]
		\end{corollary}
		
		% 
		% \MFF{$(\Psi^{\sfrac12})^\qcinfty$  \`e definita da un limite per le ipotesi che abbiamo su $\Psi$? Vorrei capire se (iii) del  Lemma~\ref{lemmahqc} vale in generale o meno.}
		
		{We conclude this section by proving that, under our hypotheses, the superior limit in the definition of $\Psi^{\sfrac12}$ is in fact a limit and that the operations of quasi-convexification and of recession for 
			$\Psi^{\sfrac12}$ commute. 
			
			\begin{proposition}\label{p:Psi1/2infty}
				We have that
				\begin{enumerate}
					\item\label{lemmapsi12lim} $
					\displaystyle(\Psi^{\sfrac12})^\infty(\xi)=
					{\displaystyle(\Psi_\infty)^{\sfrac12}(\xi)}=
					\lim_{t\to\infty}\frac{\Psi^{\sfrac12}(t\xi)}{t},$ for all $\xi\in\R^{m\times n}$;
					\item\label{lemmapsi12qcinfty}	$(\Psi^{\sfrac12})^{\qc,\infty}=(\Psi^{\sfrac12})^{\infty,\qc}.$
					\item\label{lemmapsi2} In the special case $\Psi_2(\xi):=\dist^2(\xi,\SO(n))$ one obtains $h^{\qc,\infty}(\xi)=\ell|\xi|$
					for all $\xi\in\R^{m\times n}$.
				\end{enumerate}
			\end{proposition}
			\begin{proof}
				{The second equality in \ref{lemmapsi12lim} follows immediately from \eqref{eqdefPsiinfty}. Then, the first is a consequence of the very definition of recession function.} 
				Alternatively, by \eqref{eqpsipsiinf} we infer that, for all $\delta>0$, there is $C_\delta>0$ satisfying
				\begin{equation}\label{psi12liminf}(\Psi^{\sfrac12})^\infty(\xi)\leq(1+\delta)\Psi^{\sfrac12}(\xi)+C_\delta,\qquad \text{for all }\xi\in\R^{m\times n}.
				\end{equation}
				This, together with the definition of recession function, implies \ref{lemmapsi12lim}.
				
				\ref{lemmapsi12qcinfty}
				Since $(\Psi^{\sfrac12})^{\qc}\leq \Psi^{\sfrac12}$, we immediately deduce $(\Psi^{\sfrac12})^{\qc,\infty}\leq(\Psi^{\sfrac12})^{\infty}$. 
				By \cite[Rem.~2.2(ii)]{fonsecamueller1993relaxation}, 
				$(\Psi^{\sfrac12})^{\qc,\infty}$  is quasiconvex, hence
				$(\Psi^{\sfrac12})^{\qc,\infty}\leq(\Psi^{\sfrac12})^{\infty,\qc}$. 
				
				Let us check the converse inequality.
				Let $\xi\in\R^{m\times n}$. By definition of quasi-convexification and \eqref{psi12liminf} we have
				\[(\Psi^{\sfrac12})^{\infty,\qc}(\xi)\leq\int_{(0,1)^n}(\Psi^{\sfrac12})^\infty(\xi+\nabla\varphi)\dx\leq (1+\delta)\int_{(0,1)^n}\Psi^{\sfrac12}(\xi+\nabla\varphi)\dx+C_\delta,\]
				for all $\varphi\in C^\infty_c((0,1)^n;\R^m)$. Hence, taking the infimum over $\varphi$ gives
				\[(\Psi^{\sfrac12})^{\infty,\qc}(\xi)\leq(1+\delta)(\Psi^{\sfrac12})^\qc(\xi)+C_\delta.\]
				Since $(\Psi^{\sfrac12})^{\infty}$ and therefore $(\Psi^{\sfrac12})^{\infty,\qc}$ are positively one-homogeneous, we obtain
				\[(\Psi^{\sfrac12})^{\infty,\qc}\leq(\Psi^{\sfrac12})^{\qc,\infty},\]
				which yields the thesis.
				
				\ref{lemmapsi2} From the definition of $\Psi_2$ one easily obtains $(\Psi_2^{\sfrac12})^\infty(\xi)=|\xi|$. As this function is quasiconvex, it coincides with 	
				$(\Psi_2^{\sfrac12})^{\infty,\qc}$, the assertion follows then from \ref{lemmapsi12qcinfty}
				and Proposition~\ref{p:identification Cantor energy density} .
			\end{proof}
			
			\clearpage

			\section{Lower bound}\label{lowerbound}
			
			%In what follows we shall study the $\Gamma(L^1(\Omega,\R^{m+1}))\hbox{-}$convergence of the family $(\calF_\eps)$.
			%To this aim it will be convenient to prove several intermediate results on the 
			%$\Gamma(L^{2}(\Omega;\R^m)\times L^1(\Omega))\hbox{-}\liminf_\eps\calF_\eps$ and the $\Gamma(L^{2}(\Omega;\R^m)\times L^1(\Omega))\hbox{-}\limsup_\eps\calF_\eps$.
			%For notational convenience we shall simply write $\Gamma(L^1)$ for $\Gamma(L^1(\Omega,\R^{m+1}))$ and 
			%$\Gamma(L^2)$ for $\Gamma(L^{2}(\Omega;\R^m)\times L^1(\Omega))$.
			%Being the $L^{2}(\Omega;\R^m)\times L^1(\Omega)$ topology finer than the $L^{1}(\Omega,\R^{m+1})$,
			%we have $\Gamma(L^2)\hbox{-}\liminf_\eps\calF_\eps\leq\Gamma(L^1)\hbox{-}\liminf_\eps\calF_\eps$ and 
			%$\Gamma(L^2)\hbox{-}\limsup_\eps\calF_\eps\leq\Gamma(L^1)\hbox{-}\limsup_\eps\calF_\eps$.
			
			\subsection{Domain of the limits}\label{s:scalar case}

			In order to characterize the compactness properties and the space in which the limit is finite 
			it is useful to consider the scalar simplification of functional,
			$\calF_\eps^\scal:W^{1,2}(A;\R\times[0,1])\to[0,\infty]$,
			\begin{equation}\label{Feps1}
			\calF_\eps^\scal(u,v;A):=\int_A \Big(
			%\big(1\wedge\frac{\eps {\ell^2} v^2}{(1-v)^2} \big) 
			{f_\eps^2(v)}|\nabla u|^2 + \frac{(1-v)^2}{4\eps}+\eps|\nabla v|^2 
			\Big)\dx.
			\end{equation}
			From \eqref{e:Psi gc}, one immediately obtains  that for any $(u,v)\in W^{1,2}(A;\R^m\times[0,1])$
			\begin{equation}\label{eqfepsfeps1}
			\frac1{c}\max_{i=1,\dots, m}
			\calF_\eps^\scal(u_i,v;A){-c\calL^n(A)}\le 
			\calF_\eps(u,v;A)\le c \sum_{i=1}^m
			\calF_\eps^\scal(u_i,v;A){+c\calL^n(A)}
			\end{equation}
			with the same constant $c\ge 1$ as in \eqref{e:Psi gc}.
			{In particular, \cite[Prop.~6.1]{ContiFocardiIurlano2016}
				implies that if $(u_\eps,v_\eps)\to (u,v)$ in $L^1(\Omega;\R^{m+1})$ with
				\[
				\liminf_{\eps\to0}\calF_\eps(u_\eps,v_\eps)<\infty
				\]
				then $u\in(GBV(\Omega))^m$ and $v=1$ $\calL^n$-a.e. $\Omega$ 
				(for a different proof see Remark~\ref{r:domain direct}).
				In addition, for every $i\in\{1,\ldots,m\}$ 
				\begin{equation}\label{e:one-dim estimate}
				\int_\Omega h_\scal^{\conv}({|\nabla u_i|})\dx
				+\int_{J_{u_i}}g_\scal(|[u_i](x)|)\dd\mathcal{H}^{n-1}
				+\ell|D^cu_i|(\Omega)<\infty\,.
				\end{equation}
				Here $h_\scal^{\conv}:[0,\infty)\to[0,\infty)$ is the convex function 
				explicitly defined by
				\[
				h_\scal^{\conv}(t):={(\ell t\wedge t^2)^\conv=}\begin{cases}
				t^2, & \textrm{ if } t\in[0,\frac\ell 2], \cr
				\ell t-\frac{\ell^2}4, & \textrm{ otherwise,}
				\end{cases}
				\]
				{(cf. 
				\eqref{e:hascal}-\eqref{e:hascalconv}).}
				We remark that it coincides with the simplified model $h_\simp$ for $m=1$ (cf. Lemma~\ref{lemmahqc}). Further,
				{$g_\scal:[0,\infty)\to[0,1]$ is the function implicitly defined by 
					\begin{equation}\label{e:gscal}
					g_\scal(t):=\inf_{\mathcal U_t}\int_0^1|1-\beta|
					\sqrt{f^2(\beta)|\alpha'|^2+|\beta'|}\dd s
					\end{equation}
					{where} $\mathcal U_t:=\{\alpha,\,\beta\in W^{1,2}((0,1)):\,
					\alpha(0)=0,\,\alpha(1)=t,\,0\leq\beta\leq 1,\,\beta(0)=
					\beta(1)=1\}$. In particular, $g_\scal$ satisfies
					\begin{enumerate}
						\item $g_\scal$ is subadditive: $g_\scal(t_1+t_2)\leq g_\scal(t_1)+g_\scal(t_2)$ for every $t_1,\,t_2\in[0,\infty)$, 
						\item $0\leq g_\scal(t)\leq 1\wedge\ell t$, 
						\item $\frac{g_\scal(t)}t\to\ell$ as $t\to 0^+$ 
					\end{enumerate}
					(cf. formula (1.6) in \cite[Theorem~1.1]{ContiFocardiIurlano2016} for the definition of $g_\scal$, and \cite[Section~4]{ContiFocardiIurlano2016} for further properties).
					
					In formula \eqref{e:one-dim estimate} the total variation of the Cantor part of the scalar function $u_i\in GBV(\Omega)$, $|D^cu_i|(\Omega)$, is defined as the least upper bound of the family of measures $|D^c\big((u_i\wedge k)\vee(-k)\big)|(\Omega)$, for $k>0$ (cf. \cite[Definition~4.33, Theorem~4.34]{AFP}). A similar construction can be performed for every 
					$u\in(GBV(\Omega))^m$.
					
					{Precisely}, \cite[Lemma~2.10]{AlFoc} implies that for every 
					$u\in(GBV(\Omega))^m$ for which $|D^cu|$ is a finite measure
					on $\Omega$, 
{one can construct} %{it can be built} 
a vector measure on 
					$\Omega$ with total variation coinciding exactly with 
					$|D^cu|(B)$ for every Borel subset $B$ of $\Omega$.
					For this reason such a vector measure, is denoted by $D^cu$.
					Let us briefly recall the construction of $D^cu$. 
					{To this aim, the family of truncations $\mathcal{T}_k$ defined in \eqref{e:Tk} is employed. Indeed,}
					for every $u\in(GBV(\Omega))^m$ such that $|D^cu|$ is a finite measure on $\Omega$, 
					{it is possible to show that} the following limit exists for every Borel subset $B$ of $\Omega$
					\begin{equation}\label{e:D^cu GBV}
					\lambda(B):=\lim_{k\to\infty}D^c(\mathcal{T}_k(u))(B)\,.
					\end{equation}
					{In addition, $\lambda$ is actually independent from the chosen family 
						of truncations.}
					The set function $\lambda$ turns out to be a vector Radon measure 
					on $\Omega$, and moreover equality $|\lambda|(B)=|D^cu|(B)$ is true 
					for every $B$ as above.
					
					Finally, for functions $u\in(GBV(\Omega))^m$ satisfying 
					estimate \eqref{e:one-dim estimate} it is also true that 
					\begin{equation}\label{e:Juinfty}
					\mathcal{H}^{n-1}(\{x\in J_u:\,u^+(x)=\infty \text{ or } u^-(x)=\infty\})=0
					\end{equation}
					(cf. \cite[Proposition~2.12, Remark~2.13]{AlFoc}), here one works with the one-point compactification of $\R^m$.} 
				We remark that we {deal} with $(GBV(\Omega))^m$
				and not with the strictly larger space $GBV(\Omega;\R^m)$, which is not even a vector space, see \cite[Remark~4.27]{AFP}.
				
				\subsection{Surface energy {in $BV$}}\label{ss:surface lb}
				We prove below the lower bound in $BV$ for the surface term. 
				We recall that the definition of the surface energy {density} $g$ has been given in \eqref{eqdefGsnu}.
				
				\begin{proposition}\label{lowerboundBVsfc}
					Let $u\in BV(\Omega;\R^m)$,
					and $(u_\eps,v_\eps)\to(u,1)$ in $L^1(\Omega;\R^{m+1})$.
					Then {for all $A\in\calA(\Omega)$}
					\begin{equation}\label{lbjump}
					\int_{J_u{\cap A}}g([u],\nu_u)\dd\calH^{n-1}  \le
					\liminf_{\eps\to0}\Functeps(u_\eps,v_\eps;A)
					% 	{\Gamma(L^1)\hbox{-}}\liminf_{\eps\to0} \Functeps(u, 1{;A}),
					\end{equation}
					where $g$ has been defined in \eqref{eqdefGsnu}. 
				\end{proposition}
				\begin{proof}
					Let $(u_\eps,v_\eps)\to(u,1)$ in {$L^1(\Omega;\R^{m+1})$} be such that 
					\[
					\liminf_{\eps\to0}\Functeps(u_\eps,v_\eps;A)<\infty.
					\]
					Up to subsequences and with a small abuse of notation, we can assume that the previous lower limit is in fact a limit. 
					Let us define the measures $\mu_\eps\in \calM^+_b(A)$
					\[{\mu_\eps}:=\left(f_\eps^2(v_\eps) \Psi(\nabla u_\eps) + \frac{(1-v_\eps)^2}{4\eps}+\eps|\nabla v_\eps|^2\right) \calL^n{\LL {A}}.\] 
					Extracting a further subsequence, we can assume that 
					\begin{equation}\label{convmeps}
					\mu_\eps\rightharpoonup \mu \quad weakly^*\text{ in }\calM (A)=(C_c^0(A))'
					\end{equation}
					as $\eps\to0$, for some $\mu\in \calM^+_b(A)$, so that
					\[\liminf_{\eps\to0}\Functeps(u_\eps,v_\eps;A)\geq \mu(A).\]
					
					Equation \eqref{lbjump} will follow once we have proved that
					\begin{equation}\label{densg}\frac{\dd \mu}{\dd\calH^{n-1}\res J_u}(x_0)\geq g([u](x_0),\nu_u(x_0)), \quad \calH^{n-1} \text{-a.e.\ } x_0\in J_u\cap A\,.
					\end{equation}
					We will prove the last inequality for points $x_0\in J_u\cap A$ such that
					\[\frac{\dd \mu}{\dd\calH^{n-1}\res J_u}(x_0)=\lim_{\rho\to0}\frac{\mu(Q^\nu_\rho(x_0))}{\calH^{n-1}(J_u\cap Q^\nu_\rho(x_0))} \quad \text{exists finite},\]
					\[\lim_{\rho\to 0}\frac{\mathcal{H}^{n-1}(J_u\cap Q^\nu_\rho(x_0))}{\rho^{n-1}}=1,\]
					where $\nu:=\nu_u(x_0)$ and $Q^\nu_\rho(x_0){:=x_0+\rho Q^\nu}$ is the cube centred 
					in $x_0$, with side length $\rho$, and one face orthogonal to $\nu$. We remark that such conditions define a set of full measure in $J_u\cap A$.
					
					For $x_0\in J_u\cap A$ as above, we get
					\begin{align*}%\label{blowupgam}
					&\frac{\dd \mu}{\dd\calH^{n-1}\res J_u}(x_0)=\lim_{\rho\to0}\frac{\mu(Q^\nu_\rho(x_0))}{\rho^{n-1}}
					=\lim_{{\scriptsize \begin{array}{c}
							\rho\!\!\in\!\! I\\ \!\!\rho\!\!\to\!\!0\end{array}}}\lim_{\eps\to0}\frac{\mu_\eps(Q^\nu_\rho(x_0))}{\rho^{n-1}}
					\end{align*} 
					where we used \eqref{convmeps} and 
					\[I:=\Big\{\rho\in(0,\frac{2}{\sqrt n}\mathrm{dist}(x_0,\partial A)):\,
					\mu(\partial Q^\nu_\rho(x_0))=0\Big\}.\]
					We introduce
					\[\gamma_\eps:=\inf\{z\in[0,1]:f(z)\geq \eps^{-\sfrac12}\}
					{=\frac{1}{1+\ell\eps^{1/2}}}
					,\]
					\[\tilde v_\eps:=\min\{v_\eps,\gamma_\eps\}\]
					and compute
					\begin{align*}
					\calF_\eps(u_\eps,\tilde v_\eps;Q^\nu_\rho(x_0))&=\calF_\eps(u_\eps,v_\eps;Q^\nu_\rho(x_0)\setminus\{v_\eps>\gamma_\eps\})\\
					&+\int_{{Q^\nu_\rho(x_0)\cap }\{v_\eps>\gamma_\eps\}}{\Psi(\nabla u_\eps)}\dx+{\frac{(1-\gamma_\eps)^2}{4\eps}} \calL^n(Q^\nu_\rho(x_0)\cap\{v_\eps>\gamma_\eps\})\\
					&\leq\calF_\eps(u_\eps,v_\eps;Q^\nu_\rho(x_0))+\frac{\ell^2}4\rho^n\,,
					\end{align*}
					{where in the last step we used that the definition of {$\gamma_\eps$} implies $1-\gamma_\eps=\ell\gamma_\eps\eps^{1/2}{\le\ell\eps^{1/2}}$}.
					Therefore
					\begin{align}\label{blowupgam}
					\frac{\dd \mu}{\dd\calH^{n-1}\res J_u}(x_0)
					\geq \limsup_{{\scriptsize \begin{array}{c}
							\rho\!\!\in\!\! I\\ \!\!\rho\!\!\to\!\!0\end{array}}}\limsup_{\eps\to0}
					{\frac{\calF_\eps(u_\eps,\tilde v_\eps;{Q^\nu_\rho}(x_0))}{\rho^{n-1}}}\,.
					%\int_{Q^\nu_\rho(x_0)}\left(\eps f^2(\tilde v_\eps) \Psi(\nabla u_\eps) + \frac{(1-{\tilde v_\eps})^2}{4\eps}+\eps|D\tilde v_\eps|^2\right)dx,		
					\end{align} 
					% {To justify \eqref{blowupgam} simply note that 
					
					By \eqref{eqpsipsiinf}, for every $\delta\in(0,1)$ one has
					$\Psi(\xi)\ge (1-\delta)\Psiinfty(\xi)$ for $\xi$ sufficiently large. As $\Psiinfty$ is continuous, 
					there is $C(\delta)>0$ such that
					\begin{equation*}
					\Psi (\xi)+ C(\delta) \ge (1-\delta)\Psiinfty(\xi)\hskip5mm \text{ for all }\xi.
					%\Psiinfty(\xi)\le (1+\delta)\Psi(\xi)+ C(\delta) \hskip5mm \text{ for all }\xi.
					\end{equation*}
					We choose $\delta_\rho\to0$ such that $\rho C(\delta_\rho)\to0$.
					As $\eps f^2(\tilde v_\eps)\le 1$, we have
					\begin{equation*}
					\eps f^2(\tilde v_\eps) \Psi(\nabla u_\eps)  
					\ge (1-\delta_\rho)\eps f^2(\tilde v_\eps) \Psiinfty  (\nabla u_\eps) -C(\delta_\rho)
					% \eps f^2(\tilde v_\eps) \Psiinfty (\nabla u_\eps)  
					%  \le (1+\delta_\rho)\eps f^2(\tilde v_\eps) \Psi (\nabla u_\eps) +C(\delta_\rho)
					\end{equation*}
					with $\rho^{1-n} \calL^n(Q_\rho^\nu) C(\delta_\rho)=\rho C(\delta_\rho)\to0$ as $\rho\to0$.
					% Since
					% \[\rho^{1-n}\int_{Q^\nu_\rho(x_0)\cap \{v_\eps>\gamma_\eps\}}\frac{(1-\gamma_\eps)^2}{4\eps}dx\leq \rho,\] 
					We conclude by \eqref{blowupgam} that
					\begin{equation}\label{gamm}
					\frac{\dd \mu}{\dd\calH^{n-1}\res J_u}(x_0)
					% \geq \lim_{{\scriptsize \begin{array}{c}
					% 		\rho\!\!\in\!\! I\nonumber\\ \!\!\rho\!\!\to\!\!0\end{array}}}\lim_{\eps\to0}\frac{\mu_\eps(Q^\nu_\rho(x_0))}{\rho^{n-1}}\nonumber\\
					\geq \limsup_{{\scriptsize \begin{array}{c}
							\rho\!\!\in\!\! I\\ \!\!\rho\!\!\to\!\!0\end{array}}}\limsup_{\eps\to0} \frac{\calF_\eps^{\infty}(u_\eps,\tilde v_\eps; Q^\nu_\rho(x_0))}{\rho^{n-1}},
					\end{equation}
					where $\calF_\eps^{\infty}$ has been defined in \eqref{Feps*}. 
					%C{At this point this contains already the recession function!}
					Setting $y:=(x-x_0)/\rho\in Q^\nu$, and changing variable in the previous expression we get
					\begin{equation*}%\label{eqdmudlimsuplimsupf}
					\frac{\dd \mu}{\dd\calH^{n-1}\res J_u}(x_0)\geq \limsup_{{\scriptsize \begin{array}{c}
							\rho\!\!\in\!\! I\\ \!\!\rho\!\!\to\!\!0\end{array}}}\limsup_{\eps\to0} \calF_\eps^{\infty}(u_\eps^\rho,\tilde v_\eps^\rho;Q^\nu),
					\end{equation*}
					where $w^\rho(y):=w(\rho y+x_0)$ for $y\in Q^\nu$. 
					Recalling that $u_\eps\to u$ in $L^1(\Omega;\R^m)$,
					by diagonalization we can find subsequences $\{\rho_k\}_k$ and $\{\eps(\rho_k)\}_k$ such that $u_{\eps(\rho_k)}^{\rho_k}\to [u](x_0)\chi_{\{y\cdot\nu>0\}}+u^-(x_0)$ in $L^1(Q^\nu;\R^m)$ and
					{\begin{equation*}
						\frac{\dd \mu}{\dd\calH^{n-1}\res J_u}(x_0)\geq \lim_{k\to\infty}\calF_{{\eps(\rho_k)}}^{\infty}(u_{{\eps(\rho_k)}}^{\rho_k},\tilde v_{{\eps(\rho_k)}}^{\rho_k};Q^\nu).
						\end{equation*}}
					%{$\calF_{{\eps(\rho_k)}}^{\infty}(u_{{\eps(\rho_k)}}^{\rho_k},\tilde v_{{\eps(\rho_k)}}^{\rho_k};Q^\nu)\to \frac{\dd \mu}{\dd\calH^{n-1}\res J_u}(x_0)$ as $k\to\infty$. }
					Being $\calF_\eps^{\infty}$ invariant for translations of the first argument, we find 
					\[\frac{\dd \mu}{\dd\calH^{n-1}\res J_u}(x_0)\geq \liminf_{k\to\infty}\calF^{\infty}_{\eps{(\rho_k)}}(u_{\eps(\rho_k)}^{\rho_k},\tilde v_{\eps(\rho_k)}^{\rho_k};Q^\nu)\geq g([u](x_0),\nu_u(x_0)),\]
					that is \eqref{densg}, and this concludes the proof. 
				\end{proof}

				\subsection{Diffuse part in $BV$}\label{ss:lb diffuse}
				
				\begin{proposition}\label{lowerboundBVdff2}
					Let $u\in BV(\Omega;\R^m)$, 
					$(u_\eps,v_\eps)\to (u,1)$ in $L^1(\Omega;\R^{m+1})$,
					$A\in \calA(\Omega)$. Then
					\begin{equation}\label{e:stima dal basso2}
					\int_A h^\qc(\nabla u)\dx + \int_A h^\qcinfty (\dd D^cu) \le  \liminf_{\eps\to0} \Functeps(u_\eps, v_\eps; A)
					%   \Gamma(L^1)\hbox{-}\liminf_{\eps\to0} \Functeps(u, 1, A),
					\end{equation}
					where {$h^\qc$ and $h^\qcinfty$ have been defined in \eqref{e:h}-\eqref{eqdefqcinfty}.}
				\end{proposition}
				
				We remark that this statement can be proven using the 
				lower-semicontinuity result by Fonseca
				and Leoni  \cite[Th. 1.8]{FonsecaLeoni2001}, following an argument similar to that used in \cite[Subsection~4.1]{AlFoc}.
				Instead, our proof is based on the following result from
				\cite[Theorem~4.1]{AmbrosioDalmaso1992}, see also \cite[Theorem~5.47]{AFP}.
				\begin{theorem}[Ambrosio-Dal Maso]\label{theoambdm92}
					Let $\phi:\R^{m\times n}\to[0,\infty)$ be quasiconvex and such that
					\begin{equation*}
					0\le \phi(\xi)\le c(1+|\xi|) \quad\text{ for all } \xi\in\R^{m\times n},
					\end{equation*}
					and define $F:L^1(\Omega;\R^m)\to\R$ by
					\begin{equation*}
					F(u):=\begin{cases}
					\displaystyle\int_\Omega \phi(\nabla u) \dx, & \text{ if } u\in W^{1,1}(\Omega;\R^m),\\
					\infty, & \text{ otherwise in }
					L^{1}(\Omega;\R^m).
					\end{cases}
					\end{equation*}
					Then for any $u\in BV(\Omega;\R^m)$ we have
					\begin{equation*}
					sc^-(L^1)\hbox{-}F(u)=\int_\Omega \phi(\nabla u) \dx + \int_\Omega \phi^\infty( \dd D^su),
					\end{equation*}
					where $\phi^\infty(\xi):=\limsup_{t\to\infty} \phi(t\xi)/t$. 
					In particular the latter functional is lower semicontinuous with respect to the strong $L^1(\Omega;\R^m)$ convergence.
				\end{theorem}

				We start with a truncation result.
				\begin{lemma}\label{lemmaabdelta}
					There are two functions $\alpha,\beta:(0,1)\to(0,1)$, with 
					$\lim_{\delta\uparrow 1}\alpha_\delta=1$
					and
					$\lim_{\delta\uparrow 1}\beta_\delta=0$, such that 
					for any $\eps>0$, $(u_\eps,v_\eps)\in W^{1,2}(\Omega;\R^m\times[0,1])$, $\delta\in(0,1)$ 
					{and} $A\in\calA(\Omega)$
					there is $\tilde u_\eps^\delta\in GSBV(A;\R^m)$ such that
					\begin{equation*}
					H_\delta(\tilde u_\eps^\delta;A)\le \Functeps(u_\eps,v_\eps;A) 
					+ {h(0) \calL^n(A\cap\{v_\eps\leq \delta\})},
					\end{equation*}
					where $H_\delta$
					%:L^1(\Omega;\R^m\times [0,1])\times \calA(\Omega)\to[0,\infty]$ 
					is defined 
					for $A\in \calA(\Omega)$ and $w\in L^1(A;\R^m)$ by
					\begin{equation}\label{eqdefHdelta}
					H_\delta(w;A):=
					\begin{cases}
					\displaystyle \alpha_\delta\int_A h^\qc(\nabla w) \dx + \beta_\delta
					\calH^{n-1}(A\cap J_w),
					& \text{ if }w\in GSBV(A;\R^m),\\
					\infty, &\text{ otherwise}.
					\end{cases}
					\end{equation}
					If one has
					$(u_\eps,v_\eps)\to(u,1)$  
					in $L^1(\Omega;\R^{m+1})$ {as $\eps\to0$,}
					then $\tilde u_\eps^\delta\to u$ in $L^1(A;\R^{m})$
					as $\eps\to0$, for {any fixed} $\delta\in(0,1)$.
				\end{lemma}
				{We stress that, for the sake of notational simplicity, we will omit here and below the explicit dependence of $\tilde u^\delta_\eps$ on the set $A$.}
				\begin{proof}
					We fix $\delta\in(0,1)$ and {$\eps>0$. 
						We} compute, for any pair $(u,v)\in W^{1,2}(\Omega;\R^{m}\times[0,1])$, 
					\begin{align}\label{eqFdelta}
					\Functeps&(u,v;A)\geq\int_{\{\eps f^2(v)>1\}\cap A}\Psi(\nabla u)\dx
					+\int_{\{\eps f^2(v)\leq 1\}\cap A}\Big(\eps f^2(v) \Psi(\nabla u) +\delta^2 \frac{(1-v)^2}{4\eps}\Big)\dx\notag\\ 
					&+ \int_A\Big((1-\delta^2)\frac{(1-v)^2}{4\eps} + \eps |\nabla v|^2\Big)\dx\notag \\
					&\geq\int_{\{\eps f^2(v)>1\}\cap A}\Psi(\nabla u)\dx
					+\delta \int_{\{\eps f^2(v)\leq 1\}\cap A}v\ell{\Psi}^{\sfrac12}(\nabla u)\dx\notag\\
					&+\sqrt{1-\delta^2}\int_A |\nabla(\Phi(v))|\dx\notag\\
					&\geq
					\delta \int_{A}\Big(\Psi(\nabla u)\wedge v\ell{\Psi}^{\sfrac12}(\nabla u)\Big)\dx
					+\sqrt{1-\delta^2}\int_A |\nabla(\Phi(v))|\dx\notag\\
					&\geq  \delta \int_{A}v h (\nabla u)\dx
					+\sqrt{1-\delta^2} \int_A|\nabla(\Phi(v))|\dx,
					\end{align}
					where $h$ has been introduced in \eqref{e:h} and  $\Phi:[0,1]\to[0,\frac12]$ is defined by 
					\begin{equation}\label{e:Phi}
					\Phi(t):=\int_0^t (1-s) \ds = t-\frac12 t^2\,.
					\end{equation}
					We observe that $\Phi$ is strictly increasing, $\Phi(1)=\frac12$ and in particular $\Phi$ is bijective. By the coarea formula,
					\begin{equation*}
					\int_A |\nabla (\Phi(v))|\dx = \int_0^{1/2}
					\calH^{n-1}(A\cap \partial^*\{\Phi(v)>t\}) \dt.
					\end{equation*}
					Therefore there is $\bar t\in (\Phi(\delta^2),\Phi(\delta))$ such that
					\begin{equation*}
					(\Phi(\delta)-\Phi(\delta^2))
					\calH^{n-1}(A\cap \partial^*\{\Phi(v)>\bar t\}) \le \int_A |\nabla (\Phi(v))|\dx.
					\end{equation*}
					We define 
					\begin{equation*}
					\tilde u:=u\chi_{\{\Phi(v)>\bar t\}\cap A} \in GSBV(A;\R^m)
					\end{equation*}
					{(dropping the dependence on both $\eps$ and $\delta$ from $\tilde u$)} and obtain from \eqref{eqFdelta},
					\begin{align*}
					\Functeps(u,v;A)\ge  &\delta\Phi^{-1}(\bar t)\int_{A} h(\nabla \tilde u) \dx 
					+ \sqrt{1-\delta^2}(\Phi(\delta)-\Phi(\delta^2))
					\calH^{n-1}(A\cap J_{\tilde u})\\
					& {-h(0) \calL^n(\{\Phi(v){\leq}\bar t\}{\cap A})}.
					\end{align*}
					We recall that $\bar t\ge \Phi(\delta^2)$ and that $\Phi$ is increasing, 
					define $\alpha_\delta:=\delta^3$, $\beta_\delta:=\sqrt{1-
						\delta^2}(\Phi(\delta)-\Phi(\delta^2))$, and conclude
					{\begin{align*}
						\Functeps&(u,v;A)\ge  \alpha_\delta\int_{A} h(\nabla \tilde u) \dx 
						+ \beta_\delta
						\calH^{n-1}(A\cap J_{\tilde u}){-h(0)\calL^n(\{v\leq\delta\}{\cap A})}.
						\end{align*}
					}
					% {Therefore}
					% $\Functeps(u,v)\ge H_\delta(\tilde u){-\rho_j^\delta}$  where
					% $H_\delta$ was defined in \eqref{eqdefHdelta}
					% {and 
					% $\rho_j^\delta:=c (1-\delta^3)|A|+c |\{\Phi(v)>\bar t\}|$ 
					% obeys
					% $\lim_{\delta\to0}\limsup_{j\to\infty} \rho_j^\delta=0$.}
					
					We also remark that $\|\tilde u-u\|_{L^1(A)}
					\leq\|u\|_{L^1(\{v\le \Phi^{-1}(\bar t)\})}$, hence, if the sequence $u_\eps$  is equiintegrable and $v_\eps\to1$ in $L^1(A)$, we obtain that $u_\eps-\tilde u_\eps\to0$ in $L^1(A;\R^m)$.
				\end{proof}
The next lemma is a minor reformulation of \cite[Lemma~5.1]{Larsen}. The latter improves the statement of \cite[Theorem~3.95]{AFP} on the convergence of the blow-ups of a $BV$-function in a Cantor point. A more general version of this result can be found in \cite[Lemma 10.6]{Rindler}.
\begin{lemma}\label{choicesequence} Let $u\in BV(\Omega;\R^m)$ and let $\eta:\Omega\to S^{m-1}$, $\xi:\Omega\to S^{n-1}$ be Borel maps such that $D^cu=\eta\otimes\xi|D^cu|$. Then, for $|D^cu|$-a.e. $x\in\Omega$ and for all given $\mu\in\mathcal{M}^+(\Omega)$, there exists a sequence $\rho_i\to0$, as $i\to \infty$, such that
\begin{eqnarray}
&\mu(\partial Q_{\rho_i}^{\xi(x)}(x))=0, \qquad \text{ for all $i\geq1$,}\label{chargeboundary}\\	
& \displaystyle t_{\rho_i}:=\frac{|Du|(Q^{\xi(x)}_{\rho_i}(x))}{\rho_i^n}\to\infty, \qquad t_{\rho_i}\rho_i\to0,\label{cantorscaling}\\
&\displaystyle \frac{ u(x+\rho_i y)-u_{Q^{\xi(x)}_{\rho_i}(x)}}{t_{\rho_i} \rho_i}  
\to \eta(x) \chi( y\cdot \xi(x)) \qquad {\text{strictly-} BV(Q^{\xi(x)};\R^m),\label{blowupcantor}}
\end{eqnarray}
as $i\to\infty$,
for some nondecreasing function $\chi:(-\sfrac12,\sfrac12)\to\R$ with \begin{equation}\label{variation1}
|D\chi|((-\sfrac12,\sfrac12))=1,
\end{equation}
where $u_{Q^{\xi(x)}_{\rho_i}(x)}$ denotes the average of $u$ over $Q^{\xi(x)}_{\rho_i}(x)$.
	\end{lemma}	
\begin{proof}
For simplicity we will denote $Q_1:=Q^{\xi(x)}$, $Q_\rho(x):=x+\rho Q_1$, and
\[u^{\rho}_x(y):=\frac{ u(x+\rho y)-u_{Q^{\xi(x)}_{\rho}(x)}}{t_{\rho} \rho}, \qquad \text{for }y\in Q_1.\] 
By general properties of $BV$ functions \eqref{cantorscaling} holds for the entire family $\rho\to0$ and by Radon-Nikodym differentiation
\begin{equation}\label{differentiation}
\lim_{\rho\to 0}\frac{D^cu^{\rho}_x(Q_1)}{|D^cu^{\rho}_x|(Q_1)}=\eta(x)\otimes \xi(x),
\end{equation}
$|D^cu|$-a.e. $x\in \Omega$.
Up to a further $|D^cu|$-negligible set, \cite[Theorem~3.95]{AFP} and \cite[Lemma~5.1]{Larsen} provide a sequence $\rho_i\to 0$ such that
\begin{eqnarray}
&|Du^{\rho_i}_x|\rightharpoonup \gamma\qquad \text{weakly*-}\calM(Q_1),\label{convvar}\\
& u^{\rho_i}_x(y)\to u_x(y):=\eta(x)\chi(y\cdot \xi(x))\qquad \text{weakly*-} BV(Q_1;\R^m)\label{convfunc},
\end{eqnarray}
as $i\to\infty$, for some $\gamma\in \calM^+(Q_1)$ with $\gamma(Q_1)=1$ and some nondecreasing function $\chi:(-\sfrac12,\sfrac12)\to\R$ with $|D\chi|((-\sfrac12,\sfrac12))\leq1$. 

Let us check that the sequence $\rho_i\to0$ can be chosen such that \eqref{chargeboundary} holds. Indeed, fixed $i\in \mathbb{N}\setminus \{0\}$, we have $\mu(\partial Q_{s\rho_i}(x_0))=0$ for $\calL^1$-a.e. $s\in (0,\sfrac{1}{\rho_i})$. Moreover, the maps 
\begin{eqnarray*}
&s\in (0,\sfrac{1}{\rho_i})\mapsto u^{s\rho_i}_x\in L^1(Q_1,\R^m),\\
&s\in (0,\sfrac{1}{\rho_i})\mapsto |Du^{s\rho_i}_x|\in\calM^+(Q_1)
\end{eqnarray*}
are continuous as $s\to 1^-$, respectively for the convergences $L^1(Q_1;\R^m)$ and weak*-$\calM(Q_1)$, by definition of $u^{\rho}_x$ and $t_{\rho}$. Hence, we can find $s_i\in(0,1)$ such that \eqref{chargeboundary}, \eqref{cantorscaling} and the $L^1(Q_1,\R^m)$ convergence in \eqref{blowupcantor} hold for $s_i\rho_i$ in place of $\rho_i$.

We next check \eqref{variation1}. By \eqref{convvar} and \eqref{convfunc} we have that $|Du_x|\leq\gamma$. Hence, for $t\in(0,1)$ such that $\gamma(\partial Q_t)=0$, recalling that $|Du^{\rho_i}_x|(Q_1)=\gamma(Q_1)=1$, we obtain
\begin{eqnarray*}
&|Du^{\rho_i}_x|(Q_t)\to\gamma(Q_t), \qquad
|Du^{\rho_i}_x|(Q_1\setminus Q_t)\to\gamma(Q_1\setminus Q_t),\\
& Du^{\rho_i}_x(Q_t)\to Du_x(Q_t).
\end{eqnarray*}
We infer that
\[
\limsup_{i\to\infty}|Du^{\rho_i}_x(Q_1)- Du_x(Q_1)|\leq 2\gamma(Q_1\setminus Q_t),
\]
and letting $t\to 1^-$ gives $Du^{\rho_i}_x(Q_1)\to Du_x(Q_1)$ as $i\to\infty$. In conclusion
\[Du_x(Q_1)=\lim_{i\to\infty}Du^{\rho_i}_x(Q_1)=\lim_{i\to\infty}\frac{Du^{\rho_i}_x(Q_1)}{|Du^{\rho_i}_x|(Q_1)}=\eta(x)\otimes\xi(x),\]
and then $D\chi(-\sfrac12,\sfrac12)=1$. This gives \eqref{variation1} by monotonicity of $\chi$. Finally, $|Du_x|(Q_1)=1$ provides the strict-$BV(Q_1;\R^m)$ convergence in \eqref{blowupcantor}.
\end{proof}
				\begin{proof}[Proof of Proposition~\ref{lowerboundBVdff2}]
					{\em Step 0: Preparation.}
					We assume $(u_\eps,v_\eps)\to (u,1)$ in $L^1(\Omega;\R^{m+1})$ for some $u\in BV(\Omega;\R^m)$. 
					Let $A\subseteq\calA(\Omega)$, $\delta\in(0,1)$ and let $\tilde u_\eps^\delta$ be as in Lemma~\ref{lemmaabdelta}. We define the measure
					\begin{equation*}
					\mu_\eps^\delta:=\alpha_\delta h^\qc(\nabla \tilde u_\eps^\delta){\calL^n\LL A} + \beta_\delta \calH^{n-1}\LL  
					(A\cap J_{\tilde u_\eps^\delta}),
					\end{equation*}
					{so that $\mu_\eps^\delta(A)=H_\delta(\tilde u_\eps^\delta;A)\le \calF_\eps(u_\eps,v_\eps;A)
						+ {h(0)\calL^n(A\cap\{v_\eps\leq\delta\})}$.}
					Passing to a subsequence we can assume that 
					$\displaystyle{\lim_{\eps\to0}}\Functeps({u_\eps,v_\eps};A)$ exists finite and that 
					$\mu_\eps^\delta\weakto\mu^\delta$ weakly{$^*$} in the sense of measures on $A$ as $\eps\to0$, for some $\mu^\delta\in\calM^+_b(A)$. 
					If we can show that
					\begin{equation}\label{eqlbvolulmepart}
					\frac{\dd \mu^\delta}{\dd\calL^{n}}(x_0)
					\geq \alpha_\delta h^\qc(\nabla u(x_0))\quad\text{ 
						for $\calL^n$-a.e. $x_0\in A$}
					\end{equation}
					and
					\begin{equation}\label{eqlbvcantorpart}
					\frac{\dd \mu^\delta}{\dd|Du|}(x_0)
					\geq \alpha_\delta  h^{\qcinfty}\left(\frac{\dd Du}{\dd|Du|}(x_0)\right)
					\quad\text{ for $|D^cu|$-a.e. $x_0\in A$}
					\end{equation}
					for all $\delta\in(0,1)$, then the conclusion follows.

					{\em Step 1: Absolutely continuous part.}
					We prove \eqref{eqlbvolulmepart}. 
					We can assume that the left-hand side is finite.
					%For simplicity we drop $\delta$ from the notation.
					First we observe that for $\calL^n$-a.e. $x_0\in A$ one has
					\begin{equation*}
					\frac{\dd \mu^{\delta}}{\dd\calL^{n}}(x_0)
					=\lim_{\rho\to0} \frac{\mu^{\delta}(Q_\rho(x_0))}{\rho^n}
					=\lim_{\rho\to0\atop \rho\in I} \lim_{\eps\to0}\frac{\mu_\eps^{\delta}(Q_\rho(x_0))}{\rho^n}
					\end{equation*}
					where $Q_\rho(x_0):=x_0+(-\frac12\rho,\frac12\rho)^n$ and
					$I:=\{\rho\in (0, {\frac{2}{\sqrt n}}\dist(x_0,\partial A)):\,
					\mu^\delta(\partial Q_\rho(x_0))=0\}$.
					We define {$u^\rho:Q_1\to\R^m$ by}
					\begin{equation*}
					u^\rho(y):=\frac{  u(x_0+\rho y)- u(x_0)} {\rho}.
					\end{equation*}
					By the properties of $BV$, for $\calL^n$-a.e. $x_0\in A$, after possibly extracting a further subsequence, 
					$u^\rho(y)\to \nabla u(x_0)y$ in $L^1(Q_1;\R^m)$ {as $\rho\to0$}.
					We further define
					\begin{equation*}
					u^\rho_\eps(y):=\frac{ \tilde u_\eps^\delta(x_0+\rho y)- u(x_0)} {\rho}
					\end{equation*}
					so that $u^\rho_\eps\to u^\rho$ in $L^1(Q_1;\R^m)$ {as $\eps\to0$} for any fixed $\rho>0$ {(and $\delta\in(0,1)$)}. 
					We take a diagonal subsequence so that $w_i(y):=u^{\rho_i}_{\eps_i}(y)\to \nabla u(x_0)y$ in $L^1(Q_1;\R^m)$ and
					\begin{equation}\label{eqdmudlx0abdsdf}
					\frac{\dd \mu^{\delta}}{\dd\calL^{n}}(x_0)
					=\lim_{i\to\infty}\left[
					\int_{Q_1} \alpha_\delta h^\qc(\nabla w_i) \dx + \frac{\beta_\delta}{\rho_i} \calH^{n-1}(J_{w_i}\cap Q_1)\right]\,.
					\end{equation}
					We fix $M\in\N$ and for every $i$, {by averaging} we choose $k_i\in\{M+1,\dots, 2M\}$ such that
					\begin{equation}\label{e:average}
					\int_{\{a_{k_i}<|w_i|<a_{k_i+1}\}} h^\qc(\nabla w_i) \dx 
					\le \frac 1M
					\int_{Q_1} h^\qc(\nabla w_i) \dx\,, 
					\end{equation}
					which implies that $\hat w_i:=\mathcal{T}_{k_i}(w_i)$, {with $\mathcal{T}_{k_i}$ defined in \eqref{e:Tk}}, obeys
					\begin{equation}\label{eqthatkwcm} 
					\begin{split}
					&\int_{Q_1} h^\qc(\nabla \hat w_i) \dx 
					%  + \frac{\beta_{\delta}}{\rho_{\eps_i}} \calH^{n-1}(J_{\mathcal{T}_{k_i}(w_i)}\cap Q_1)
					\le (1+\frac CM) 
					\int_{Q_1} h^\qc(\nabla w_i) \dx {+C\calL^n(\{|w_i|\ge a_{k_i}\})}.
					%  + \frac{\beta_{\delta}}{\rho_{\eps_i}} \calH^{n-1}(J_{w_i})
					%  +\alpha_\delta h^\qc(0) \calL^n(\{|w_i|\ge a_{M+2}\})
					\end{split}
					\end{equation}
					{
						Indeed, in view of \eqref{e:h gc} and $\|\nabla\mathcal{T}_{k_i}\|_{L^\infty(\R^m)}\leq1$ we have
						\begin{align*}
						\int_{Q_1}& h^\qc(\nabla \hat w_i) \dx \leq\int_{\{|w_i|\leq a_{k_i}\}} h^\qc(\nabla w_i) \dx\\
						&+\int_{\{a_{k_i}<|w_i|< a_{k_i+1}\}} h^\qc(\nabla \hat w_i)  \dx +h(0)\calL^n(\{|w_i|\ge a_{k_i+1}\})\\
						&\leq\int_{Q_1} h^\qc(\nabla w_i) \dx+C\int_{\{a_{k_i}<|w_i|< a_{k_i+1}\}} h^\qc(\nabla w_i) +C\calL^n(\{|w_i|\ge a_{k_i}\})\,.
						\end{align*}
						The inequality in \eqref{eqthatkwcm} then follows from \eqref{e:average}.}
					
					Moreover, note that if ${a_M}>\|\nabla u(x_0)y\|_{L^\infty(Q_1)}+1$ then 
					$w_i\to \nabla u(x_0)y$ implies 
					$\hat w_i\to \nabla u(x_0)y$ in $L^1(Q_1;\R^m)$.
					
					We recall that $\mathcal{T}_{k_i}\in C^1$ implies
					$\calH^{n-1}(J_{\hat w_i}\cap Q_1)\le
					\calH^{n-1}(J_{w_i}\cap Q_1)$.
					From \eqref{eqdmudlx0abdsdf} and $\rho_i\to0$ we deduce $\calH^{n-1}(J_{w_i}\cap Q_1)\to0$ and, 
					with $|\hat w_i|\le a_{M+1}$ pointwise, 
					\begin{equation*}
					|D^s\hat w_i|(Q_1)=\int_{J_{\hat w_i}\cap Q_1} |[\hat w_i]| \dd\calH^{n-1} \le {2 a_{M+1}} \calH^{n-1}(J_{w_i}\cap Q_1)\to0
					\end{equation*}
					and therefore
					\begin{equation*}
					\int_{Q_1} h^{\qcinfty}(\dd D^s\hat w_i)
					\le 
					c|D^s\hat w_i|(Q_1)\to0.
					\end{equation*}
					With \eqref{eqdmudlx0abdsdf} and \eqref{eqthatkwcm}, {using that $w_i\to \nabla u(x_0)y$ in measure, we get}
					\begin{equation*}
					\alpha_\delta \lim_{i\to\infty} \left[
					\int_{Q_1} h^\qc(\nabla \hat w_i) \dx +
					\int_{Q_1} h^{\qcinfty}(\dd D^s\hat w_i) \right]
					\le 
					(1+\frac CM)\frac{\dd \mu^{\delta}}{\dd\calL^{n}}(x_0).
					\end{equation*}
					By the lower semicontinuity of the functional in the left-hand side
					(Theorem~\ref{theoambdm92})
					and $\hat w_i\to \nabla u(x_0)y$ {in $L^1(Q_1;\R^m)$} we deduce
					\begin{equation*}
					\alpha_\delta 
					h^\qc(\nabla u(x_0)) \le
					(1+\frac CM)\frac{\dd \mu^{\delta}}{\dd\calL^{n}}(x_0)
					\end{equation*}
					for $\calL^n$-a.e. $x_0$,
					every $M$, and  every $\delta$.
					This proves \eqref{eqlbvolulmepart}.
					
					{\em Step 2: Cantor part.}
					We prove \eqref{eqlbvcantorpart}.
					By Alberti's rank-one theorem we can assume without loss of generality that
					\begin{equation}\label{e:cantor point}
					\frac{\dd Du}{\dd|Du|}(x_0)=\eta(x_0)\otimes \xi(x_0)
					\end{equation}
					with $\eta(x_0)\in S^{m-1}$, $\xi(x_0)\in S^{n-1}$
					{for $|D^cu|$-a.e. $x_0\in A$.}
					We fix a unit cube $Q_1:=Q^{\xi(x_0)}$ with one face orthogonal to $\xi(x_0)$, write $Q_\rho(x_0):=x_0+\rho Q_1$,  {and select a sequence $\rho_i\to0$ as in Lemma \ref{choicesequence},
					applied for the given $u\in BV(\Omega,\R^m)$ and $\mu:=\mu^\delta$.} %and define

					As above, for $|D^cu|$-a.e. $x_0$ one has
					\begin{equation*}
					\frac{\dd \mu^\delta}{\dd|Du|}(x_0)
					=\lim_{\rho\to0} \frac{\mu^\delta(Q_\rho(x_0))}{|Du|(Q_\rho(x_0))}
					=\lim_{{i\to\infty}} \lim_{\eps\to0}\frac{\mu_\eps^\delta(Q_{{\rho_i}}(x_0))}{|Du|(Q_{{\rho_i}}(x_0))}.
					\end{equation*}
					We define
					\begin{equation*}
					u_\eps^\rho(y):=
					\frac{ \tilde u_\eps^{\delta}(x_0+\rho y)-u_{Q_\rho(x_0)}}{t_\rho \rho},
					\end{equation*}
					so that,
defining $u_{x_0}(y):=\eta(x_0)\chi_{x_0}(y\cdot \xi(x_0))$,
					{$\lim_{\rho\to0}\lim_{\eps\to0 }u_\eps^\rho= u_{x_0}$ in $L^1(Q_1;\R^m)$ (for every $\delta\in (0,1)$) and}
					\begin{equation*}
					\begin{split}
					\frac{\dd \mu^\delta}{\dd|Du|}(x_0)
					&=\lim_{i\to\infty} \lim_{\eps\to0}
					\left[\frac{\alpha_\delta}{{\rho_i^n} t_{\rho_i}}
					\int_{Q_{\rho_i}(x_0)} h^\qc(\nabla \tilde u_\eps^\delta) \dx
					+ \frac{\beta_\delta}{{\rho_i^n} t_{\rho_i}}
					\calH^{n-1}(J_{\tilde u_\eps^\delta}\cap Q_{\rho_i}(x_0))
					\right]\\
					&=\lim_{i\to\infty} \lim_{\eps\to0}
					\left[\frac{\alpha_\delta}{t_{\rho_i}}
					\int_{Q_1} h^\qc(t_{\rho_i} \nabla u_\eps^{{\rho_i}}) \dy 
					+ \frac{\beta_\delta}{{\rho_i} t_{\rho_i}}
					\calH^{n-1}(J_{u_\eps^{{\rho_i}}}\cap Q_1)
					\right].
					\end{split}
					\end{equation*}
					% We define
					% \begin{equation}
					%  u^\rho(y):=\frac{ \tilde u(x_0+\rho y)- u(x_0)} {\rho}.
					% \end{equation}
					Taking a diagonal subsequence we see that there {is 
						$\eps_i\to0$} such that
					\begin{equation*}
					w_i:= u^{\rho_i}_{\eps_i} \to u_{x_0} \text{ in } L^1(Q_1;\R^m)
					\end{equation*}
					with $|Du_{x_0}|(Q_1)=1$, and setting $t_i:=t_{\rho_i}\to\infty$,
					\begin{equation*}
					\frac{\dd \mu^\delta}{\dd|Du|}(x_0)
					=\lim_{i\to\infty}
					\left[\frac{\alpha_\delta}{t_{i}}
					\int_{Q_1} h^\qc(t_i \nabla w_i) \dy 
					+ \frac{\beta_\delta}{\rho_i t_{i}}
					\calH^{n-1}(J_{w_i}\cap Q_1)
					\right].
					\end{equation*}
					We fix $M>0$ and, by averaging,
					for every $i$ choose $k_i\in\{M+1,\dots, 2M\}$ such that
					\begin{equation*}
					\int_{Q_1 \cap \{a_{k_i}<|w_i|<a_{k_i+1}\}} h^\qc(t_i\nabla w_i) \dx 
					\le \frac 1M
					\int_{Q_1} h^\qc(t_i \nabla w_i) \dx,
					\end{equation*}
					which implies that $\hat w_i:={\mathcal{T}_{k_i}(w_i)}\in SBV\cap L^\infty(Q_1;\R^m)$ obeys, arguing as in Step 1 above and by taking into account that $t_i\to\infty$,
					\begin{align}\label{eqlb2lis1cm}
					&\limsup_{i\to\infty}
					\int_{Q_1} {\frac{{ \alpha_\delta}}{t_i}}h^\qc(t_i \nabla \hat w_i)\dy + \frac{\beta_\delta}{\rho_i t_i} \calH^{n-1}(J_{\hat w_i}\cap Q_1)
					\notag\\
					&\le(1+\frac CM)
					\lim_{i\to\infty}
					\int_{Q_1} {\frac{{\alpha_\delta} }{t_i}}h^\qc(t_i \nabla w_i)\dy + \frac{\beta_\delta}{\rho_i t_i} \calH^{n-1}(J_{w_i}\cap Q_1)+ \frac{C}{{t_i}}\calL^n(\{|w_i|>a_{k_i}\})
					\notag\\
					&= (1+\frac CM)
					\frac{\dd \mu^\delta}{\dd|Du|}(x_0).
					\end{align}
					Further, {since $\chi$ is bounded}, {for $M$ sufficiently large we have}
					$r_i:=\|\hat w_i-u_{x_0}\|_{L^1(Q_1)} \to0$.
					For every $i$ we select $q_i\in(1-r_i^{1/2},1)$ such that
					\begin{equation*}
					\int_{\partial Q_{q_i}} |\hat w_i^--u_{x_0}^+|\dd\calH^{n-1}\le \frac{1}{r_i^{1/2}} \|\hat w_i-u_{x_0}\|_{L^1(Q_1)} =r_i^{1/2}\to0,
					\end{equation*}
					where $\hat w_i^-$ and $u_{x_0}^+$ denote the inner and outer trace, respectively, and define 
					\begin{equation*}
					w_i^*:=\begin{cases} 
					\hat w_i, &\text{ in } Q_{q_i},\\
					u_{x_0}, &\text{ in } Q_1\setminus Q_{q_i}.
					\end{cases}
					\end{equation*}
					Then, the choice of $q_i$, {\eqref{eqlb2lis1cm}, and $\rho_i t_{i}\to 0$ yield}
					\begin{equation}\label{e:wistar sing}
					\int_{Q_1} h^{\qcinfty} (\dd D^sw_i^*)\le
					c r_i^{1/2}
					+ c_M \calH^{n-1}(J_{\hat w_i}\cap Q_1)
					+{\int_{Q_1\setminus Q_{q_i}} h^{\qcinfty}(\dd D^su_{x_0})}
					\to0.
					\end{equation}
					{In addition}, we get {from \eqref{e:h gc} and $t_i\to\infty$}
					\begin{equation}\label{e:wistar ass}
					\lim_{i\to\infty} \int_{Q_1\setminus Q_{q_i}} 
					{\frac{1}{t_i} h^\qc\left({t_i} \nabla u_{x_0}\right)\dy}
					\le \lim_{i\to\infty} c \int_{{Q_1\setminus Q_{q_i}}} \dd |Du_{x_0}|=0\,.
					\end{equation}
					Further, $w_i^*\in BV(Q_1;\R^m)$ 
					and $\mathrm{supp}(w_i^*-u_{x_0})\subset\subset Q_1$.
					By \cite[Lemma~5.50]{AFP} {and Theorem \ref{theoambdm92}}
					\[
					\int_{Q_1} {\frac{1}{t_i}}h^\qc\left({{t_i}}\nabla w_i^*\right)\dy
					+ \int_{Q_1} h^{\qcinfty}(\dd D^sw_i^*)
					\geq {\frac{1}{t_i}  h^\qc\left({{t_i}}
						Du_{x_0}(Q_1)
						\right).
					}
					\]
					Therefore, {being $Du_{x_0}(Q_1)=\eta(x_0)\otimes \xi(x_0) D\chi((-\sfrac 12,\sfrac 12))$ a rank-one matrix, the latter estimate together with \eqref{e:wistar sing} and \eqref{e:wistar ass} yield that}
					\begin{equation*}\begin{split}
					h^{\qcinfty}(Du_{x_0}(Q_1))=&
					\lim_{i\to\infty} \frac1{t_i}
					h^\qc(t_i Du_{x_0}(Q_1))\\
					\le & \liminf_{i\to\infty} \left[
					\int_{Q_1}  \frac1{t_i} h^\qc(t_i \nabla w_i^*)\dy 
					+ \int_{Q_1} h^{\qcinfty}( \dd D^sw_i^*)\right]\\
					\le & \liminf_{i\to\infty} 
					\int_{Q_1}  \frac1{t_i} h^\qc(t_i \nabla \hat w_i)\dy.
					\end{split}
					\end{equation*}
					Recalling \eqref{eqlb2lis1cm}, we infer that
					\begin{equation*}
					\alpha_\delta h^{\qcinfty}(Du_{x_0}(Q_1))
					\le  (1+\frac CM)
					\frac{\dd \mu^\delta}{\dd|Du|}(x_0),
					\end{equation*} 
					for {every $M$ sufficiently large}. Therefore, by letting $M\to\infty$ we conclude that
					\begin{equation*}
					\alpha_\delta h^{\qcinfty}(Du_{x_0}(Q_1))
					\le
					\frac{\dd \mu^\delta}{\dd|Du|}(x_0).
					\end{equation*} 
					As $Du_{x_0}(Q_1)=\eta(x_0)\otimes \xi(x_0) D\chi((-\sfrac12,\sfrac12))=
					\eta(x_0)\otimes \xi(x_0)$, this and \eqref{e:cantor point}
					conclude the proof of \eqref{eqlbvcantorpart}.
				\end{proof}
				{
					The lower bound in $BV$ follows at once from the lower bounds for the surface and the diffuse parts. 
					\begin{theorem}\label{t:lbcompBV}
						Let $u\in BV(\Omega;\R^m)$. Then, {for all $A\in \calA(\Omega)$}
						\begin{equation}\label{e:lbBV}
						\calF_0(u,1;{A}) \le \Gamma(L^1)\hbox{-}\liminf_{\eps\to0} \Functeps(u, 1;{A}),
						\end{equation}
						where $\calF_\eps$ and $\calF_0$ have been defined in \eqref{functeps} and \eqref{F0}. 
					\end{theorem}
					\begin{proof} {For simplicity, we will prove the statement for $A=\Omega$.}
						We argue by localization. Assume that $(u_\eps,v_\eps)\to(u,1)$ in $L^1(\Omega;\R^{m+1})$, with $u\in BV(\Omega;\R^m)$, and that
						\[\liminf_{\eps\to0}\calF_\eps(u_\eps,v_\eps)<\infty.\]
						Set
						\begin{eqnarray*}&\mu(A):=\displaystyle\liminf_{\eps\to0}\calF_\eps(u_\eps,v_\eps;A),\qquad \text{for all } A\in \calA(\Omega),\\
							&\lambda:=\calL^n\res\Omega+\calH^{n-1}\res J_u+|D^cu|\\
							&\psi_1:=g([u],\nu_u),\qquad \psi_2:=h^{\qc}(\nabla u)+h^{\qc,\infty}(\frac{\text{d} D^cu}{\text{d} |D^c u|}),	
						\end{eqnarray*}	
						and notice that $\mu$ is a monotone set function which is superadditive on disjoint open sets, $\lambda$ is a positive Borel measure and $\psi_i$ are positive Borel functions satisfying 
						\[\mu(A)\geq\int_A\psi_i \text{d} \lambda, \qquad \text{for }i=1,2\text{ and }A\in\calA(\Omega)\]
						thanks to Propositions \ref{lowerboundBVsfc} and \ref{lowerboundBVdff2}.
						By \cite[Proposition 1.16]{Braides} we conclude 
						\[\mu(\Omega)\geq\int_\Omega(\psi_1\vee\psi_2)\text{d}\lambda,\]
						which gives the thesis.	
				\end{proof}}
				\begin{remark}\label{r:domain direct}
					From the argument in Lemma~\ref{lemmaabdelta} one can also deduce directly that $u\in (GBV(\Omega))^m$.
					Indeed, consider $(u_\eps,v_\eps)\to(u,v)$ in $L^1(\Omega;\R^{m+1})$ with $\sup_\eps\calF_\eps(u_\eps,v_\eps)<\infty$.
					Necessarily $ v=1$ $\calL^n$-a.e. on $\Omega$. Moreover, with fixed $\delta\in(0,1)$, 
					keeping the notation introduced in Lemma~\ref{lemmaabdelta}, 
					using the growth conditions on $h$ (see \eqref{e:h gc})  we get
					\begin{align*}
					\int_\Omega|\nabla \tilde{u}_\eps^\delta|\dx+\calH^{n-1}(J_{\tilde{u}_\eps^\delta})
					\leq c(\calF_\eps(u_\eps,v_\eps)+1)\,,
					\end{align*}
					for some positive constant $c$ depending on $\delta$ and on $\calL^n(\Omega)$. 
					In particular, for each component $(\tilde{u}_\eps^\delta)_i$ of $\tilde{u}_\eps^\delta$, $i\in\{1,\ldots,n\}$, 
					we have $(\tilde{u}_\eps^\delta)_i\in GSBV(\Omega)$ and
					\begin{align*}
					\int_\Omega|\nabla (\tilde{u}_\eps^\delta)_i|\dx+\calH^{n-1}(J_{(\tilde{u}_\eps^\delta)_i})
					\leq c(\calF_\eps(u_\eps,v_\eps)+1)\,.
					\end{align*}
					Then, %the truncation argument used several times throughout the paper, for instance in Proposition~\ref{lowerboundBVdff2}, 
					if $k>0$ and $\tau_k(s):=(s\vee k)\wedge(-k)$, from the estimate above we infer that
					$|D(\tau_k((\tilde{u}_\eps^\delta)_i))|(\Omega)\leq C_k$,
					with $C_k>0$ depending on $k$ and on the sequence, but not on $\eps$. Therefore, there is a subsequence that converges weakly in 
					$BV(\Omega)$.
					This implies, recalling that $\tilde{u}_\eps^\delta\to u$ in $L^1(\Omega;\R^m)$ as $\eps\to0$ for all 
					$\delta\in(0,1)$, that $\tau_k(u_i)\in BV(\Omega)$ for all $k$. In conclusion, we deduce that $u_i\in GBV(\Omega)$, for 
					all $i\in\{1,\ldots,n\}$, and thus $u\in (GBV(\Omega))^m$.
				\end{remark}
				
				{
					\subsection{Lower bound in $GBV$}\label{ss:GBV}
					In this section we extend the validity of the lower bound {Theorem~\ref{t:lbcompBV}} to every $u\in (GBV(\Omega))^m$. We first prove that the functional $\calF_0$ is continuous under truncations.
					\begin{proposition}\label{contF0}
						Let $\calF_0$ and $\mathcal{T}_k$ be defined as in \eqref{F0} and \eqref{e:Tk}, respectively. 
						Then, for all $u\in (GBV(\Omega))^m$ with 
						$\calF_0(u,1)<\infty$ we have	
						\[
						\lim_{k\to\infty}\calF_0(\mathcal{T}_k(u),1)=\calF_0(u,1)\,.
						\]
					\end{proposition}
					\begin{proof}
						We prove the convergence of the volume, Cantor and surface terms separately. {It is useful to recall for the rest of the proof that 
							$\|\nabla \mathcal{T}_k\|_{L^\infty(\R^m)}\le 1$.} 
						% In this respect note that 
						% \begin{align*}
						% \calF_0(\mathcal{T}_k(u),1,\Omega)&=\int_\Omega h^\qc\big(\nabla(\mathcal{T}_k(u))\big)\dx\\&+
						%  \int_{S(\mathcal{T}_k(u))}g([u],\nu_u)d\calH^{n-1}+\int_\Omega h^{\qcinfty}\big(dD^c(\mathcal{T}_k(u))\big)\,.
						% \end{align*}
						
						For the volume part, we observe that \eqref{e:h gc} implies 
						$|\nabla u|\in L^1(\Omega)$. 
						We have $\nabla(\mathcal{T}_k(u))=\nabla u$ for $\calL^n$-a.e. $x\in\Omega_k:=\{|u|\leq a_k\}$, therefore in view of 
						\eqref{e:h gc}
						%item \ref{lemmahqcgrowth} in Lemma~\ref{lemmahqc} and \eqref{e:Psi gc} 
						we get
						\[
						\Big|\int_\Omega h^\qc\big(\nabla(\mathcal{T}_k(u))\big)\dx-
						\int_\Omega h^\qc(\nabla u)\dx\Big|\leq c\int_{\Omega\setminus \Omega_k}(1+|\nabla u|)\dx\,,
						\]
						so that, {as $a_k\to\infty$ as $k\uparrow\infty$, we conclude}
						\[
						\lim_{k\to\infty} \int_\Omega h^\qc\big(\nabla(\mathcal{T}_k(u))\big)\dx=\int_\Omega h^\qc(\nabla u)\dx\,.
						\]
						For the surface term we recall that $J_{\mathcal{T}_k(u)}\subseteq J_u$ for every $k\in\N$ with $\nu_{\mathcal{T}_k(u)}=\nu_u$ for $\mathcal H^{n-1}$-a.e. 
						$x\in J_{\mathcal{T}_k(u)}$. Then, thanks to \eqref{e:Juinfty} we infer that
						$(\mathcal{T}_k(u))^\pm\to u^\pm$, $\chi_{J_{\mathcal{T}_k(u)}}\to \chi_{J_u}$ 
						and $|[\mathcal{T}_k(u)]|\le |[u]|$ $\mathcal H^{n-1}$-a.e. in $J_u$, and then we conclude
						\begin{align*}
						\lim_{k\to\infty} \int_{J_{\mathcal{T}_k(u)}}g([\mathcal{T}_k(u)],\nu_{\mathcal{T}_k(u)})\dd\calH^{n-1}&=
						\lim_{k\to\infty} \int_{J_u}g([\mathcal{T}_k(u)],\nu_u)\chi_{J_{\mathcal{T}_k(u)}}\dd\calH^{n-1}\\
						&=\int_{J_u}g([u],\nu_u)\dd\calH^{n-1}
						\end{align*}
						thanks to Lemmata \ref{lemmapropg} and \ref{lemmapropg2} \ref{lemmapropgcont} and to the Dominated Convergence Theorem.
						
						For what the Cantor part of the energy is concerned, by \eqref{e:h gc} we have that $0\le h^{\qc,\infty}(\xi)\le c |\xi|$.
						Further, the definitions of $\mathcal{T}_k$ and of $D^cu$ outlined in \eqref{e:D^cu GBV} yield in particular
						\[D^c(\mathcal{T}_k(u))\res \Omega_k=D^c u\res \Omega_k\] 
						\[|D^c(\mathcal{T}_k(u))|\ll |D^c u|,  
						\hskip5mm 
						\frac{\dd|D^c(\mathcal{T}_k(u))|}{\dd|D^c u|}\le 1
						\,.
						\] 
						Thus, 
						\begin{multline*}
						\Big|\int_{\Omega}h^{\qcinfty}(\dd D^c(\mathcal{T}_k(u)))-\int_{\Omega_k}h^{\qcinfty}(\dd D^cu)\Big|\\
						\leq c \int_{\Omega\setminus\Omega_k}\dd |D^cu|
						= c|D^cu|(\Omega\setminus\Omega_k), 
						\end{multline*}
						and therefore
						\[\lim_{k\to\infty}\int_{\Omega}h^{\qcinfty}(\dd D^c(\mathcal{T}_k(u)))=\int_{\Omega}h^{\qcinfty}(\dd D^cu),\]
						which concludes the proof.
					\end{proof}
					We are ready to prove the lower bound for generalized functions of bounded variations.
					\begin{theorem}\label{t:lbcomp}
						Let $u\in (GBV(\Omega))^m$. Then
						\begin{equation}\label{e:lbcomp}
						\calF_0(u,1) \le \Gamma(L^1)\hbox{-}\liminf_{\eps\to0} \Functeps(u, 1),
						\end{equation}
						where $\calF_\eps$ and $\calF_0$ have been defined in \eqref{functeps} and \eqref{F0}. 
					\end{theorem}
					\begin{proof}
						{Let} $u\in (GBV(\Omega))^m$ and let $(u_\eps,v_\eps)\in {W^{1,2}}(\Omega; \R^{m+1})$ be such that $(u_\eps,v_\eps)\to (u,1)$ in $L^1(\Omega;\R^{m+1})$, with $v_\varepsilon\in[0,1]$ $\calL^n$-a.e. in $\Omega$. Without loss of generality we can suppose that $\liminf_{\eps\to0}\calF_{\eps}(u_\eps,v_\eps)<\infty$, {that the latter is actually a limit (up to a subsequence not relabeled), and that} $(u_\eps,v_\eps)\to(u,1)$ $\calL^n$-a.e. {in $\Omega$}. {In particular, from Section \ref{s:scalar case} we infer that $u\in(GBV(\Omega))^m$, with $|\nabla u|\in L^1(\Omega)$ and satisfying \eqref{e:Juinfty} {and \eqref{e:one-dim estimate}}, so that $\calF_0(u,1)<\infty$.}
						
						Recalling the definition of the truncation $\mathcal{T}_k$ in \eqref{e:Tk}, we have that $\mathcal{T}_k(u_\eps)\to \mathcal{T}_k(u)$ in $L^1(\Omega;\R^m)$ 
						{for any $k$} and that $\mathcal{T}_k(u)\in BV({\Omega};\R^m)$, {being $\mathcal{F}_0(u,1)<\infty$.}
						{Hence,} we can apply {Theorem~\ref{t:lbcompBV}} to say that 
% 						I{la versione localizzata di Theorem~\ref{t:lbcompBV} credo servisse solo qui.}
						\begin{equation}\label{estTkA}
						{\calF_0(\mathcal{T}_{k_{M}}(u),1)\leq \liminf_{\eps\to0}\calF_{\eps}(\mathcal{T}_{k_{M}}(u_\eps),v_\eps).}
						\end{equation}
						We claim that for all $M\in\N$ there is $k_{M}\in\{M+1,\dots,2M\}$ independent of $\eps$ such that {after extracting a further subsequence}
						\begin{equation}\label{estTk}
						{\calF_{\eps}(\mathcal{T}_{k_{M}}(u_\eps),v_\eps)\leq \Big(1+\frac{c}{M}\Big)\calF_{\eps}(u_\eps,v_\eps)
						{+c\calL^n(\{|u_\eps|>a_M\})},}
						\end{equation}
						for some $c>0$ independent of $\eps$ and of $M$.
						Given this for granted, we get by \eqref{estTkA}, \eqref{estTk}
						{and by the convergence $u_\eps\to u$ in measure}
						\[
						{\limsup_{M\to\infty}\calF_0(\mathcal{T}_{k_M}(u),1)}\leq\liminf_{\eps\to0}\calF_{\eps}(u_\eps,v_\eps).
						\]
						Finally, using the continuity under truncations for $\calF_0$ established in {Proposition~\ref{contF0}}, we obtain 
						\[
						\calF_0(u,1)\leq\liminf_{\eps\to0}\calF_{\eps}(u_\eps,v_\eps)
						\]
						and hence \eqref{e:lbcomp}.
						
						It remains to prove \eqref{estTk}. To this aim we argue as in Proposition~\ref{periodicity} using De~Giorgi's averaging-slicing method on the range.
						First, for all $k\in\N$ we split the energy contributions
						\begin{multline}\label{e:Feps* Tk2}
						{\calF_{\eps}(\mathcal{T}_k(u_\eps),v_\eps)=\calF_{\eps}(u_\eps,v_\eps;\{|u_\eps|\leq a_k\})}\\
						{+\calF_{\eps}(\mathcal{T}_k(u_\eps),v_\eps;\{a_k<|u_\eps|< a_{k+1}\})%\notag\\&
						+\calF_{\eps}(0,v_\eps;\{|u_\eps|\geq a_{k+1}\})\,.}
						\end{multline}
						By \eqref{e:Psi gc} and the definition of $\mathcal{T}_k$, the last but one term in the previous expression can be estimated as
						\begin{align}\label{e:Feps* Tk 22}
						\calF_{\eps}&{(\mathcal{T}_k(u_\eps),v_\eps;\{a_k<|u_\eps|<a_{k+1}\})
						\leq c\int_{\{a_k<|u_\eps|< a_{k+1}\}}f^2_\eps(v_\eps)\Psi(\nabla u_\eps)\dx}\notag\\
						&{+c\calL^n(\{a_k<|u_\eps|<a_{k+1}\})
						+{\calF_{\eps}(0,}v_\eps;\{a_k<|u_\eps|< a_{k+1}\})\,,}
						\end{align}
						% and  
						% \begin{align}\label{e:Feps* Tk 32}
						% \calF_{\eps}(0,v_\eps;\{|u_\eps|\geq a_{k+1}\}\cap A)={\calF_{\eps}(0,}v_\eps;\{|u_\eps|\geq a_{k+1}\}\cap A)\,,
						% \end{align}
						for some $c>0$.
						Summing \eqref{e:Feps* Tk2} and \eqref{e:Feps* Tk 22} and averaging, 
						we conclude that there exists $k_{M,\eps}\in\{M+1,\dots,2M\}$ such that
						\begin{align*}
						{\calF_{\eps}(\mathcal{T}_{k_{M,\eps}}(u_\eps),v_\eps)}&
						{\leq\frac 1M\sum_{k=M+1}^{2M}\calF_{\eps}(\mathcal{T}_k(u_\eps),v_\eps)}\\
						&{\leq \Big(1+\frac{c}{M}\Big)\calF_\eps(u_\eps,v_\eps){+c\calL^n(\{|u_\eps|>a_M\})}\,,}
						\end{align*}
						for some $c>0$. As $\eps\to 0$, there exists a subsequence of $\{k_{M,\eps}\}$ that is independent of $\eps$. 
						This yields \eqref{estTk} and concludes the proof.
					\end{proof}
				}

\clearpage

\section{Upper bound}\label{upperbound}

% We will prove that the $\overline{\Gamma}$-limit of $\calF_\eps$ satisfies the hypotheses of \cite[Theorem 3.12]{BouchitteFonsecaMascarenhas}, so that it can be represented as an integral functional. 
% Its diffuse and surface densities will be identified by a direct computation.
% In order to be able to obtain existence of minimizers, we shall also consider a perturbed version of the functional 
% which includes an additional uniformly coercive term. We prove the upper bound directly for the modified functional. 
% We fix a function $\eta:(0,1]\to[0,1]$ such that
% \begin{equation}\label{eqassetaeps}
% \lim_{\eps\to0} \frac{\eta_\eps}{\eps}=0
% \end{equation}
% and define
% \begin{equation}\label{eqdefcalfepseta}
%  \calF_\eps^\eta(u,v;A):=\calF_\eps(u,v;A)+\eta_\eps \int_A \Psi(\nabla u) \dx,
% \end{equation}
% where $\calF_\eps$ has been defined in \eqref{functeps}.

In this Section we prove the $\Gamma-\limsup$ inequality in Theorem~\ref{t:finale}.}
In order to be able to obtain existence of minimizers for the perturbed functionals (see Section~\ref{compandconv}), we consider a perturbed version of the functional 
which includes an additional uniformly coercive term, and prove the upper bound directly for the modified functional. 
We fix a function $\eta:(0,1]\to[0,1]$ such that
\begin{equation}\label{eqassetaeps}
\lim_{\eps\to0} \frac{\eta_\eps}{\eps}=0
\end{equation}
and define
\begin{equation}\label{eqdefcalfepseta}
 \calF_\eps^\eta(u,v;A):=\calF_\eps(u,v;A)+\eta_\eps \int_A \Psi(\nabla u) \dx,
\end{equation}
where $\calF_\eps$ has been defined in \eqref{functeps}.

One key ingredient in the proof of the upper bound is that
 the $\overline{\Gamma}$-limit of {$\calF_\eps^\eta$} satisfies the hypotheses of \cite[Theorem 3.12]{BouchitteFonsecaMascarenhas}, so that it can be represented as an integral functional. 
Its diffuse and surface densities will be identified by a direct computation.

In order to prove that $\overline{\Gamma}\text{-}\lim_{\eps\to0}\calF_\eps^{{\eta}}(u,1;\cdot)$ is a Borel measure, we first check the weak subadditivity of the $\Gamma$-upper limit of $\calF_\eps^{{\eta}}$. 
\begin{lemma}\label{l:wsub}
	Let $u\in L^1(\Omega;\R^m)$, let $A',A,B\in\mathcal{A}(\Omega)$ with $A'\subset\subset A$, then 
	\begin{multline}\label{e:subadd}
		{\Gamma(L^1)\hbox{-}\limsup_{{\eps\to0}}\calF_\eps^\eta(u,1;A'\cup B)\leq}\\ 	{\Gamma(L^1)\hbox{-}\limsup_{{\eps\to0}}\calF_\eps^\eta(u,1;A)+	\Gamma(L^1)\hbox{-}\limsup_{{\eps\to0}}\calF_\eps^\eta(u,1;B),}
	\end{multline}
	{where $\calF_\eps^\eta$ has been defined in \eqref{eqdefcalfepseta}.}
\end{lemma} 
\begin{proof}
		{To simplify the notation let us set  $\calF'':=\Gamma(L^1)\hbox{-}\limsup_{{\eps\to0}}\calF_\eps^\eta$.}
	It is not restrictive to assume that the right-hand side of \eqref{e:subadd} is finite, so that 
	$u\in {(GBV\cap L^1(A\cup B))^m}$.
	%and $v=1$ $\calL^n$-a.e.\ in $A\cup B$.
%{\color{red}{Assume that $u\in BV\cap L^\infty(A\cup B,\R^n)$.}}
Let $(u_\eps^A,v_\eps^A),\ (u_\eps^B,v_\eps^B)\in {W^{1,2}}(\Omega;\R^{m+1})$ be such that
	\begin{equation}\label{e:rec1}
	(u_\eps^J,v_\eps^J)\to (u,1) \text{ in } L^1(\Omega;\R^m)\times L^1(\Omega)\,,
	\end{equation}
	and
	\begin{equation}\label{e:rec2}
	\limsup_{\eps\to0} \calF_\eps^\eta(u_\eps^J,v_\eps^J;J)=\calF''(u,1;J),
	\end{equation}
	for $J\in\{A,B\}$. 
%	{\color{red}{We also assume that 
%	\begin{equation}\label{e:rec3}
%	\{u_\eps^J\}_\eps\text{ is bounded in } L^\infty(\Omega,\R^n) \text{ for }J\in\{A,B\}.  
%	\end{equation}}}
\medskip
	
{{\bf{Step 1. Estimate {\eqref{e:subadd}} is valid if $u\in BV\cap L^2(A\cup B;\R^m)$ and 
{\eqref{e:rec2} holds for two sequences converging to $u$}
%$u_\eps^J\to u$ 
in $L^2(\Omega;\R^m)$.}}} 
	For $\delta:=\dist(A^\prime,\partial A)>0$ and some $M\in\N$, we set for all 
	$i\in\{1,\ldots,M\}$
	\[
	A_i:=\left\{x\in \Omega:\,\dist(x,A^\prime)<\frac{\delta}{M}i\right\},
	\]
	and $A_0:=A^\prime$, so that $A_{i-1}\subset\subset A_i\subset A$.	
	Let $\varphi_i\in C_c^1(\Omega)$ be a cut-off function between $A_{i-1}$ and $A_i$,
	i.e., $\varphi_i|_{A_{i-1}}=1$, $\varphi_i|_{A_{i}^c}=0$, and 
	$\|\nabla \varphi_i\|_{L^\infty(\Omega)}\leq \frac{2M}{\delta}$. Then, we define
	\begin{equation}\label{e:uki}
	u_\eps^i:=\varphi_i\,u_\eps^A+(1-\varphi_i)u_\eps^B,
	\end{equation}
	and
	\begin{equation}\label{e:vki}
	v_\eps^i:=\begin{cases}
	\varphi_{i-1}\,v_\eps^A+(1-\varphi_{i-1})(v_\eps^A\wedge v_\eps^B), & \text{ on } A_{i-1},\cr
	v_\eps^A\wedge v_\eps^B, & \text{ on } A_i\setminus {A}_{i-1}, \cr
	\varphi_{i+1}(v_\eps^A\wedge v_\eps^B)+(1-\varphi_{i+1})\,v_\eps^B, & \text{ on } \Omega\setminus {A}_i.
	\end{cases}
	\end{equation}
	For $i\in\{2,\ldots,M-1\}$,  $(u_\eps^i,v_\eps^i)\in {W^{1,2}}(\Omega;\R^{m+1})$. Arguing exactly as in \cite[Lemma 6.2]{ContiFocardiIurlano2016}, for all $\eps>0$ we can find an index	$i_\eps\in\{2,\dots,M-1\}$ such that 
	\begin{multline*}
	\calF_\eps^\eta(u_\eps^{i_\eps},v_\eps^{i_\eps};A'\cup B)\leq \calF_\eps^\eta(u_\eps^A,v_\eps^A;A)+\calF_\eps^\eta(u_\eps^B,v_\eps^B;B)\\
	+\frac{c}{M}\,\Big(\calF_\eps^\eta(u_\eps^A,v_\eps^A;B\cap(A\setminus A'))
	+\calF_\eps^\eta(u_\eps^B,v_\eps^B;B\cap(A\setminus A')){{+\calL^n(B\cap (A\setminus A'))}}\Big)\\
	+\frac{c\,M}{\delta^2}\int_{B\cap(A\setminus A')}|u_\eps^A-u_\eps^B|^2\,\dx
	+\frac{c\,M\eps}{\delta^2}\int_{B\cap(A\setminus A')}|v_\eps^A-v_\eps^B|^2\,\dx.
	\end{multline*}
Passing first to the limit as $\eps\to0$ and then as $M\to\infty$ we obtain \eqref{e:subadd} {having assumed that $u_\eps^J\to u$ 
in $L^2(\Omega;\R^m)$, $J\in\{A,B\}$.}
\medskip

{{\bf{Step 2. Estimate \eqref{e:subadd} is valid if $u\in (GBV\cap L^1(A\cup B))^m$.}} 
We use De Giorgi's slicing/averaging techniques on the co-domain by employing the truncation 
functions introduced in \eqref{e:Tk}. The argument is analogous to that developed in Step 1 of 
Proposition \ref{periodicity}.

Note that if $u\in (GBV(A\cup B))^m$ then $\mathcal{T}_k(u)\in BV\cap L^\infty(A\cup B;\R^m)$. 
In addition, for all $k\in\N$ {and $J\in\{A,B\}$} it is easy to check that $\mathcal{T}_k(u_\eps^J)\in W^{1,2}(\Omega;\R^m)$,
that $\mathcal{T}_k(u_\eps^J)\to \mathcal{T}_k(u)$ {as $\eps\to0$} in $L^2({\Omega};\R^m)$, and that 
\begin{align}\label{e:Feps Tk}
 \calF_\eps^\eta(\mathcal{T}_k(u_\eps^J),v_\eps^J;J)&=\calF_\eps^\eta(u_\eps^J,v_\eps^J;\{|u_\eps^J|\leq a_k\})\notag\\
&+\calF_\eps^\eta(\mathcal{T}_k(u_\eps^J),v_\eps^J;\{a_k<|u_\eps^J|< a_{k+1}\})\notag\\&
+\calF_\eps^\eta(0,v_\eps^J;\{|u_\eps^J|\geq a_{k+1}\})\,.
\end{align}
We estimate the last but one term. The growth conditions on $\Psi$ (cf. \eqref{e:Psi gc}) and
$\|\nabla \mathcal{T}_k\|_{L^\infty(\R^m)}\leq 1$ yield for a constant $c>0$
\begin{align}\label{e:Feps Tk 2}
 \calF_\eps^\eta&(\mathcal{T}_k(u_\eps^J),v_\eps^J;\{a_k<|u_\eps^J|<a_{k+1}\})
 \leq c\int_{\{a_k<|u_\eps^J|< a_{k+1}\}}(\eta_\eps+f^2_\eps(v_\eps^J))\Psi(\nabla u_\eps^J)\dx\notag\\
 &+c\calL^n(\{a_k<|u_\eps^J|<a_{k+1}\})+{\calF_{\eps}(0,}v_\eps^J;\{a_k<|u_\eps^J|< a_{k+1}\})\,.
\end{align}
% and moreover
% \begin{align}\label{e:Feps Tk 3}
%  \calF_\eps(0,v_\eps^J;\{|u_\eps^J|\geq a_{k+1}\})={\calF_{\eps}(0,}v_\eps^J;\{|u_\eps^J|\geq a_{k+1}\})\,.
% \end{align}
%where $G_\eps $ is the functional introduced in \eqref{e:Geps}.
Collecting \eqref{e:Feps Tk} and \eqref{e:Feps Tk 2} we conclude that 
\begin{align*}
  \calF_\eps^\eta&(\mathcal{T}_k(u_\eps^J),v_\eps^J;J)\leq \calF_\eps^\eta(u_\eps^J,v_\eps^J;J)\\
  &+ c\int_{\{a_k<|u_\eps^J|< a_{k+1}\}}(\eta_\eps+f^2_\eps(v_\eps^J))\Psi(\nabla u_\eps^J)\dx
  {+c\calL^n(\{|u_\eps^J|>a_k\})}\,.
\end{align*}
Let now $M\in\N$, by summing up the latter inequality for both $A$ and $B$ and by averaging, there exists 
$k_{\eps,M}\in\{M+1,\ldots,2M\}$ such that 
\begin{align}\label{e:truncation Psik}
\calF_\eps^\eta&(\mathcal{T}_{k_{\eps,M}}(u_\eps^A),v_\eps^A;A)+\calF_\eps^\eta(\mathcal{T}_{k_{\eps,M}}(u_\eps^B),v_\eps^B;B)\notag\\
&\leq\frac 1M\sum_{k=M+1}^{2M}\Big(\calF_\eps^\eta(\mathcal{T}_k(u_\eps^A),v_\eps^A;A) + \calF_\eps^\eta(\mathcal{T}_k(u_\eps^B),v_\eps^B;B)\Big)\notag\\
&\leq\Big(1+\frac cM\Big)\Big(\calF_\eps^\eta(u_\eps^A,v_\eps^A;A)+\calF_\eps^\eta(u_\eps^B,v_\eps^B;B)\Big)\notag\\
&+c\calL^n({\{|u_\eps^A|\ge a_{M+1}\}})+c\calL^n({\{|u_\eps^B|\ge a_{M+1}\}}).
\end{align}
Up to a subsequence, we may take the index $k_{\eps,M}=k_M$, i.e.~to be independent of $\eps$. Therefore, passing to the limit
as $\eps\to0$, {the convergence $u_\eps^J\to u$ in measure for $J\in\{A,B\}$}, \eqref{e:rec1}, \eqref{e:rec2}, \eqref{e:truncation Psik} and  Step 1 yield
\begin{align}\label{e:truncation Psik 2}
\calF''(\mathcal{T}_{k_M}(u),1;A'\cup B)\leq&\calF''(\mathcal{T}_{k_M}(u),1;A)+\calF''(\mathcal{T}_{k_M}(u),1;B)\notag\\
\leq&\Big(1+\frac cM\Big)\Big(\calF''(u,1;A)+\calF''(u,1;B)\Big)
\\
&{+c\calL^n({\{|u|\ge a_{M+1}\}})}
\end{align}
Eventually, since $\mathcal{T}_{k_M}(u)\to u$ in $L^1(\Omega;\R^m)$ as $M\uparrow\infty$, 
by the lower semicontinuity of $\calF''$ for the $L^1(\Omega;\R^m)$ convergence 
we conclude \eqref{e:subadd}.
}
\end{proof}

We are now ready to prove the upper bound inequality.
\begin{theorem}\label{t:limsupndim}
	Let $\calF_{\eps}^\eta$ and $\calF_0$ be defined in  \eqref{eqdefcalfepseta} and \eqref{F0}, respectively.
	For every $(u,v)\in L^1(\Omega;\R^{m+1})$ it holds
	\begin{equation}\label{e:gammalisup estimate}
		{\Gamma{(L^1)\hbox{-}}\limsup_{{\eps\to0}}\calF_{\eps}^\eta(u,v)}\leq \Functlim(u,v).
	\end{equation}
\end{theorem}
\begin{proof}
	Given a subsequence $(\calF_{\eps_k}^\eta)$ of $(\calF_\eps^\eta)$, there exists a further subsequence, not relabeled, which $\overline{\Gamma}$-converges to some functional $\widehat{\calF}$, that is,
\begin{equation}\label{e:additivity}
\widehat \calF=(\calF')_-=(\calF'')_-,
\end{equation}
where $\calF'$ and $\calF''$ denote here the $\Gamma(L^1)$-lower and upper limits of $\calF_{\eps_k}^\eta$ and where the subscript $_-$ denotes the inner regular envelope of the relevant functional	
 (\cite[Definition~16.2 and Theorem~16.9]{DalMaso1993}). 
	
We remark that $\widehat{\calF}(u,v;\cdot)$ is the restriction of a Borel measure to open sets by \cite[Theorem~14.23]{DalMaso1993}. Indeed, $\widehat{\calF}(u,v;\cdot)$ is increasing and inner regular by definition; additivity follows {from \eqref{e:additivity},}
	once one checks that $(\calF')_-$ is superadditive and $(\calF'')_-$ is {subadditive.} 
	%{\color{red}{with respect to the topology $L^2(\Omega,\R^n)\times L^1(\Omega)$}} of the sequence $\calF_{\eps_k}$.
	The former condition is a direct consequence of the additivity of $\calF_{\eps}(u,v;\cdot)$ and \cite[{Proposition~16.12}]{DalMaso1993}. 
The latter follows from Lemma~\ref{l:wsub} along the lines of \cite[{Proposition~18.4}]{DalMaso1993}, 
using Lemma~\ref{l:wsub} instead of 
\cite[(18.6)]{DalMaso1993}.
    
We divide the proof of \eqref{e:gammalisup estimate} into several steps. First note that it is sufficient to prove 
it for $v=1$ $\calL^n$-a.e. on $\Omega$.

{\bf Step 1. Estimate on the diffuse part {for $u\in BV(\Omega;\R^m)$}.} 
We first prove a global rough estimate for $\calF''$ which actually turns out to be sharp for the diffuse part 
if $u\in BV(\Omega;\R^m)$. To this aim we set $H:L^1(\Omega;\R^m)\times \mathcal{A}(\Omega)\to [0,\infty]$ as
\begin{equation}\label{e:H}
 H(u;A):=\int_A h(\nabla u)\,\dx
\end{equation}
if $u\in W^{1,1}(\Omega;\R^m)$, and $\infty$ otherwise, where 
$h$ has been defined in \eqref{e:h}. We next prove the bound
\begin{equation}\label{linearest}
\calF''(u,1;A)\leq H(u;A)
%\int_A\ell\Big(\Psi^{\sfrac12}(\nabla u)\Big)^{qc}dx+\int_A\ell(\Psiinfty)^{\sfrac12}\Big(\frac{dD^su}{d|D^su|}\Big)d|D^su|,
\end{equation}
for $u\in W^{1,1}(\Omega;\R^m)$. Given this estimate for granted, 
on setting  $H^{*}:L^1(\Omega;\R^m)\times \mathcal{A}(\Omega)\to [0,\infty]$ 
\begin{equation}\label{e:Hqc}
 H^{*}(u;A):=\int_A h^{\qc}(\nabla u)\,\dx
\end{equation}
if $u\in W^{1,1}(\Omega;\R^m)$, and $\infty$ otherwise, the lower semicontinuity of 
$\calF''$ with respect to the $L^1(\Omega;\R^m)$ topology 
{and the relaxation result with respect to the {sequential} weak topology in $W^{1,1}(\Omega;\R^m)$ in}
\cite[Statement~III.7]{AcerbiFusco1984semicontinuity} {(or \cite[Theorem 9.1]{Dacorogna})} imply then that 
\begin{equation*}%\label{linearest qc}
\calF''(u,1;A)\leq H^{*}(u;A)\,.
\end{equation*}
In turn, from the estimate above, Theorem~\ref{theoambdm92} finally yields
\begin{equation}\label{e:lsc envelope H}
\mathcal H(u;A):=\mathrm{sc}^-(L^1)\hbox{-}H^{*}(u;A)
=\int_A h^{\qc}(\nabla u)\,\dx+\int_A h^\qcinfty\big(\dd D^su\big)\,,
\end{equation}
for every $u\in BV(\Omega;\R^m)$. Therefore, the bound 
\begin{equation}\label{e:stima limsup diffusa}
\calF''(u,1;A)\leq \mathcal H(u;A)%=\calF_0(u,1;A\setminus S_u)
\end{equation}
follows for every $u\in BV(\Omega;\R^m)$ and $A\in\mathcal{A}(\Omega)$.
% by taking into account the integral reresentation of $\calF''$ 

To prove \eqref{linearest}, assume first that $u$ is an affine function, say $u{(x)}=\xi x+b$, 
with $\xi\in \R^{m\times n}$, $b\in \R^m$.  
Then, the pair
\[
u_{k}:=u,\qquad v_{k}:= 1-\sqrt{2\ell\eps_k}{\Psi}^{\sfrac14}(\xi),\]
is such that $(u_k,v_k)\to (u,1)$ in $L^2(\Omega;\R^m)\times L^1(\Omega)$ and
recalling $\eta_{\eps_k}\to0$ 
\[\limsup_{k\to\infty}\calF_{\eps_k}^\eta(u_{k},v_{k};A)
\leq \calL^n(A)\ell{\Psi}^{\sfrac12}(\xi).
\]
Instead, if 
\[\bar u_{k}:=u,\qquad \bar v_{k}:= 1\]
we get 
\[\limsup_{k\to\infty}\calF_{\eps_k}^\eta(u_{k},v_{k};A)
{=} \calL^n(A)\Psi(\xi).\]
Therefore, we conclude \eqref{linearest} for every affine function $u$ in view of the last two estimates. 

{Assume now that $u\in C^0(\Omega;\R^m)$ is a piecewise affine function, say $u{(x)}=\sum_{i=1}^N(\xi_ix+b_i)\chi_{\Omega_i}(x)$, 
with $\xi_i\in \R^{m\times n}$, $b_i\in \R^m$, and $\Omega_i\in\calA(\Omega)$ 
{disjoint and}
with Lipschitz boundary, 
and such that $\calL^n(\Omega\setminus\cup_{i=1}^N\Omega_i)=0$. Then, set
\[
u_k:=u\;,\qquad v_k:=
\sum_{i=1}^N \varphi_iv_k^i
\]
where for each $i\in\{1,\ldots,N\}$ 
\[
v_{k}^i:= \begin{cases}
1-\sqrt{2\ell\eps_k}{\Psi}^{\sfrac14}(\xi_i), & 
\textrm{if ${\Psi}^{\sfrac12}(\xi_i)>\ell$},\cr
1, & \textrm{if ${\Psi}^{\sfrac12}(\xi_i)\leq \ell$},
\end{cases}
\]
and $\{\varphi_i\}_{1\leq i\leq N}$ is a partition of unity subordinated to the covering $\{\Omega_i^\delta\}_{1\leq i\leq N}$
of $\Omega$, $\Omega_i^\delta$ an open $\delta$-neighborhood of $\Omega_i$ for $\delta>0$, i.e.~
$\varphi_i\in C^\infty_c(\Omega_i^\delta)$, %$\rm{supp}\,\varphi_i\subset{\Omega_i^\delta}$, %in $\Omega$, 
$0\leq \varphi_i\leq 1$, $\varphi_i=1$ on {$\Omega_i^{-\delta}$}, 
$\sum_{i=1}^N\varphi_i=1$ in $\Omega$
(we write $\Omega_i^{-\delta}:=\{x: B_\delta(x)\subseteq \Omega_i\}$). %$\|\nabla \varphi_i\|_\infty\leq \delta^{-1}$, for some $\delta>0$.
Then, a straightforward computation shows that
\begin{align*}
\limsup_{k\to\infty}\,&\calF_{\eps_k}^\eta(u_{k},v_{k};A)\\&\leq
\sum_{i=1}^N \calL^n(\Omega_i^\delta\cap A)h(\xi_i)+
 c
%{\ell\max_{1\leq i\leq N}\Psi^{\sfrac12}(\xi_i)\,
\sum_{i=1}^N\calL^n(\Omega_i^\delta\setminus\Omega_i^{
{-\delta}}),
\end{align*}
{where $c$ depends on $\ell$, $\Psi$, and $\xi_1,\dots, \xi_N$.}
%for some constant $c>0$, 
Therefore we conclude \eqref{linearest} when $u$ is piecewise affine, namely as $\delta\to0$ in the latter inequality we have 
\[
 \calF''(u,1;A)\leq\sum_{i=1}^N \calL^n(\Omega_i\cap A)h(\xi_i)=H(u;A)\,.
\]
If $u\in W^{1,1}(\Omega;\R^m)$, we consider an extension of $u$ itself (still denoted by $u$ for convenience) to
$W^{1,1}_0(\Omega';\R^m)$, for some open and bounded
$\Omega'\supset\supset\Omega$ (recall that $\Omega$ is {assumed} to be Lipschitz regular). Then, we use a classical density result
\cite[Proposition~2.1 in Chapter X]{ekeland1999convex} to find 
$u_k\in W^{1,1}_0(\Omega';\R^m)$ piecewise affine such that 
$u_k\to u$ in $W^{1,1}(\Omega';\R^m)$.
The continuity of $H$ for the $W^{1,1}(\Omega;\R^m)$ convergence, and the lower semicontinuity of $\calF''$ 
for the $L^1(\Omega;\R^{m+1})$ convergence finally imply \eqref{linearest}.
}
\medskip 

{\bf Step 2. Inner regularity of $\calF''(u,1;\cdot)$ and {existence of the $\Gamma(L^1)$-limit in} $A\in\mathcal A(\Omega)$ {for} $u\in BV(\Omega;\R^m)$.} 
First we show that if $u\in BV(\Omega;\R^m)$ then
\begin{equation}\label{inner}
\calF''(u,1;\cdot)=(\calF'')_-(u,1;\cdot).
\end{equation}
Given an open set $A$ and $\delta>0$, we can find open sets $A'$, $A''$, and $C$, with 
$A'\subset\subset A''\subset\subset A$  
and $A\setminus A'\subset C$, such that $\mathcal H(u;C)\leq\delta$,
where $\mathcal H$ is defined in \eqref{e:lsc envelope H}.
%we have set
%\[\calF^\infty(u):=\int_\Omega h^{\qc}(\nabla u)dx+\int_{\Omega}h^\qcinfty(dD^su)\,,\]
%for $u\in BV(\Omega,\R^n)$.
Then, by Lemma~\ref{l:wsub} and \eqref{e:stima limsup diffusa} we get
\[\calF''(u,1;A)\leq \calF''(u,1;A'\cup C)\leq \calF''(u,1;A'')+\mathcal H(u;C)\leq \calF''(u,1;A'')+\delta.\]
 Hence, \eqref{inner} holds true and in turn by \eqref{e:additivity} we have
\[\widehat \calF(u,1;\cdot)\leq \calF'(u,1;\cdot)\leq \calF''(u,1;\cdot)=\widehat \calF(u,1;\cdot),\]
so that the $\Gamma$-limit of $\calF_{\eps_k}^{{\eta}}(u,1;\cdot)$ exists and coincides with $\widehat \calF(u,1;\cdot)$ 
for all $u\in BV(\Omega;\R^m)$.
\medskip 

{\bf Step 3. Integral representation of the $\Gamma(L^1)$-limit on $BV(\Omega;\R^m)\times\{1\}$.}
We now would like to represent $\widehat \calF$ as an integral functional through \cite[Theorem 3.12]{BouchitteFonsecaMascarenhas} and to estimate its diffuse and surface densities. In order to satisfy the coercivity hypothesis \cite[Eq. (2.3')]{BouchitteFonsecaMascarenhas}, we introduce an auxiliary functional
\[
\widehat{\calF}_\lambda(u,1):=\widehat{\calF}(u,1)+\lambda
{|Du|(\Omega)}%\int_{J_u}|[u]|\dd\calH^{n-1},
\]
for all $u\in BV(\Omega;\R^m)$, where $\lambda\in(0,1]$ is a small parameter. Indeed, 
{\eqref{eqfepsfeps1}, \eqref{e:one-dim estimate}, \eqref{e:hascal} and \eqref{e:stima limsup diffusa} yield}
\[
{\lambda|Du|(\Omega)-c\calL^n(\Omega)\leq
\widehat{\calF}_\lambda(u,1)\leq 
c(|Du|(\Omega)+\calL^n(\Omega))\,,}
\]
for all $u\in BV(\Omega;\R^m)$ and for some $c>0$. 
Note that $\widehat{\calF}_\lambda$ also satisfies the continuity hypothesis \cite[Eq. (2.4)]{BouchitteFonsecaMascarenhas}, since \[\calF_{\eps_k}^\eta(u(\cdot-z),v(\cdot-z);z+A)=\calF_{\eps_k}^\eta(u,v;A),\]  \[\calF_{\eps_k}^\eta(u+b,v;A)=\calF_{\eps_k}^\eta(u,v;A),\]
for all $(u,v)\in {W^{1,2}}(\Omega;\R^{m+1})$, $z,b\in\R^m$, $A\in\calA(\Omega)$, and analogous properties then hold for $\widehat{\calF}$.

The integral representation result 
\cite[Theorem 3.12]{BouchitteFonsecaMascarenhas} then applies to $\widehat{\calF}_{\lambda}+c\calL^n$ and gives, for $u\in BV(\Omega;\R^m)$ and $A\in\calA(\Omega)$, taking also into account the aforementioned translation invariance,
\[
\widehat{\calF}_{\lambda}(u,1;A)=
\int_Ah_\lambda(\nabla u)\dx+\int_{J_u\cap A} g_\lambda([u],\nu_u)\dd\calH^{n-1}+\int_{A}h_\lambda^\infty({\dd D^cu}),
\]
where 
\begin{multline}\label{e:flambdadff}
h_\lambda(\xi):=\limsup_{\delta\downarrow 0}\frac{1}{\delta^{n}}
\inf\left\{\widehat{\calF}_{\lambda}(w,1;\delta Q):\,w\in BV\big(\delta Q;\R^m\big),\right.
\\ w(x)=\xi x \text{ on }\partial(\delta Q)\Big\},
\end{multline}
for $\xi\in\R^{m\times n}$, $Q$ being a cube with side length $1$ centered in the origin;
\begin{multline}\label{e:flambda}
g_\lambda(z,\nu):=\limsup_{\delta\downarrow 0}\frac{1}{\delta^{n-1}}
\inf\left\{\widehat{\calF}_{\lambda}(w,1;\delta Q^{\nu}):\,w\in BV\big(\delta Q^{\nu};\R^m\big),\right.
\\ w=w_z \text{ on }\partial (\delta Q^{\nu})\Big\},
\end{multline}
for $z\in\R^m$, $\nu\in S^{n-1}$, $Q^{\nu}$ being a cube with side length $1$ and a face orthogonal to $\nu$ and $w_z:=z\chi_{\{x\cdot\nu>0\}}$;
\[
h^\infty_\lambda(\xi):=\limsup_{t\to\infty}\frac{h_\lambda(t\xi)}{t},
\]
for $\xi\in\R^{m\times n}$.
Let us estimate separately the three densities above. First, observe that by \eqref{e:stima limsup diffusa} we have
\begin{equation}\label{lambdalin}
h_\lambda(\xi)\leq \frac{1}{\delta^n}\widehat{\calF}_{\lambda}(\xi x,1;\delta Q)\leq h^{\qc}(\xi){+\lambda|\xi|},
\end{equation}
so that
\begin{equation}\label{recessione}
h_\lambda^\infty(\xi)\leq h^{\qcinfty}(\xi){+\lambda|\xi|},
\end{equation}
for all $\xi\in\R^{m\times n}$. 
%Moreover, taking $\bar u_k:=\xi x$ and $\bar v_k:=1$ as constant recovery sequence, we find
%\[h_\lambda(\xi)\leq\frac{1}{\delta^n}\widehat{\calF}_{\lambda}(\xi x,1,\delta Q)\leq \Psi(\xi),\]
%which combined with \eqref{lambdalin} gives
%\[h_\lambda(\xi)\leq h(\xi),\]
%and then by relaxation
%\begin{equation}\label{hqc} 
%h_\lambda(\xi)\leq h^{\qc}(\xi),
%\end{equation}
%for all $\xi\in\R^{m\times n}$.
%
We next show that
	\begin{equation}\label{e:upbvsalto}
	g_\lambda(z,\nu)\leq g(z,\nu)+\lambda|z|,
	\end{equation}
for $z\in \R^m$, $\nu\in S^{n-1}$.
From \eqref{e:flambda} we have
\begin{align}\label{sfc}
g_\lambda(z,\nu)&\leq \limsup_{\delta\downarrow 0}\frac{1}{\delta^{n-1}}\widehat{\calF}_{\lambda}(w_z,1;\delta\,Q^{\nu})
\notag\\
&=\limsup_{\delta\downarrow 0}\frac{1}{\delta^{n-1}}\widehat{\calF}(w_z,1;\delta\,Q^{\nu})+\lambda |z|.
\end{align}
{In turn, by definition of $\widehat \calF$ for every sequence  $(\tilde u_{{k}},\tilde v_{{k}})\to (w_z,1)$
in $L^1(\delta\,Q^{\nu};\R^{m+1})$  we have
\begin{equation}\label{recupp}
\widehat{\calF}(w_z,1;\delta\,Q^{\nu})\le
\limsup_{k\to\infty}
\calF_{\eps_k}^\eta(\tilde u_{k},\tilde v_{k};\delta\,Q^{\nu}).
\end{equation}
The proof of \eqref{e:upbvsalto} therefore reduces to the construction of a suitable sequence $(\tilde u_{{k}},\tilde v_{{k}})$, which depends implicitly on $\delta\in(0,1)$, $z$ and $\nu$.}
By {Proposition \ref{propuvjump}}
{applied with the sequences $\eps_k^*:=\eps_k/\delta$ and $\eta_k^*:=\eta_{\eps_k}$}, there are $(u_k^*,v_k^*)\to (w_z,1)$ in {$L^2(Q^\nu;\R^{m+1})$}, such that
\begin{equation}\label{estg}
\lim_{k\to\infty} 
\calF^{\infty}_{{\eps_k^*}}({u_k^*,v_k^*}; Q^{{\nu}})= g(z,{\nu})
\end{equation}
and
\begin{equation}\label{estg-2}
{\lim_{k\to\infty} 
\eta^*_{k}}\|\nabla u_k^*\|_{L^2(Q^{{\nu}})}^2=0.
\end{equation}
% and
% \begin{equation*}
% u_k^*(y)=(w_z\ast\varphi_{1})\Big(\frac{y}{{\eps_k^*}}\Big)\,,\hskip6mm
% v_k^*(y)=(\chi_{\{|{x\cdot\nu}|\ge2\}}\ast\varphi_{1})\Big(\frac{y}{{\eps^*_k}}\Big) \hskip6mm
% \text{ for } y\in\partial Q^{{\nu}}.
% \end{equation*}
{We define $(\tilde u_k,\tilde v_k)\in L^2(\delta \,Q^\nu;\R^{m+1})$ by}
	\[
	\tilde u_k(y):={u_k^*\left(\frac{y}{\delta}\right)}
	\,,\hskip1cm 
	\tilde v_k(y):={v_k^*\left(\frac{y}{\delta}\right)}.
	\]
{Obviously} $(\tilde u_k,\tilde v_k)\to(w_z,1)$ in $L^2({\delta\, Q^\nu};\R^{m+1})$. 
A change of variable and a straightforward computation 
	{using $\eps_k=\delta\eps_k^*$}
	yield 
\begin{equation}\label{eqdeltaresc}
 {\begin{split}\calF_{\eps_k}^\infty(\tilde u_k,\tilde v_k;\delta\,Q^\nu)
& =\delta^{n-1}
 \calF_{\eps_k^*}^\infty( u_k^*, v_k^*;\,Q^\nu),
 \\
 \|\nabla \tilde u_k\|_{L^2(\delta\, Q^\nu)}^2
 &=\delta^{n-2}
 \|\nabla u_k^*\|_{L^2(Q^\nu)}^2.
 \end{split}}
\end{equation}
Fixed $\rho>0$, by \eqref{eqpsipsiinf} we have
\[\Psi(\xi)\leq (1+\rho)\Psiinfty(\xi),\]
for $|\xi|$ large, and then
\[\Psi(\xi)\leq (1+\rho)\Psiinfty(\xi)+C(\rho),\]
for some $C(\rho)>0$ and all $\xi\in \R^{m\times n}$.
Then, {with \eqref{eqdeltaresc}} 
\begin{equation*}
{\begin{split} \calF_{\eps_k}(\tilde u_k,\tilde v_k;\delta\,Q^\nu)
& \le (1+\rho) \calF_{\eps_k}^\infty(\tilde u_k,\tilde v_k;\delta\,Q^\nu)
 +C(\rho)\calL^n(\delta\, Q^\nu)\\
& = (1+\rho) \delta^{n-1}\calF_{\eps_k^*}^\infty(u_k^*,v_k^*; Q^\nu)
 +C(\rho)\delta^n.
 \end{split}}
\end{equation*}
{Similarly, from
the growth conditions in \eqref{e:Psi gc}
and \eqref{eqdeltaresc}, 
\begin{equation*}
 \eta_{\eps_k}\int_{\delta \, Q^\nu}
 \psi(\nabla\tilde u_k) \dd x
 \le c \eta_{\eps_k} (\|\nabla \tilde u_k\|_{L^2(\delta\, Q^\nu)}^2+\delta^n)
 =
c  \eta_{\eps_k}\delta^{n-2} \|\nabla u_k^*\|_{L^2(Q^\nu)}^2
 +c\eta_{\eps_k}\delta^n.
\end{equation*}
}{Summing these two estimates,
\begin{equation*}
 \calF^\eta_{\eps_k}(\tilde u_k,\tilde v_k;\delta\,Q^\nu)
 \le (1+\rho) \delta^{n-1}\calF_{\eps_k^*}^\infty(u_k^*,v_k^*; Q^\nu)
 +C(\rho)\delta^n+
c  \eta_{\eps_k}\delta^{n-2} \|\nabla u_k^*\|_{L^2(Q^\nu)}^2
 +c\eta_{\eps_k}\delta^n,
\end{equation*}
and taking  the limit $k\to\infty$, by \eqref{recupp}, \eqref{estg} and \eqref{estg-2},}
\begin{equation}
{\begin{split} 
\widehat\calF(w_z,1;\delta \, Q^\nu)\le
\limsup_{k\to\infty} \calF^\eta_{\eps_k}(\tilde u_k,\tilde v_k;\delta\,Q^\nu)
& \le (1+\rho) \delta^{n-1}g(z,\nu)
 +C(\rho)\delta^n.
 \end{split}}
\end{equation}
{We divide by $\delta^{n-1}$ and take the limit $\delta\to0$. Comparing with \eqref{sfc},
\begin{equation}
 g_\lambda(z,\nu)\le (1+\rho) g(z,\nu)+\lambda|z|,
\end{equation}
and since $\rho$ was arbitrary
\eqref{e:upbvsalto} follows.}
% \begin{equation}
%  g_\lambda(z,\nu)\le g(z,\nu)+\lambda|z|.
% \end{equation}}
% As $\rho\to0$ we get \eqref{recupp}.

In conclusion, as $\lambda\to0$, estimates \eqref{lambdalin}, \eqref{recessione} and \eqref{e:upbvsalto} imply that for all 
$u\in BV(\Omega;\R^m)$ 
\[
\widehat \calF(u,1)\leq \calF_0(u,1).
\]
This, together with the lower bound {Theorem \ref{t:lbcompBV}} allows to identify uniquely the 
$\Gamma$-limit of the subsequence $\calF_{\eps_k}^{{\eta}}$. Finally, Urysohn's property (\cite[Proposition~8.3]{DalMaso1993}) 
extends the result to the whole family $\calF_{\eps}^\eta$.
\medskip

{{\bf Step 4. Representation of the $\Gamma(L^1)$-limit on $(GBV(\Omega))^m\times\{1\}$.} 
To extend the validity of \eqref{e:gammalisup estimate} to $u\in (GBV(\Omega))^m$ we argue by truncation.
Indeed, if $k\in\N$ and $\mathcal{T}_k$ is the truncation operator defined in \eqref{e:Tk}, then by Steps 1-3 we infer that 
\[
 \calF''(\mathcal{T}_k(u),1)\leq \calF_0(\mathcal{T}_k(u),1)\,.
\]
}
{The conclusion then follows by the $L^1${{-}}lower semicontinuity of $\calF''$ and by Proposition \ref{contF0}.} 
\end{proof}

We are ready to prove Theorem \ref{t:finale}.
\begin{proof}[Proof of Theorem \ref{t:finale}]
The lower bound 
has been proven in Theorem \ref{t:lbcomp}.
The upper bound follows by Theorem \ref{t:limsupndim} with $\eta_\eps=0$.
\end{proof}

{
\clearpage
\section{Compactness and convergence of minimizers}\label{compandconv}
Next theorem establishes the compactness of sequences {equibounded in energy and in $L^1$.} 
\begin{theorem}\label{t:comp}
	Let $\calF_{\eps}$ be defined in \eqref{functeps}. If $(u_\eps,v_\eps)\in {W^{1,2}}(\Omega;\R^{m+1})$ is such that
	$$\sup_\eps \left(\calF_{\eps}(u_\eps,v_\eps)+\|u_\eps\|_{L^1(\Omega)}\right)<\infty,$$
	then there exists a subsequence $(u_j,v_j)$ of $(u_\eps,v_\eps)$ and a function $u\in (GBV\cap L^1(\Omega))^m$ 
	such that $u_j\to u$ $\calL^n$-a.e. and $v_j\to 1$ in $L^1(\Omega)$.
\end{theorem}
\begin{proof}
This follows arguing componentwise, that is, estimating $\calF_{\eps}$ with its one-dimensional counterpart 
evaluated in a component, and applying the one-dimensional compactness result obtained in \cite[Theorem 3.3]{ContiFocardiIurlano2016}
as done in subsection~\ref{s:scalar case} (see also the argument in Remark~\ref{r:domain direct}).
\end{proof}
% \begin{proof}
% 	Let us introduce the one-dimensional counterpart of $\calF_{\eps}$
% 	\[\calF_{\eps}^1(w,v):=\int_\Omega \left( f_\eps^2(v) |\nabla w|^2+ \frac{(1-v)^2}{4\eps} + \eps |\nabla v|^2 \right) \dx,
% 	\]
% 	for $(w,v)\in {W^{1,2}}(\Omega,\R^2)$. Then, we observe that
% 	\[
% 	\calF_{\eps}^1(u^i_\eps,v)+\|u^i_\eps\|_{L^1(\Omega)}\leq c\calF_{\eps}(u_\eps,v_\eps)+c\|u^i_\eps\|_{L^1(\Omega)},
% 	\]
% 	for $i=1,\dots,m$, where $u_\eps:=(u_\eps^1,\dots,u_\eps^m)$ and $c>0$. Using \cite[Theorem 3.3]{ContiFocardiIurlano2016} $m$ times and diagonalizing, we can extract a subsequence $(u_j,v_j)$ of $(u_\eps,v_\eps)$ and a function $u\in (GBV\cap L^1(\Omega))^m$ 
% 	such that $u_j\to u$ $\calL^n$-a.e. and $v_j\to 1$ in $L^1(\Omega)$.
% \end{proof}

Convergence of minimizers and of minimum values follow now 
in a standard way by {Theorems \ref{t:finale}} and \ref{t:comp}. 
Let $\eta_\eps>0$ be as in \eqref{eqassetaeps}, i.e.~such that 
$\sfrac{\eta_\eps}\eps\to0$ as $\eps\to0$, consider the corresponding family $\calF_\eps^\eta$ 
defined in \eqref{eqdefcalfepseta} and let $w\in L^q(\Omega;\R^m)$, with $q>1$. 
Let now $\calG_\eps,\,\calG_0:{L^q(\Omega;\R^{m})\times L^1(\Omega)}\to[0,\infty]$ be defined as 
\[
\calG_\eps(u,v):=
\begin{cases}
 \displaystyle
\calF_\eps^\eta(u,v)+\int_\Omega |u-w|^q\dx\,,
&\text{ if } (u,v)\in W^{1,2}(\Omega;\R^{m}\times[0,1]),\\
\infty, & \text{ otherwise}
\end{cases}
\]
and 
\[
\calG_0(u{,v}):=\calF_0(u,v)+\int_\Omega |u-w|^q\dx\,, 
\]
where $\calF_\eps$ and  $\calF_0$ have been defined in \eqref{functeps} and \eqref{F0}, respectively.

The assumption on the asymptotic ratio $\sfrac{\eta_\eps}\eps\to0$ as $\eps\to0$ is needed to avoid that the term 
$\eta_\eps\Psi(\nabla u)$ competes with the term $(1-v)^2/\eps$, overall influencing the limit behaviour. 
Indeed, if $\eta_\eps\sim \eps$, we expect to gain a control on $|[u]|$, so loosing the limit cohesive effect (compare with \cite{FocardiIurlano}).}

Instead, the {addition} of the term $\eta_\eps \Psi(\nabla w)$ is instrumental to guarantee the existence of a minimizer 
for $\calG_\eps$, provided that $\Psi$ is quasiconvex. In general, the coercivity {of $\calG_\eps$} only ensures existence 
of minimizing sequences $(u_\eps^j)_j$ converging weakly in $W^{1,{2}}(\Omega;\R^m)$ to some $\bar u_\eps$ {minimizing the 
relaxation of $\calG_\eps$}. 
Since existence at fixed $\eps$ does not interact with the {$\Gamma$-}convergence, we state our result for asymptotically 
minimizing sequences. 
\begin{corollary}\label{c:min}
	Let $(u_\eps,v_\eps)\in {W^{1,2}}(\Omega;\R^{m+1})$ be such that
	\[
	\limsup_{\eps\to0}\bigl(\calG_\eps(u_\eps,v_\eps)-m_\eps\bigr)=0,
	\]
	where $m_\eps:=\inf_{(u,v)\in {W^{1,2}}(\Omega;\R^{m+1})}\calG_\eps(u,v)$.
	Then $v_\eps\to 1$ in $L^1(\Omega)$ and a subsequence of $u_\eps$ converges in $L^q(\Omega;\R^m)$
	to a solution of
	\[
	\min_{u\in (GBV(\Omega))^m}{\calG_0(u,1)}.
	\]
	Moreover, $m_\eps$ tends to the minimum value of {$\calG_0$}.
\end{corollary} 
\begin{proof}

The proof of the corollary will be divided in three steps.

{\bf{Step 1. $\Gamma$-limit of $\cal F_\eps^\eta$ in $L^q\times L^1$.}}
We check that passing from the $L^1\times L^1$ to the $L^q\times L^1$ topology, the expression of the $\Gamma$-limit 
of $\calF_\eps^\eta$ remains the same
$$\Gamma(L^q\times L^1)\text{-}\lim_{\eps\to0}\calF_\eps^\eta(u,v)=\calF_0(u,v).$$
The lower bound is an immediate consequence of that in $L^1\times L^1$ 
(Theorem~\ref{t:lbcomp}, being the $L^q$ convergence stronger than the $L^1$ convergence).

As for the upper bound, we argue by truncation. First take a subsequence of $\calF_\eps^\eta$ (not relabelled for convenience) and fix $u\in BV\cap L^\infty(\Omega;\R^m)$ with $\calF_0(u,1)<\infty$. Then Theorem \ref{t:limsupndim} yields the existence of a sequence $(u_\eps,v_\eps)\in W^{1,2}(\Omega;\R^{m+1})$,
such that $(u_\eps,v_\eps)\to (u,1)$ in $L^1{(\Omega;\R^{m+1})}$ and
$$\limsup_{\eps\to0}\calF_{\eps}^\eta(u_\eps,v_\eps)\leq \calF_0(u,1).$$
Fix $M\in\mathbb N$ large enough such that $a_M>\|u\|_\infty$ (see \eqref{e:Tk} for the definition of $a_M$) and, for every $\eps>0$, choose $k_{\eps,M}\in \{{M+1},\dots,2M\}$ such that
$$\int_{\{a_{k_{\eps,M}}<|u_\eps|<a_{k_{\eps,M}+1}\}}(\eta_\eps+f_\eps^2(v_\eps))\Psi(\nabla u_\eps)\dx\leq \frac 1M\int_{\Omega}(\eta_\eps+f_\eps^2(v_\eps))\Psi(\nabla u_\eps)\dx.$$
This implies
$$\calF_\eps^\eta(\mathcal{T}_{k_{\eps,M}}(u_\eps),v_\eps)\leq (1+\frac CM)\calF_\eps^\eta(u_\eps,v_\eps)%+\frac CM
{+C\calL^n(\{a_{M+1}<|u_\eps|\})}\,,$$
with $\mathcal{T}_{k_{\eps,M}}(u_\eps)$ uniformly bounded in $L^\infty$, $\mathcal{T}_{k_{\eps,M}}$ being defined in \eqref{e:Tk}. This argument has been used several times throughout the paper, see for example Theorem \ref{t:lbcomp}. Passing to a further subsequence in $\eps$, we can take $k_{\eps,M}=k_M$ independent of $\eps$. Since $(\mathcal{T}_{k_M}(u_\eps))_\eps$ is uniformly bounded in $L^\infty$ and $M$ is large, we get $\mathcal{T}_{k_M}(u_\eps)\to \mathcal{T}_{k_M}(u)=u$ in $L^q{(\Omega;\R^m)}$ 
and in particular $\calL^n(\{a_{M+1}<|u_\eps|\})\to0$ {as $\eps\to0$}, hence
$$
\limsup_{\eps\to0}\calF_{\eps}^\eta(\mathcal{T}_{k_M}(u_\eps),v_\eps)\leq \Big(1+\frac CM\Big)\calF_0(u,1)%+\frac CM
%{+C\calL^n(\{a_{M+1}<|u_\eps|\})}
\,.
$$
Diagonalizing with respect to $M$ and recalling the lower estimate, we conclude that every subsequence of $\{\calF_\eps^\eta\}_\eps$ has a subsequence that $\Gamma(L^q\times L^1)$-converges to $\calF_0$ in $L^\infty(\Omega;\R^m)\times L^1(\Omega)$. Finally Urysohn's lemma gives the convergence of the entire sequence in the same space.

Let us consider now the general case $u\in (GBV\cap L^q(\Omega))^m$. Then $\mathcal{T}_k(u)\in ({BV\cap L^\infty(\Omega)})^m$, 
with $\mathcal{T}_k$ again defined by \eqref{e:Tk}, and
$$\Gamma(L^q\times L^1)\text{-}\limsup_{\eps\to0} \calF_\eps^\eta(\mathcal{T}_k(u),1)\leq \calF_0(\mathcal{T}_k(u),1),$$
by the first part of the proof. As $k\to\infty$ we have $\mathcal{T}_k(u)\to u$ in $L^q{(\Omega;\R^m)}$ and we conclude by the lower semicontinuity of the $\Gamma$-limsup and the continuity of $\calF_0$ (see Proposition \ref{contF0}).

{\bf Step 2. $\Gamma$-limit of $\calG_\eps$ in $L^1\times L^1$.} 
We check now that
$$\Gamma(L^1\times L^1)\text{-}\lim_{\eps\to0}\calG_\eps(u,v)=\calG_0(u,v).$$
The lower bound simply follows by Theorem \ref{t:lbcomp} using $\eta_\eps\ge0$ and the lower semicontinuity of $\int_\Omega|w-u|^q\dx$ with respect to the convergence in $L^1$. {In particular, if $\Gamma(L^1\times L^1)\text{-}\liminf_{\eps\to0}\calG_\eps(u,v)<\infty$, 
then $u\in (GBV(\Omega)\cap L^q)^m$ and $v=1$ $\calL^n$-a.e. on $\Omega$.}

As for the upper bound, from Step 1 we know that for all $u\in (GBV(\Omega)\cap L^q)^m$ there exists a recovery sequence for $\calF_\eps^\eta$ in $L^q\times L^1$. This is in particular a recovery sequence for $\calG_\eps$ in $L^1\times L^1$, which gives the conclusion.

{{\bf Step 3. Convergence of minimizers.}
Let now $(u_\eps,v_\eps)\in W^{1,2}\cap  L^q(\Omega;\R^{m+1})$ be a minimizing sequence for $\calG_\eps$. Being 
$$\sup_{\eps>0}{(}\calF_\eps(u_\eps,v_\eps)+\|u_\eps\|_{L^q(\Omega)}{)}<\infty,$$ Theorem \ref{t:comp} gives the existence of a function $u\in (GBV(\Omega)\cap L^q)^m$ and of a subsequence, not relabelled, such that $u_\eps\to u$ $\calL^n$-a.e. on {$\Omega$} 
and $v_\eps\to1$ in $L^1{(\Omega;\R^m)}$. In addition, by H\"older inequality
\begin{align*}
\int_{\{|u_\eps-u|>1\}}|u_\eps-u|\dx&\leq \|u_\eps-u\|_{L^q(\Omega)}{\big(\calL^n(}\{|u_\eps-u|>1\})\big)^{1-\sfrac 1q}\\
&\leq c{\big(\calL^n\big(}\{|u_\eps-u|>1\}\big)^{1-\sfrac 1q}\,,
\end{align*}
and the right-hand side tends to $0$ since $u_\eps\to u$ in measure {on $\Omega$}. Also, $(u_\eps-u)\chi_{\{|u_\eps-u|\leq1\}}\to0$ 
in $L^1{(\Omega;\R^m)}$ by dominated convergence, hence we conclude that $u_\eps\to u$ in $L^1{(\Omega;\R^m)}$.

By Step 2 and a general property of $\Gamma$-convergence \cite[Corollary 7.20]{DalMaso1993}, we conclude that $(u,1)$ is a minimizer of $\calG_0$ and that $\calG_\eps(u_\eps,v_\eps)\to\calG_0(u,1)$. Finally, we check that in fact $u_\eps\to u$ in $L^q{(\Omega;\R^m)}$. From the previous steps we have
\begin{eqnarray*}
&\calG_\eps(u_\eps,v_\eps)\to\calG_0(u,1),\\
&\displaystyle\int_\Omega|u-w|^q\dx\leq\liminf_{\eps\to0}\int_\Omega|u_\eps-w|^q\dx,\\
&\displaystyle \calF_0(u,1)\leq \liminf_{\eps\to0}\calF_\eps^\eta(u_\eps,v_\eps),
\end{eqnarray*}
so that
$$\int_\Omega|u_\eps-w|^q\dx\to\int_\Omega|u-w|^q\dx.$$
Together with the pointwise convergence, this implies $u_\eps\to u$ in $L^q{(\Omega;\R^m)}$ by generalized dominated convergence.}
\end{proof}	
\clearpage

\section*{Acknowledgements}

{This work was partially supported
by the Deutsche Forschungsgemeinschaft through project 211504053/SFB1060 and
project 441211072/SPP2256. M.F. is a member of GNAMPA of INdAM.}

\section*{Data availability statement}

Data sharing not applicable to this article as no datasets were generated or analysed during the current study

\bibliographystyle{alpha-noname}
\bibliography{cfi}

\end{document}